\documentclass{amsart}
\usepackage{amsmath, amsfonts, latexsym, array, amssymb, amscd }
\usepackage[all]{xypic}
\usepackage[all]{xy}
\usepackage{graphicx}
\usepackage{color}
\usepackage{enumitem}
\usepackage{epstopdf}
\usepackage{hyperref}

\numberwithin{equation}{section}
\numberwithin{figure}{section}

\newtheorem{thm}{Theorem}[section]

\newtheorem{cor}[thm]{Corollary}
\newtheorem{lem}[thm]{Lemma}

\newtheorem{prop}[thm]{Proposition}

\newtheorem{defn}[thm]{Definition}

\newtheorem{thmx}{Theorem}

\newcommand{\comment}[1]{}

\usepackage{ulem}

\newcommand{\dS}{d_{\mathrm{S}}}
\newcommand{\dK}{d_{\mathrm{K}}}

\newcommand{\Sing}{\mathrm{Sing}}
\newcommand{\Reg}{\mathrm{Reg}}
\newcommand{\htop}{h_{\mathrm{top}}}
\newcommand{\ZZ}{\mathbb{Z}}
\newcommand{\RR}{\mathbb{R}}

\newcommand{\tl}{\tilde{\lambda}}

\newcommand{\ph}{\varphi}
\newcommand{\PPP}{\mathcal{P}}
\newcommand{\SSS}{\mathcal{S}}
\newcommand{\GGG}{\mathcal{G}}
\newcommand{\RRR}{\mathcal{R}}
\newcommand{\CCC}{\mathcal{C}}
\newcommand{\BBB}{\mathcal{B}}
\newcommand{\AAA}{\mathcal{A}}

\newcommand{\III}{\mathcal{I}}
\newcommand{\JJJ}{\mathcal{J}}
\newcommand{\MMM}{\mathcal{M}}
\newcommand{\KKK}{\mathcal{K}}
\newcommand{\DDD}{\mathcal{D}}
\newcommand{\XXX}{\mathcal{X}}
\newcommand{\sing}{\mathrm{Sing}}
\newcommand{\eps}{\epsilon}
\newcommand{\vg}{\varphi^u}
\newcommand{\mureg}{\mu^{\mathrm{Reg}}}
\newcommand{\wiggle}{\sigma}

\DeclareMathOperator{\inj}{inj}

\newcommand{\FFF}{\mathcal{F}}
\newcommand{\UUU}{\mathcal{U}}
\newcommand{\Per}{\mathrm{Per}}

\newcommand{\overl}{\overline\lambda}
\newcommand{\underl}{\underline\lambda}

\DeclareMathOperator{\supp}{supp}

\DeclareMathOperator{\tr}{tr}

\begin{document}
\title[Equilibrium states for rank 1 geodesic flows] {Unique equilibrium states for geodesic flows in nonpositive curvature}
\author{K.~Burns, V.~Climenhaga,  T.~Fisher, and D.~J.~Thompson}
\address{K.~Burns, Department of Mathematics, Northwestern University, Evanston, IL 60208, \emph{E-mail address:} \tt{burns@math.northwestern.edu}}
\address{ V.~Climenhaga,  Department of Mathematics, University of Houston, Houston, TX 77204, \emph{E-mail address:} \tt{climenha@math.uh.edu}}
\address{T.~Fisher, Department of Mathematics, Brigham Young University, Provo, UT 84602, \emph{E-mail address:} \tt{tfisher@mathematics.byu.edu}}
\address{D.~J.~Thompson, Department of Mathematics, The Ohio State University, Columbus, OH 43210, \emph{E-mail address:} \tt{thompson@math.osu.edu}}

\thanks{V.C.\ is supported by NSF grants DMS-1362838 and 
DMS-1554794.  T.F.\ is supported by Simons Foundation grant \# 239708. D.T.\ is supported by NSF grant DMS-$1461163$.}

\date{\today}
\commby{}

\begin{abstract}
We study  geodesic flows over compact rank 1 manifolds and prove that sufficiently regular potential functions have unique equilibrium states if the singular set does not carry full pressure.  In dimension 2, this proves uniqueness for scalar multiples of the geometric potential on the interval $(-\infty,1)$, which is optimal.  In higher dimensions, we obtain the same result on a neighborhood of 0, and give examples where uniqueness holds on all of $\RR$.  For general potential functions $\varphi$, we prove that the pressure gap holds whenever $\varphi$ is locally constant on a neighborhood of the singular set, which allows us to give examples for which uniqueness holds on a $C^0$-open and dense set of H\"older potentials.
\end{abstract}

\maketitle
\setcounter{tocdepth}{1}

\section{Introduction}\label{s.intro}

We study uniqueness of equilibrium states for the geodesic flow over a compact rank 1 manifold with nonpositive sectional curvature.  In negative curvature, geodesic flow is Anosov and every H\"older potential has a unique equilibrium state.  
In nonpositive curvature, the flow is nonuniformly hyperbolic and may have phase transitions; the challenge is to exhibit a class of potential functions where uniqueness holds.

The first major result in this direction was Knieper's proof of uniqueness of the measure of maximal entropy using Patterson--Sullivan measures \cite{knieper98}. 
We use different techniques, inspired by Bowen's criteria to show uniqueness of equilibrium states \cite{Bow75}. This approach has been generalized by the second and fourth named authors, giving uniqueness of equilibrium states under non-uniform versions of Bowen's hypotheses  \cite{CT4}.  
We give conditions under which these techniques can be applied to geodesic flows on rank 1 manifolds, and demonstrate that these conditions are satisfied for a large class of potential functions.

Throughout the paper, $M= (M^n,g)$ will be a closed connected $C^\infty$ Riemannian manifold with nonpositive sectional curvature and dimension $n$, and  $\mathcal{F}=(f_t)_{t\in\mathbb{R}}$ will denote the geodesic flow on the unit tangent bundle $T^1M$.  There are two continuous invariant subbundles $E^s$ and $E^u$ of $TT^1M$,  each of dimension $n-1$, which are orthogonal to the flow direction $E^c$ in the natural Sasaki metric; these can be interpreted as normal vector fields to the stable and unstable horospheres.  If the curvature is strictly negative, $\FFF$ is Anosov and $TT^1M = E^s \oplus E^c \oplus E^u$ is the Anosov splitting.

In nonpositive curvature, $E_v^s$ and $E_v^u$ may intersect nontrivially. The \emph{rank} of a vector $v\in T^1M$ is $1 + \dim (E^s_v \cap E^u_v)$, which is the dimension of the space of parallel Jacobi vector fields for the geodesic through $v$. 
The \emph{rank of $M$} is the minimum rank over all vectors in $T^1M$.  We assume that $M$ has rank $1$.  
For a rank $1$ manifold, the \emph{regular set}, denoted $\Reg$, is the set of $v\in T^1M$ with rank $1$. The set $\Reg$ is dense since it is open and invariant, and the geodesic flow is topologically transitive. The \emph{singular set}, denoted $\mathrm{Sing}$, is the set of vectors whose rank is larger than 1. If $\Sing$ is empty, then the geodesic flow is Anosov; this includes the negative curvature case. The case when $\Sing$ is nonempty is a prime example of nonuniform hyperbolicity.

We study uniqueness of equilibrium states for the geodesic flow $\mathcal F$. An \emph{equilib\-ri\-um state} for a continuous function $\varphi\colon T^1M \to \RR$, which we call a \emph{potential function}, is an invariant Borel probability measure that maximizes the free energy $h_{\mu}(\FFF) + \int \varphi \,d \mu$, where $h_{\mu}(\FFF)$ is the measure-theoretic entropy with respect to the geodesic flow.  This maximum is denoted by $P(\varphi)$ and is called the \emph{topological pressure} of $\varphi$ with respect to the geodesic flow $\FFF$. In the case when $\varphi=0$, the topological pressure is the topological entropy $\htop(\FFF)$. Since $\FFF$ is entropy expansive, 
equilibrium states exist for any continuous function, but uniqueness is a subtle question beyond the uniformly hyperbolic setting.

The \emph{geometric potential} 
$\vg(v)=-\lim_{t\to 0} \frac{1}{t}\log \det(df_t|_{E^u_v})$
and its scalar multiples $q\vg$ ($q\in \RR$) are of particular interest. When $q=1$, the Liouville measure $\mu_L$ is an equilibrium state for $\vg$; in the Anosov setting, it is the unique equilibrium state. When $q=0$, equilibrium states for $q \vg$ are measures of maximal entropy; uniqueness of the measure of maximal entropy in rank 1  was proved by Knieper \cite{knieper98}. In the case of surfaces without focal points, this result has been established recently using different methods by Gelfert and Ruggiero \cite{GR}.
 When $M$ is a rank 1 surface, the family $q \vg$ contains geometric information about the spectrum of the maximum Lyapunov exponent \cite{BG}. 

We now state and discuss our main theorems. Let $P(\Sing, \varphi)$ denote the topological pressure of the potential $\varphi|_\Sing$ with respect to the geodesic flow restricted to the singular set (setting $P(\sing, \varphi)=-\infty$ if $\Sing=\emptyset$).

\begin{thmx}\label{t.multiples}
Let $\FFF$ be the geodesic flow over a closed rank 1 manifold $M$ and let $\varphi\colon T^1M\to \RR$ be $\varphi=q\vg$ or be H\"older continuous. If $P(\sing, \varphi)<P(\varphi)$, then $\varphi$ has a unique equilibrium state $\mu$. This equilibrium state 
is hyperbolic, fully supported, and is the weak$^\ast$ limit of weighted regular closed geodesics; see \S\ref{s.periodicpressure}.
\end{thmx}
The hypothesis $P(\sing, \varphi)<P(\varphi)$ is a sharp condition for having a unique equilibrium state which is fully supported; if $P(\sing, \varphi)=P(\varphi)$, then $\varphi$ has at least one equilibrium state supported on $\Sing$.  
We remark that an ergodic $\mu$ is hyperbolic if and only if $\mu(\Reg)=1$ (see Corollary \ref{c.hyperbolic}), and that the proof of equidistribution for weighted regular closed geodesics in \S\ref{s.geodgeneral} also establishes counting estimates, which are of independent interest.

The proof of Theorem \ref{t.multiples} uses general machinery developed by the second and fourth authors~\cite{CT4}, which was inspired by Bowen's work on uniqueness using the \emph{expansivity} and \emph{specification} properties \cite{Bow75} and its extension to flows by Franco \cite{Franco}. The results in \cite{CT4} use weaker versions of these properties which are formulated at the level of finite-length orbit segments; see \S\ref{s.abstract}. This allows us to avoid issues with asymptotic behavior of orbits that would be hard to control in our setting. The idea is that every orbit segment can be decomposed into `good' and `bad' parts, where the `good' parts satisfy Bowen's conditions, and the `bad' parts carry smaller topological pressure than the whole system.

Bowen's result applies to potentials satisfying a regularity condition that we call the \emph{Bowen property}; our result uses the non-uniform Bowen property from \cite{CT4}, which holds here for all H\"older potentials. Verifying this condition for the potentials $q \vg$ is a significant point in our argument; see \S\ref{sec:Bowen}. It is not  currently known if horospheres are $C^{2+\alpha}$ for rank 1 manifolds in dimension greater than $2$, which is necessary for H\"older continuity of the  unstable distribution. Even in dimension $2$, where horocycles are known to be $C^{2+\frac12}$ by \cite{GW99}, H\"older continuity of the unstable distribution, and thus $\vg$, is an open question.

For the class of potentials under consideration, Theorem \ref{t.multiples} reduces the problem of uniqueness of equilibrium states to checking if the pressure gap $P(\Sing,\ph) < P(\ph)$ holds. The following result establishes this gap,  and hence uniqueness of equilibrium states,  for a large class of H\"older continuous potentials.

\begin{thmx}\label{thm:pressure-gap}
With $\FFF$ and $M$ as above, let $\ph\colon T^1 M \to \RR$ be a continuous function that is locally constant on a neighborhood of $\Sing$.
Then $P(\Sing,\ph) < P(\ph)$.
\end{thmx}
The case $\ph=0$ recovers Knieper's result that the singular set has smaller entropy than the whole system.  In Knieper's work \cite{knieper98}, this was obtained a posteriori as a consequence of the uniqueness result. Knieper uses the Patterson-Sullivan construction to build a measure of maximal entropy $\mu$, and shows it is unique by exploiting properties of Busemann densities and other asymptotic geometry arguments.  It is built into the construction that $\mu(\Reg)=1$, and it thus follows that $\Sing$ has smaller entropy.  The argument presented here proceeds quite differently. It does not rely on the uniqueness of the measure of maximal entropy, and gives the first \emph{direct} constructive proof of the entropy gap. The main idea, which is explained in detail in \S \ref{s.entropygap}, is to approximate orbit segments in $\Sing$ by orbit segments in $\Reg$ with the specification property. We then use this property to build a collection of orbits with greater topological entropy than the singular set.  

We now state our results for the family of potentials $q \ph^u$. When $M$ is a surface, an easy argument provided in \S \ref{s.mainresults} shows that $P(\Sing, q \ph^u) =0$, and that $P(q \ph^u)>0$ for $q<1$. Thus, the following result is a corollary of Theorem \ref{t.multiples}.

\begin{thmx}\label{t.geometric}
If $M$ is a closed rank 1 surface, then the geodesic flow has a unique equilibrium state $\mu_q$ for the potential $q\vg$ for each $q\in(-\infty, 1)$,
and the function $q\mapsto P(q\vg)$ is $C^1$ on this interval.  
Each $\mu_q$ is hyperbolic, fully supported, and is the weak$^\ast$ limit of weighted regular closed geodesics.
\end{thmx}
It follows from work of Ledrappier, Lima, and Sarig \cite{LS, LLS} that these equilibrium states are Bernoulli, see \S\ref{s.mainresults}. For rank 1 surfaces, this uniqueness result is optimal;
any invariant measure supported on $\Sing$ is an equilibrium state for $ q \vg$ when $q\geq 1$. In higher dimensions, $\Sing$ can have positive entropy, but we can still exploit the entropy gap $\htop(\Sing) < \htop(\FFF)$. An easy argument, which we give in \S \ref{s.mainresults}, gives the following result on $q \vg$ for higher dimensional manifolds as a consequence of the entropy gap.

\begin{thmx}\label{t.highergeometric}
Let $\mathcal F$ be the geodesic flow for a closed rank 1 manifold. There exists $q_0>0$ such that the potential $q\vg$ has a unique equilibrium state $\mu_q$ for each $q\in (-q_0, q_0)$. The function $q \mapsto P(q \vg)$ is $C^1$ on $(-q_0, q_0)$.  
Each $\mu_q$ is hyperbolic, fully supported, and is the weak$^\ast$ limit of weighted regular closed geodesics.
\end{thmx}

The entropy gap, and hence the $q_0$ provided by this theorem, may be arbitrarily small, see \S\ref{s.mainresults}.
 If $\htop(\Sing)=0$, we observe in \S\ref{s.mainresults} that the pressure gap holds on $(-q_0,1)$. In \S \ref{sec:heintze}, we give 
an example of a 3-dimensional $M$ with nonempty singular set for which the pressure
gap holds for all $q\in \RR$, and thus $q_0=\infty$. It is an open question whether the inequality $P(\Sing, q\vg)< P(q\vg)$ always holds for all $q \in (-\infty, 1)$ when $\dim(M)>2$.

As a further application, we prove in \S\ref{sec:genericity} that if the singular set is a finite union of periodic orbits, then our uniqueness results hold for $C^0$-generic H\"older potentials; this includes the case when $\dim M = 2$ and the metric is real analytic.



The outline of the paper is as follows.
In \S \ref{s.background}, we introduce background material, particularly the existence and uniqueness result from \cite{CT4}. In \S \ref{s.general}, we state our most general theorem on equilibrium states for geodesic flow,  Theorem \ref{t.geodgeneral}.
In \S\S \ref{s.spec}-\ref{s.geodgeneral}, we build up a proof of Theorem \ref{t.geodgeneral}. 
In \S \ref{sec:Bowen}, we investigate regularity of the potentials $q\vg$.
In \S \ref{s.entropygap}, we prove Theorem \ref{thm:pressure-gap}.
In \S \ref{s.mainresults}, we complete the proofs of Theorems \ref{t.multiples}, \ref{t.geometric}, and \ref{t.highergeometric}.
In \S \ref{s.examples}, we apply our results to some examples.

\section{Preliminaries}\label{s.background}

In this section, we review definitions and results concerning pressure, specification, expansivity, geometry, and hyperbolicity.

\subsection{Topological Pressure}

Let $X$ be a compact metric space, $\mathcal{F} =\{f_t\}$ a continuous flow on $X$,  and $\varphi\colon X\to \mathbb{R}$ a continuous function.  We denote the space of $\FFF$-invariant probability measures on $X$ by $\mathcal{M}(\mathcal{F})$, and note that $\mathcal{M}(\mathcal{F})=\bigcap_{t\in\mathbb{R}}\mathcal{M}(f_t)$.  We denote the space of ergodic $\FFF$-invariant probability measures on $X$ by $\mathcal{M}^e(\mathcal{F})$.

We recall the definition of the topological pressure of $\varphi$ with respect to $\mathcal{F}$, referring the reader to \cite{BR75, Wa} for more background.  For $\epsilon>0$ and $t>0$ the \emph{Bowen ball of radius $\epsilon$ and order $t$} is
\[
B_t(x,\epsilon)= \{ y\in M\mid d(f_s x, f_s y)<\epsilon\text{ for all }0\leq s\leq t\}.
\]
Given $\epsilon>0$ and $t\in[0, \infty)$, a set $E\subset X$ is \emph{$(t, \epsilon)$-separated} if for all distinct $x, y\in E$ we have $y\notin \overline{B_t(x, \epsilon)}$.

We write $\Phi(x,t) = \int_{0}^{t} \varphi(f_sx)\,ds$ for the integral of $\varphi$ along an orbit segment of length $t$.  Let
\begin{equation}\label{eqn:Lambda-sep}
\Lambda(\varphi,\epsilon, t) = \sup
\left\{ \sum_{x\in E} e^{\Phi(x, t)} \mid E\subset X \text{ is $(t,\epsilon)$-separated} \right\}.
\end{equation}
Then the \emph{topological pressure of $\varphi$ with respect to $\FFF$} is
\[
P(\mathcal{F}, \varphi) = \lim_{\epsilon\to 0} \limsup_{t\to\infty} \frac 1t \log \Lambda(\varphi,\epsilon, t).
\]
The dependence on $\mathcal{F}$ will usually be suppressed in the notation. 

The \emph{variational principle for pressure} 
states that if $X$ is a compact metric space and $\mathcal{F}$ is a continuous flow on $X$, then 
\[
P(\mathcal{F}, \varphi)=\sup_{\mu\in \mathcal{M}(\mathcal{F})}\left\{ h_{\mu}(\mathcal{F}) +\int \varphi \,d\mu\right\}.
\]
A measure achieving the supremum is an \emph{equilibrium state for $\varphi$}. If the entropy map $\mu \mapsto h_\mu$ is upper semi-continuous then equilibrium states exist for each continuous potential function. This is the case in our setting since the flow is $C^{\infty}$ 
\cite{sN89}; it also follows from entropy-expansivity which is proved in our setting in \cite{knieper98}.


\subsection{Criteria for uniqueness of equilibrium states}\label{s.abstract}

We review the general result proved by the second and fourth authors in~\cite{CT4} concerning the existence of a unique equilibrium state.   

Given a flow $(X,\mathcal{F})$, we think of $X\times [0,\infty)$ as the space of finite-length orbit segments by identifying $(x,t)$ with $\{ f_s(x) : 0\leq s< t\}$.  Given $\mathcal{C}\subset X\times [0, \infty)$ and $t\geq 0$ we let 
$\mathcal{C}_t=\{x\in X\, :\, (x,t)\in\mathcal{C}\}$.
The partition function associated to $\mathcal{C}$ is
\[
\Lambda(\mathcal{C}, \varphi, \delta, t)=\sup\bigg \{ \sum_{x\in E} e^{\Phi(x,t)} :  E\subset \mathcal{C}_t \text{ is }(t, \delta)\text{-separated}\bigg\}.
\]
When $\mathcal{C}=X\times [0,\infty)$ this reduces to \eqref{eqn:Lambda-sep}. 
The \emph{pressure of $\varphi$ on $\mathcal{C}$} is
\[
P(\mathcal{C}, \varphi)=\lim_{\delta\to0} \limsup_{t\to \infty} \frac{1}{t}\log \Lambda(\mathcal{C}, \varphi, \delta, t).
\]
For $\mathcal{C}=\emptyset$ we then define $P(\emptyset, \varphi)=-\infty.$

We can ask for the Bowen property and the specification property, defined below, to hold only on $\CCC$ rather than the whole space.

\begin{defn}\label{def:spec}
A collection of orbit segments $\mathcal{C}\subset X\times [0, \infty)$ has \emph{specification at scale $\rho>0$} if  there exists $\tau=\tau(\rho)$ such that for every $(x_1, t_1)$, $\dots, (x_N, t_N)\in \mathcal{C}$  there exist a point $y\in X$ and times $\tau_1, \dots, \tau_{N-1}\in[0, \tau]$ such that for $s_0=\tau_0=0$ and $s_j=\sum_{i=1}^j t_i + \sum_{i=1}^{j-1} \tau_i$, we have
\[
f_{s_{j-1} + \tau_{j-1}}(y)\in B_{t_j}(x_j, \rho)
\]
for every $j\in \{ 1,\dots, N\}$.  A collection $\mathcal{C}\subset X\times [0, \infty)$ has \emph{specification} if it has specification at all scales.
If $\mathcal{C}= X\times [0, \infty)$ has specification, then we say the flow has specification.
\end{defn}

The definition above extends the specification property for the flow originally studied by Bowen, see \cite{Franco, KH}, and is the property that is used in \cite{CT4}. 
In Theorem \ref{specgeoflow}, we prove a stronger version of this property for a suitable $\CCC$, in which 
the conclusion that `there exist a point $y$ and times $\tau_1, \dots, \tau_{N-1}\in[0, \tau]$' is replaced with the conclusion that `for \emph{every} collection of times $\tau_1, \ldots, \tau_{N-1}$ with $\tau_i \geq \tau$ for all $i$, there exists a point $y$'. That is, we are able to take all the transition times to be exactly $\tau$, or any length at least $\tau$ that we choose.
For our purposes, it is convenient to also use the notation $T_j= s_{j-1}+\tau_{j-1}$ for the time at which the orbit of $y$ is near $x_j$; see Figure \ref{fig:specification} for the relationship between the various times.

\begin{figure}[htbp]
\includegraphics[width=\textwidth]{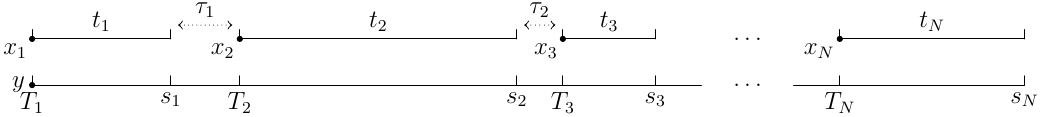}
\caption{Book-keeping in the specification property.}
\label{fig:specification}
\end{figure}

\begin{defn}\label{d.bowen}
We say that $\varphi\colon X\rightarrow \mathbb{R}$ has the \emph{Bowen property on $\mathcal{C}\subset X\times [0, \infty)$} if there are $\eps, K> 0$ such that for all $(x,t)\in \CCC$ and $y\in B_t(x,\eps)$, we have 
$
|\Phi(x,t)-\Phi(y,t)|\leq K.
$
\end{defn}
If $\varphi$ has the Bowen property on $\CCC = X \times [0, \infty)$, then our definition agrees with the original definition of Bowen.

\begin{defn}\label{def.decomp} 
A \emph{decomposition for $X\times [0, \infty)$} consists of three collections $\PPP, \GGG, \SSS\subset X\times [0, \infty)$ for which there exist three functions $p, g, s\colon X\times [0, \infty)\rightarrow [0, \infty)$ such that for every $(x,t)\in X\times [0, \infty)$, the values $p=p(x,t)$, $g=g(x,t)$, and $s=s(x,t)$ satisfy $t=p+g+s$, and 
\[
(x,p)\in \PPP,\quad
(f_p(x), g)\in \GGG,\quad
(f_{p+g}(x), s)\in \SSS.
\]
\end{defn}

The conditions we are interested in depend only on the collections $(\PPP, \GGG, \SSS)$ rather than the functions $p$, $g$, $s$. However, we work with a fixed choice of $(p,g,s)$ for the proof of the abstract theorem to apply.

We will construct a decomposition $(\PPP, \GGG, \SSS)$ such that  $\GGG$ has specification, the function $\varphi$ has the Bowen property on $\GGG$, and the pressure on $[\PPP]\cup [\SSS]$ is less than the pressure of the entire system, where
\[
[\mathcal{P}]:=
\{(x,n)\in X\times \mathbb{N} : (f_{-s}x, n+s+t)\in \mathcal{P}\textrm{ for some }s, t\in [0,1]\}
\]
and similarly for $[\SSS]$.
The reason that we control the pressure of $[\PPP]\cup[\SSS]$ rather than the collection $\PPP \cup \SSS$ is a consequence of a technical step in the proof of the abstract result in \cite{CT4} that required a passage from continuous to discrete time.

For $x\in X$ and $\epsilon>0$ we let the \textit{bi-infinite Bowen ball} be
\[
\Gamma_\epsilon(x)=\{y\in X\, :\, d(f_tx, f_ty)\leq \epsilon \textrm{ for all }t\in\mathbb{R}\}.
\]

\begin{defn}
The \emph{set of non-expansive points at scale $\eps$} is  
\[
\mathrm{NE}(\epsilon):=\{ x\in X \mid \Gamma_\epsilon(x)\not\subset  f_{[-s,s]}(x) \text{ for any }s>0 \},
\]
where $f_{[a,b]}(x) = \{f_tx : a \leq t \leq b\}$.
\end{defn}

\begin{defn} 
Given a potential $\varphi$, the \emph{pressure of obstructions to expansivity} is $P^\perp_{\mathrm{exp}}(\varphi):=\lim_{\eps\to 0} P^\perp_{\mathrm{exp}}(\varphi, \eps)$, where
\[
P^\perp_{\mathrm{exp}}(\varphi, \epsilon)=\sup_{\mu\in \mathcal{M}^e(\mathcal{F})}\left\{
h_\mu(f_1) + \int\varphi\, d\mu\, :\, \mu(\mathrm{NE}(\eps))=1\right\}.
\]
\end{defn}

The point of this definition is that every ergodic measure whose free energy exceeds $P^\perp_{\mathrm{exp}}(\varphi)$ gives zero measure to the non-expansive set, and thus `sees' only expansive behavior. 

We can now state the abstract theorem that we will use to prove our uniqueness results.

\begin{thm}\label{t.abstract}\cite[Theorem A]{CT4}
Let $(X, \mathcal{F})$ be a flow on a compact metric space, and $\varphi:X\rightarrow \mathbb{R}$ be a continuous potential function.  Suppose that $P^{\perp}_{\mathrm{exp}}(\varphi)<P(\varphi)$ and $X\times [0, \infty)$ admits a decomposition $(\PPP, \GGG, \SSS)$ with the following properties:
\begin{enumerate}[label=\textup{{(\Roman{*})}}]
\item $\GGG$ has specification;\label{cond:spec}
\item $\varphi$ has the Bowen property on $\GGG$;
\item $P([\PPP]\cup [\SSS], \varphi)<P(\varphi).$
\end{enumerate}
Then $(X, \mathcal{F}, \varphi)$ has a unique equilibrium state $\mu_\ph$.
\end{thm}

\subsection{Pressure and periodic orbits for geodesic flows} \label{s.periodicpressure} 
For $a<b$, let $\Per_{R}(a, b]$ denote the set of closed regular geodesics with length in the interval $(a,b]$. 
For each such geodesic $\gamma$, let
$\Phi(\gamma)$ be the value given by integrating $\ph$ around $\gamma$; that is, $\Phi(\gamma):=\Phi(v, |\gamma|) = \int_0^{|\gamma|} \ph(f_t v)\,dt$, where $v\in T^1M$ is tangent to $\gamma$ and $|\gamma|$ is the length of $\gamma$.
Given $T,\delta>0$, let
\begin{equation}\label{eqn:CTdelta}
\Lambda^\ast_{\Reg} (\ph, T,\delta) = \sum_{\gamma\in \Per_R(T-\delta,T]} e^{\Phi(\gamma)}.
\end{equation}
For a closed geodesic $\gamma$, let $\mu_\gamma$ be the normalized Lebesgue measure around the orbit.  We consider the measures
\[
\mureg_{T,\delta} = \frac 1{\Lambda^\ast_{\Reg} (\ph, T,\delta)} \sum_{\gamma\in \Per_R(T-\delta,T]} e^{\Phi(\gamma)} \mu_\gamma.
\]
We say that
\emph{regular closed geodesics weighted by $\varphi$ equidistribute} to a measure $\mu$ if $\lim_{T\to\infty} \mureg_{T,\delta} = \mu$ in the weak* topology for every $\delta>0$.
Equidistribution of weighted periodic orbits for equilibrium states  was first investigated for Axiom A flows by Parry \cite{wP88}, and for geodesic flow on manifolds of non-positive curvature by Pollicott \cite{mP96}.

For any $\epsilon>0$ smaller than the injectivity radius of $M$ and any $\delta>0$, choosing $v_\gamma \in T^1M$ tangent to each $\gamma \in\Per_{R}(T-\delta,T]$ gives an $(\epsilon, T)$-separated set \cite[\S6]{knieper98}.   Since $| \Phi(\gamma)- \Phi(v_\gamma, T)| \leq \delta \|\ph\|$, it follows that
\begin{equation}\label{eqn:LCL}
\Lambda^*_\Reg (\ph,T,\delta) \leq e^{\delta \|\ph\|} \Lambda(\ph,\eps,T).
\end{equation}
This shows that
\begin{equation}\label{eqn:limsup}
\limsup_{T\to\infty} \frac 1T \log \Lambda^\ast_{\Reg} (\ph, T,\delta) \leq P(\ph).
\end{equation}
It is straightforward to verify that the value of the $\limsup$ is independent of the choice of $\delta>0$. The expression on the left hand side is the \emph{upper pressure of regular closed geodesics}, and we denote this by $\overline P_{\Reg}^\ast(\ph)$. We also define
\[
\underline P_{\Reg, \delta}^\ast(\ph) = \liminf_{T\to\infty} \frac 1T \log \Lambda^\ast_{\Reg} (\ph, T,\delta).
\]
If this quantity is independent of $\delta>0$, and agrees with  $\overline P^\ast_{\Reg}(\ph)$, we can define the \emph{pressure of regular closed geodesics} to be
\[
P_{\Reg}^\ast(\ph) = \lim_{T\to\infty} \frac 1T \log \Lambda^\ast_{\Reg} (\ph, T,\delta).
\]
The second half of the proof of the variational principle in \cite[Theorem 9.10]{Wa} gives the following.

\begin{lem}\label{lem:equidist}
If $\delta>0, T_k\to\infty$ satisfy $\frac 1{T_k} \log \Lambda^\ast_{\Reg}(\ph, T_k,\delta) \to P(\ph)$ and $\mureg_{T_k,\delta} \to \mu$ as $k\to\infty$, then $\mu$ is an equilibrium state for $\ph$. 
\end{lem}
Thus, if $\underline P_{\Reg, \delta}^\ast(\ph) = P(\ph)$, every limit of the measures $\{\mureg_{T,\delta}\}$ is an equilibrium state, which gives the following proposition.

\begin{prop}\label{prop:equidist}
If
$P^\ast_{\Reg} (\ph) =P(\ph)$ and $\ph$ has a unique equilibrium state $\mu$, then the regular closed geodesics weighted by $\ph$ equidistribute to $\mu$.
\end{prop}

The growth rate in \eqref{eqn:limsup}, with the sum restricted to prime closed geodesics, was studied by Gelfert and Schapira \cite{GS14}, who called it the \emph{regular Gurevic pressure}.
They consider a variety of definitions of topological pressure for geodesic flow on rank 1 manifolds;
we refer the reader to \cite{GS14} for details.



\subsection{Geometry}\label{sec:geometry}

Throughout the paper $M$ denotes a compact, connected, boundaryless smooth manifold with a smooth Riemannian metric $g$, with non-positive sectional curvatures at every point.

For each $v\in TM$ there is a unique constant speed
geodesic denoted $\gamma_v$ such that $\dot{\gamma}_v(0)=v$.  The \emph{geodesic flow} $\mathcal{F}=(f_t)_{t\in\mathbb{R}}$ acts on $TM$  by $f_t(v)=(\dot\gamma_v)(t)$.  The unit tangent bundle $T^1M$ is compact and $\FFF$-invariant; from now on we restrict to the flow on $T^1M$.
We recall some well-known properties of geodesic flow in this setting; see \cite{wB95,pE99} for more details.

We write $d$ for the distance function on $M$ induced by the Riemannian metric. The Riemannian metric on $M$ lifts to the \emph{Sasaki metric} on $TM$. We write  $\dS$ for the distance function this Riemannian metric induces on $T^1M$.
Another distance function on $T^1M$ 
was used by Knieper in \cite{knieper98}:
\begin{equation}\label{eqn:dK}
\dK(v,w) = \max \{d(\gamma_v(t), \gamma_w(t)) \mid t \in [0,1] \}.
\end{equation}
We call $\dK$ the \emph{Knieper metric};  it is not necessarily induced by a Riemannian metric on $TM$. The two distance functions $\dS$ and $\dK$ are uniformly equivalent.  
We will typically consider Bowen balls with respect to the Knieper metric, so
\begin{align*}
B_T(v,\eps) &= \{w\in T^1 M : \dK(f_t w, f_t v) < \eps \text{ for all } 0\leq t \leq T \} \\
&= \{w\in T^1 M : d(\gamma_w(t), \gamma_v(t)) < \eps \text{ for all } 0 \leq t \leq T+1 \}.
\end{align*}

A \emph{Jacobi field} along a geodesic $\gamma$ is a vector field along $\gamma$ satisfying
\begin{equation}\label{eqn:Jacobi}
J''(t) + R(J(t), \dot{\gamma}(t))\dot{\gamma}(t)=0,
\end{equation}
where $R$ is the Riemannian curvature tensor on $M$ and $'$ represents covariant differentiation along $\gamma$.

If $J(t)$ is a Jacobi field along a geodesic $\gamma$ and both  $J(t_0)$ and $J'(t_0)$ are orthogonal to $\dot{\gamma}(t_0)$ for some $t_0$, then $J(t)$ and $J'(t)$ are orthogonal to $\dot{\gamma}(t)$ for all $t$.  Such a Jacobi field is an \emph{orthogonal Jacobi field}.

A Jacobi field $J(t)$ along a geodesic $\gamma$ is \emph{parallel at $t_0$} if $J'(t_0)=0$. A Jacobi field $J(t)$ is parallel if it is parallel for all $t \in \RR$.

Nonpositivity of the sectional curvatures implies that $\|J(t)\|$ and $\|J(t)\|^2$ are convex functions of $t$.



\subsubsection{Invariant foliations}\label{sec:foliations}

We describe three important $\FFF$-invariant subbundles $E^u$, $E^s$, and $E^c$ of $TT^1 M$.  The bundle $E^c$ is spanned by the vector field $V$ that generates the flow $\mathcal{F}$.  To describe $E^u$ and $E^s$, we first write $\JJJ(\gamma)$ for the space of orthogonal Jacobi fields for $\gamma$; given $v\in T^1 M$ there is a natural isomorphism $\xi \mapsto J_\xi$ between $T_vT^1M$ and $\JJJ(\gamma_v)$, which has the property that
\begin{equation} \label{compare}
\|df_t(\xi)\|^2= \|J_\xi(t)\|^2+\|J'_\xi(t)\|^2.
\end{equation}
An orthogonal Jacobi field $J$ along a geodesic $\gamma$ is \emph{stable} if $\|J(t)\|$ is bounded for $t\geq 0$, and \emph{unstable} if it is bounded for $t\leq 0$.  The stable and the unstable Jacobi fields each form linear subspaces of 
$\JJJ(\gamma)$, which we denote by $\JJJ^s(\gamma)$ and $\JJJ^u(\gamma)$, respectively.
The corresponding stable and unstable subbundles of $TT^1M$ are
\begin{align*}
E^u(v)&=\{ \xi \in T_v(T^1M) : J_\xi \in \JJJ^u(\gamma_v) \}, \\
E^s(v)&=\{ \xi \in T_v(T^1M) : J_\xi \in \JJJ^s(\gamma_v) \}.
\end{align*}
We also write $E^{cu} = E^c\oplus E^u$ and $E^{cs} = E^c\oplus E^s$.
The subbundles have the following properties (see \cite{pE99} for details):
\begin{itemize}
\item $\dim(E^u)=\dim(E^s)= n-1$, and $\dim(E^c)=1$;
\item the subbundles are invariant under the geodesic flow;
\item the subbundles depend continuously on $v$, see \cite{pE99, GW99};
\item $E^u$ and $E^s$ are both orthogonal to $E^c$;
\item $E^u$ and $E^{s}$ intersect non-trivially if and only if $v \in \Sing$;
\item $E^\sigma$ is integrable to a foliation $W^\sigma$ for each $\sigma\in \{u,s,cs,cu\}$.
\end{itemize}
It is proved in \cite[Theorem 3.7]{wB82} that the foliation $W^s$ is minimal in the sense that $W^s(v)$ is dense in $T^1M$ for every $v \in T^1M$. Analogously, the foliation $W^u$ is also minimal. It follows from the minimality of $W^s$ that the geodesic flow is topologically mixing \cite[Theorem 3.5]{wB82}.  

%

\subsubsection{$H$-Jacobi fields and the function $\lambda$}\label{sec:lambda-def}

Our hyperbolicity estimates will be given in terms of a function $\lambda\colon T^1M\to [0,\infty)$, which we now describe. 
Let $\gamma$ be a unit speed geodesic with $\gamma(t_0)=p \in M$, 
 and let $H \subset M$ be a hypersurface orthogonal to $\gamma$ at $p$.  
Let $\JJJ_H(\gamma)$ be the set of \emph{$H$-Jacobi fields} obtained by varying $\gamma$ through unit speed geodesics orthogonal to $H$.  This is an $(n-1)$-dimensional Lagrangian subspace of $\JJJ(\gamma)$.  Writing $H^{s,u}$ for the stable and unstable horospheres, we have $\JJJ_{H^{s,u}}(\gamma) = \JJJ^{s,u}(\gamma)$. 
Let $\UUU\colon T_p H\to T_p H$ be the symmetric linear operator defined by $\UUU(v)=\nabla_vN$, where $N$ is the field of unit vectors normal to $H$ on the same side as $\dot \gamma (t_0)$; this determines the second fundamental form of $H$.

\begin{lem}\label{J'-is-U}
If $J$ is  an $H$-Jacobi field along $\gamma$, then $J'(t_0)=\UUU(J(t_0))$.
\end{lem}
\begin{proof}  An $H$-Jacobi field $J$ along $\gamma$ is determined by a variation $\alpha(s,t)$ of $\gamma$ through unit speed geodesics such that $\alpha(s,t_0) \in H$ and  $\frac{\nabla \alpha}{\partial s}(s,t_0)$ is a field of unit normals to $H$.  Using the symmetry of the Levi-Civita connection we can make the calculation
\[
J'(t_0)=\frac{\nabla}{\partial t}\frac{\partial\alpha}{\partial s}(0,t_0)=\frac{\nabla}{\partial s}\frac{\partial \alpha}{\partial t}(0,t_0)=\nabla_{J(t_0)} N=\UUU(J(t_0)). \qedhere
\]
\end{proof}

The key consequence of Lemma \ref{J'-is-U} is that writing $\lambda_H$ for the minimum eigenvalue of the linear map $\UUU$, 
every $H$-Jacobi field $J$ has
\begin{equation}\label{eqn:JJ'}
\langle J, J \rangle '(t_0)
= 2\langle J, \UUU J\rangle (t_0)
\geq 2 \lambda_H\langle J(t_0), J(t_0)\rangle,
\end{equation}
which gives $(\log \|J\|^2)'(t_0) \geq 2\lambda_H$, and in particular
\begin{equation}\label{eqn:logJ}
(\log\|J\|)'(t_0) \geq \lambda_H.
\end{equation}
Let $\UUU^s_v \colon T_{\pi v} H^s \to T_{\pi v} H^s$ be the symmetric linear operator associated to the stable horosphere $H^s$, and similarly for $\UUU^u_v$.  Then $\UUU_v^u$ and $\UUU_v^s$ depend continuously on $v$, $\UUU^u$ is positive semidefinite, $\UUU^s$ is negative semidefinite, and $\UUU^u_{-v}=-\UUU^s_v$.  

Let $\Lambda$ be the maximum eigenvalue of $\UUU^u_v$ over all $v\in T^1M$.  If $J_\xi$ is a stable or unstable Jacobi field we have
$\|J_\xi'(t)\|\leq \Lambda \|J_\xi(t)\|$ for all $t$.  Thus if $\xi$ is in $E^s$ or $E^u$, 
then by \eqref{compare} and Lemma \ref{J'-is-U}, $\|df_t\xi\|$ and $\|J_\xi(t)\|$ are uniformly comparable in the sense that
\begin{equation}\label{eqn:unif-comp}
\| J_\xi(t)\|^2  \leq \| df_t \xi\|^2\leq (1+\Lambda^2)\|J_\xi(t)\|^2.
\end{equation}

\begin{defn} \label{d.lambda}
For $v \in T^1M$, let $\lambda^u(v)$ be the minimum eigenvalue of $\UUU^u_v$ and let $\lambda^s(v) = \lambda^u(-v)$. Let $\lambda(v) = \min ( \lambda^u(v), \lambda^s(v))$.
\end{defn}

The functions $\lambda^u$, $\lambda^s$, and $\lambda$ are continuous since the map $v\mapsto \UUU^u_v$ is continuous.
By positive (negative) semidefiniteness of $\UUU^{u,s}$, we have $\lambda^{u,s} \geq 0$.  
The following is an immediate consequence of \eqref{eqn:logJ}.

\begin{lem}\label{l.lambda_rate}
Given  $v\in T^1M$, let $J^u$ be an unstable Jacobi field along $\gamma_v$ and $J^s$ be a stable Jacobi field along $\gamma_v$. Then 
\[
\|J^u(T)\|\geq e^{\int_0^T \lambda^u(f_tv)dt}\|J^u(0)\| \textrm{ and } \|J^s(T)\|\leq e^{-\int_0^T \lambda^s(f_tv) dt}\|J^s(0)\|.
\]
\end{lem}

In \S\ref{sec:lambda} we collect some more properties of the functions $\lambda,\lambda^s,\lambda^u$.

\subsubsection{Leaf metrics}\label{sec:leaf-metrics}
In addition to the metrics $\dS$ and $\dK$ on $T^1M$, we will need to consider for each $v\in T^1M$ the \emph{intrinsic metric} on $W^s(v)$ defined by 
\begin{equation}\label{eqn:ds}
d^s(u, w) = \inf \{ \ell(\pi\zeta) \mid \zeta\colon [0,1]\to W^s(v), \zeta(0)=u, \zeta(1)=w\},
\end{equation}
where $\pi\colon T^1 M \to M$ is the canonical projection, $\ell$ denotes length of the curve in $M$, and the infimum is over all $C^1$ curves $\zeta$ connecting $u$ and $w$ in $W^s(v)$.  In other words, $d^s(u,w)$ is the distance between the footprints $\pi(u)$ and $\pi(w)$ when we restrict ourselves to motion along the horosphere $H^s(v) = \pi W^s(v)$.  
Given $\rho>0$, the \emph{local stable leaf} through $v$ of size $\rho$ is
\[
W_\rho^s(v) := \{w\in W^s(v) : d^s(v,w) \leq \rho\}.
\]
Define $d^u$, $W_\rho^u(v)$ similarly.  Locally, the intrinsic metric on $W^{cs}(v)$ is
\[
d^{cs}(u,w) = |t| + d^s(f_t u, w),
\]
where $t$ is the unique value so $f_t u \in W^s(w)$. This extends to a metric on the whole leaf $W^{cs}(v)$. We define $d^{cu}$, $W_\rho^{cs}(v)$, $W_\rho^{cu}(v)$ in the obvious way.

The minimality of the foliations $W^s$ and $W^u$, together with a standard compactness argument given in \cite[Lemma 8.1]{CFT}, gives the following result.

\begin{lem}\label{lem:unif-dense}
For every $\eps>0$, there exists $R>0$ such that $W_R^u(v)$ and $W_R^s(v)$ are $\eps$-dense in $T^1M$ for every $v\in T^1M$.
\end{lem}


If we restrict $\rho$ to be small, then the intrinsic metrics are uniformly equivalent to $\dS$ and $\dK$.  The following lemma lets us obtain a relationship between the leaf metrics and the dynamical metric
\begin{equation}\label{eqn:dt}
d_t(v,w) = \sup_{\tau\in [0,t]} \dK(f_\tau v, f_\tau w)
=\sup_{\tau\in [0,t+1]} d(\gamma_v(\tau), \gamma_w(\tau)).
\end{equation}

\begin{lem}\label{lem:monotonic-leaf}
For all $v\in T^1M$, $w\in W^s(v)$, $w' \in W^u(v)$, and $t\geq 0$, we have
\begin{align}
\label{eqn:dst}
& e^{-\Lambda t} d^s(v,w) \leq d^s(f_t v, f_t w) \leq d^s(v,w), \\
\label{eqn:dut}
& d^u(v,w') \leq d^u(f_t v, f_t w') \leq e^{\Lambda t} d^u(v,w').
\end{align}
\end{lem}
\begin{proof}
We prove \eqref{eqn:dst}; the proof of \eqref{eqn:dut} is similar.  Let $\zeta\colon [0,1]\to W^s(v)$ be a curve with $\zeta(0) = v$ and $\zeta(1)=w$, and let $\{\gamma_r : r\in [0,1]\}$ be the one-parameter family of geodesics determined by $\gamma_r'(0) = \zeta(r)$.  Each $\gamma_r$ is orthogonal to the stable horospheres of $\dot\gamma_r(t)$, so we obtain a family of stable Jacobi fields $J_r(t) := \frac{\partial}{\partial r} \gamma_r(t) \in \mathcal{J}^s(\zeta_r)$.  Since $0\leq \lambda^s \leq \Lambda$, Lemma \ref{l.lambda_rate} gives $e^{-\Lambda t} \|J_r(0)\| \leq \|J_r(t)\| \leq \|J_r(0)\|$ for all $t\geq 0$.  Since $\ell(f_t \zeta) = \int_0^1 \|\frac{\partial}{\partial r} \gamma_r(t)\| \,dr = \int_0^1 \|J_r(t)\| \,dr$, we obtain $e^{-\Lambda t} \ell(\zeta) \leq \ell(f_t \zeta) \leq \ell(\zeta)$, and taking an infimum over all $\zeta$ completes the proof.
\end{proof}

Given $v\in T^1M$ and $w\in W^{cs}(v)$, it follows from \eqref{eqn:dst} that the function $t\mapsto d^{cs}(f_t v,f_t w)$ is non-increasing, so \eqref{eqn:dK} gives
\begin{equation}\label{eqn:dtdcs}
d_t(v,w) \leq \dK(v,w) \leq d^{cs}(v,w).
\end{equation}
For $w\in W^u(v)$, we use \eqref{eqn:dut} to obtain
\begin{equation}\label{eqn:dKdu}
\begin{aligned}
&\dK(v,w) \leq e^\Lambda d^u(v,w), \\
&d_t(v,w) \leq d^u(f_{t+1} v, f_{t+1} w) \leq  e^\Lambda d^u(f_t v, f_t w).
\end{aligned}
\end{equation}

\section{Decompositions for geodesic flow}\label{s.general}

\subsection{Main theorem} \label{s.decomp}

Now we state our main uniqueness result, which we apply to obtain Theorem \ref{t.multiples}.


\begin{thm}\label{t.geodgeneral}
Let $\varphi\colon T^1M\rightarrow \mathbb{R}$ be continuous. If $P(\mathrm{Sing}, \varphi)< P(\varphi)$, and for all $\eta>0$  the potential $\varphi$ has the Bowen property on 
\[
\GGG(\eta) = \bigg\{(v, t) : \int_0^\tau \lambda(f_s v)\,ds \geq \eta \tau,  \int_0^\tau \lambda(f_{-s}f_{t} v)\,ds \geq \eta \tau ~\forall \tau\in[0,t] \bigg\},
\]
then the geodesic flow has a unique equilibrium state for $\varphi$.  This equilibrium state is fully supported and hyperbolic. The weighted sums of regular closed geodesics satisfy counting estimates given in \eqref{eqn:weighted-sums}, and are equidistributed with respect to $\mu$ as described in \S\ref{s.periodicpressure}.
\end{thm}


The set of potentials having the Bowen property on $\GGG(\eta)$ for all $\eta>0$  contains all H\"older potentials,  all scalar multiples of the geometric potential, and all linear combinations of such potentials; see \S\ref{sec:Bowen}. 


We build up a proof of Theorem \ref{t.geodgeneral} in the next few sections. We start by describing the decomposition we use to apply Theorem \ref{t.abstract}.

\begin{figure}[htbp]
\includegraphics[width=.7\textwidth]{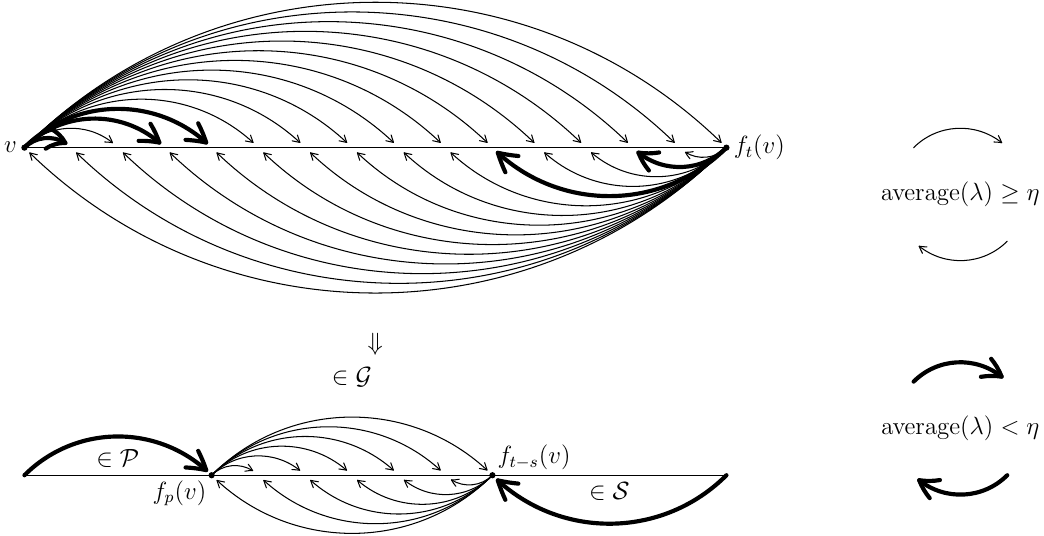}
\caption{Decomposing an orbit segment.}
\label{fig:decomposition}
\end{figure}

Given $\eta>0$, let $\BBB(\eta) := \big\{ (v, T) : \int_0^T\lambda(f_tv)\,dt< \eta T \big\}$.
We define maps 
$p,g,s \colon X \times [0, \infty) \to [0, \infty)$.  Given an orbit segment $(v,t)$, take $p=p(v,t)$ to be the large

st time such that 
$(v,p)\in \BBB(\eta)$.
Let $s=s(v,t)$ be the largest time in $[0,t-p]$ such that  the orbit segment $(f_{t-s}(v), s)$ is in $\mathcal{B}(\eta)$. The function $g$ determines the remaining part of the orbit segment denoted $(f_pv, g)$, so $g=t-p-s$. It is easily checked that $(f_pv, g) \in\GGG(\eta)$; see Figure \ref{fig:decomposition}. Thus the triple $(\BBB(\eta), \GGG(\eta), \BBB(\eta))$ equipped with the functions $(p, g, s)$ determines a decomposition for $X \times [0, \infty)$ in the sense of Definition \ref{def.decomp}. We will show that if $P(\mathrm{Sing}, \varphi)< P(\varphi)$ and if $\eta>0$ is chosen sufficiently small, then the hypotheses of Theorem~\ref{t.abstract} are satisfied using the decomposition $(\BBB(\eta), \GGG(\eta), \BBB(\eta))$. This will guarantee uniqueness of the equilibrium state.
 


\subsection{Properties of $\lambda$}\label{sec:lambda} 
In the following two lemmas, we prove that the function $\lambda\colon T^1M \to [0, \infty)$ vanishes on $\Sing$, and if $\lambda(v)=0$, then there is a nontrivial orthogonal Jacobi field $J$ on $\gamma$ such that $J(t)$ is parallel for all $t\leq 0$ or for all $t\geq 0$.

\begin{lem}\label{lem:Sing-lambda}
The following are equivalent for $v\in T^1 M$.
\begin{enumerate}[label=\textup{(\alph{*})}]
\item $v\in \Sing$.
\item $\lambda^s(f_t v)=0$ for all $t\in \RR$.
\item $\lambda^u(f_t v)=0$ for all $t\in \RR$.
\end{enumerate}
\end{lem}
\begin{proof}
If $v\in \Sing$, then there is a parallel Jacobi field $J(t)$ along $\gamma_v(t)$, which gives $\lambda^s=\lambda^u=0$.
Since $\Sing$ is invariant, this gives (b) and (c).  

Now we show that (b) implies (a). If $\lambda^s(f_t v)=0$ for every $t\in \RR$, then for every $T\geq 0$ there is a stable Jacobi field $J_T$ along $\gamma_v$ that is parallel (with unit length) for $t\geq -T$.  By compactness we get a sequence $T_k\to\infty$ for which $J_{T_k}(0)$ and $J_{T_k}'(0)$ converge to some $J(0),J'(0) \in T_{\pi v} M$; the corresponding Jacobi field $J$ is parallel for all time, so $v\in \Sing$. The proof that (c) implies (a) is similar.
\end{proof}

\begin{lem}\label{lem.lsu}
Given $v\in T^1M$, the following are equivalent:
\begin{enumerate} [label=\textup{(\alph{*})}]
\item $\lambda^u(v)=0$;
\item For all $t\leq 0$, $\lambda^u(f_tv)=0$;
\item there is a nontrivial orthogonal Jacobi field $J$ on $\gamma_v$ such that $J(t)$ is parallel 
for all $t\leq 0$.
\end{enumerate}
The analogous result holds for $\lambda^s$ and $t\geq 0$.
\end{lem}


\begin{proof}
It is immediate that (c) $\Rightarrow$ (b) $\Rightarrow$ (a), so we just have to prove that (a) $\Rightarrow$ (c).  If $\lambda^u(v) = 0$, then there is a nonzero $w\in T_{\gamma(0)}{H^u(v)}$ with $\UUU^u_w=0$.  The corresponding $H^u(v)$-Jacobi field has $J'(0)=0$ by Lemma \ref{J'-is-U}, and is bounded for $t\leq 0$, so by convexity $\|J(t)\|$ is constant for $t\leq 0$.
Differentiating $\|J(t)\|^2$ gives $0 = \langle J', J \rangle = \langle \UUU_{f_t v}^u J, J\rangle$ for $t\leq0$. Since $\UUU^u$ is positive semidefinite symmetric, it follows that $\UUU_{f_t v}^u J = 0$, so $J(t)$ is parallel for $t\leq 0$.
\end{proof}



We also have the following quantitative version of Lemma \ref{lem:Sing-lambda}, and two corollaries which are useful for our topological pressure estimates.

\begin{prop}\label{prop:proximity}
For any $\delta>0$, there are $\eta>0$ and $T>0$ such that if $\lambda^s(f_t v) \leq \eta$ for all $t\in [-T,T]$, then $\dK(v,\Sing) < \delta$.  A similar result holds for $\lambda^u$.
\end{prop}
\begin{proof}
For $\eta>0$, consider the open set $A(\eta) = \{v\in T^1 M :$ there exists $t\in [-\eta^{-1}, \eta^{-1}]$ such that $\lambda^s(f_t v) > \eta\}$. By Lemma \ref{lem:Sing-lambda}, $\Reg =\bigcup_{\eta>0} A(\eta)$. Let $K = \{v\in T^1 M : d(v,\Sing) \geq \delta\}$.  Since $\{A(\eta)\}_{\eta>0}$ is an open cover for the compact set $K$,  there exists a finite subcover.  Since the sets $A(\eta)$ are nested, this implies that $K\subset A(\eta)$ for some choice of $\eta$. 
\end{proof}


\begin{cor}\label{lem:to-Sing}
Let $\lambda(v)=0$. Then $\dK(f_tv, \Sing) \to 0$ as $t \to \infty$ or $\dK(f_tv, \Sing) \to 0$ as $t \to - \infty$.
\end{cor}
\begin{proof}
By Lemma \ref{lem.lsu}, we have $\lambda^s(f_tv)=0$ for all $t\geq 0$, or  $\lambda^u(f_tv)=0$ for all $t\leq 0$. Suppose we are in the first case. Given $\delta>0$, by Proposition \ref{prop:proximity},  there exists $\eta,T>0$ such that if $\lambda^s(f_tw)\leq \eta$ for all $t\in [-T,T]$, then $\dK(w,\Sing)<\delta$.  Thus for every $\tau\geq T$, we conclude that $\dK(f_\tau v,\Sing) < \delta$. Thus $\dK(f_tv, \Sing) \to 0$ as $t \to \infty$. A similar argument applies in the case that $\lambda^u(f_tv)=0$ for all $t\leq 0$ to show that $\dK(f_tv, \Sing) \to 0$ as $t \to - \infty$.
\end{proof}

\begin{cor}\label{c.singular}
Let $\mu$ be an invariant measure such that $\lambda(v)=0$ for $\mu$-almost every $v$.  Then $\mathrm{supp}(\mu)\subset \Sing$.
\end{cor}
\begin{proof}
By Corollary \ref{lem:to-Sing}, if $\lambda(v)=0$ and $v\in\Reg$, then $v$ cannot be both forward recurrent and backward recurrent. 
Since $\mu$-a.e.\ $v$ is both forward and backward recurrent, 
we see that $\mu(\Reg)=0$.
\end{proof} 

\begin{cor}\label{c.hyperbolic}
Any ergodic measure $\mu$ with $\mu(\Reg)=1$ is hyperbolic.
\end{cor}
\begin{proof}
Let $Z$ be the set of $v\in T^1M$ such that there exists 
$\xi \in T_v(T^1M)$ orthogonal to the flow direction with $\lim_{t \to \infty} \frac{1}{t} \log \|df_t \xi\| =0$.  Given such a $\xi$, the unstable Jacobi field $J_\xi$ along $\gamma_v$ satisfies $\lim_{t \to \infty} \frac{1}{t} \log \|J_\xi(t)\| =0$ by \eqref{eqn:unif-comp}.  By Lemma  \ref{l.lambda_rate}, it follows that $\lim_{t \to \infty} \frac{1}{t}\int_0^t \lambda(f_sv)\,ds =0$ for every $v\in Z$. If $\mu$ is not hyperbolic, then $\mu(Z)=1$, and thus $\frac 1t \int_0^t \lambda(f_sv)\,ds \to 0$ for $\mu$-a.e.\ $v$.  By the ergodic theorem, this implies that $\int \lambda\,d\mu = 0$, which in turn implies that $\lambda(v)=0$ for $\mu$-a.e.\ $v$ since $\lambda\geq 0$. By Corollary \ref{c.singular}, $\supp(\mu)\subset \Sing$. This contradicts the hypothesis that $\mu(\Reg)=1$, so we conclude that $\mu$ is hyperbolic.
\end{proof}

\subsection{Uniform estimates}\label{sec:Geta}
For $\eta>0$, we define
\begin{equation}
\Reg(\eta) = \{v: \lambda(v) \geq \eta\}.
\end{equation} Note that if $(v,t)\in \GGG( \eta)$ for some $t>0$, then $\lambda(v)\geq \eta$ and $\lambda(f_t v)\geq \eta$, and thus $v\in \Reg(\eta)$ and $f_tv \in \Reg(\eta)$. Note that $\Reg(\eta_1) \subset \Reg(\eta_2)$ if $\eta_1 \geq \eta_2$ and each $\Reg(\eta)$ is compact. 

\begin{lem} \label{l.angle}
For all $\eta >0$, there exists $\theta>0$ so that for any $v \in \Reg(\eta)$,
we have $\measuredangle(E^u(v),E^s(v)) \geq \theta$.
\end{lem}
\begin{proof}
The angle is continuous in $v$ and positive on $\Reg(\eta)$.
\end{proof}
\begin{lem}\label{lem:dense}
$\{v : \lambda(v)>0\} = \bigcup_{\eta>0} \Reg(\eta)$ is dense in $T^1 M$.
\end{lem} 
\begin{proof}
Suppose the lemma is false. Then $\{v \in T^1M: \lambda^u(v) = 0\} \cup \{v \in T^1M: \lambda^s(v) = 0\}$ has interior. Since these sets are closed, at least one of them has interior. Assume $\{v \in T^1M: \lambda^u(v) = 0\}$ has interior; the other case is similar.
If $(X, \FFF)$ is a transitive flow whose non-wandering set is $X$, as is the case in our setting \cite[Remark 4.16]{eb73}, then there exists  $x_0 \in X$  with dense forward orbit \cite[Theorem 5.10]{Wa}. Thus, there exists a vector $w$ whose forward  orbit under the geodesic flow is dense in $T^1M$ and therefore enters $\{v \in T^1M: \lambda^u(v) = 0\}$ for arbitrarily large $t$. It follows from Lemma \ref{lem:Sing-lambda} that $w \in \Sing$. But $w \in \Reg$ because its forward orbit also enters $\Reg$.
\end{proof}



To go from the Jacobi field estimates in Lemma \ref{l.lambda_rate} to local estimates near orbit segments in $\GGG(\eta)$, we can use uniform continuity of $\lambda$.
Let $\Lambda$ be the maximum eigenvalue of $\UUU^u(v)$ taken over all $v\in T^1M$, as in \S \ref{sec:lambda-def}.
Given $\eta>0$, let $\delta=\delta(\eta)>0$ be small enough so that for  any $v,w\in T^1M$,
\begin{equation}\label{eqn:delta0}
\dK(v,w) < \delta e^{\Lambda} \Rightarrow |\lambda(v) - \lambda(w)| \leq \frac\eta2.
\end{equation}  
In particular, by \eqref{eqn:dtdcs} and \eqref{eqn:dKdu},
this applies if $w\in W_\delta^s(v)$ or $w\in W_\delta^u(v)$.
Define $\tl\colon T^1 M \to [0,\infty)$ by $\tl(v) = \max(0, \lambda(v) - \frac\eta2)$, and observe that 
\begin{equation}\label{eqn:delta}
\lambda(w) \geq \tl(v) \text{ for every } v,w\in T^1M \text{ with } \dK(v,w) < \delta.
\end{equation}
In particular, if $w\in B_T(v,\delta)$ then
\begin{equation}\label{eqn:delta-2}
\int_0^T \lambda(f_t w)\,dt \geq \int_0^T \tl(f_t v)\,dt \geq \int_0^T \lambda(f_t v)\,dt - \frac\eta2T.
\end{equation}

\begin{lem}\label{lem:integrate}
Given $\eta >0$ and $\delta= \delta(\eta)$ such that \eqref{eqn:delta0} holds, $v\in T^1M$, and $w,w'\in W_\delta^s(v)$, we have the following for every $t\geq 0$:
\begin{equation}\label{eqn:Ws-contract-1}
d^s(f_t w, f_t w') \leq d^s(w,w') e^{-\int_0^t \tl(f_\tau v)\,d\tau}.
\end{equation}
Similarly, if $w,w'\in W_\delta^u(v)$, then for every $t\geq 0$ we have
\begin{equation}\label{eqn:Wu-contract-1}
d^u(f_{-t} w, f_{-t} w') \leq d^u(w,w') e^{-\int_0^t \tl(f_{-\tau} v)\,d\tau}.
\end{equation}
\end{lem}
\begin{proof}
We prove \eqref{eqn:Ws-contract-1}; \eqref{eqn:Wu-contract-1} is similar.
Recalling the definition of $d^s$ in \eqref{eqn:ds}, let $\zeta \colon [0,1] \to W^s_\delta(v)$ be a curve that connects $w$ and $w'$. It follows from Lemma \ref{lem:monotonic-leaf} that the curve $f_t \zeta$ lies in $W_\delta^s(f_tv)$. 
We want to compare the lengths $\ell(\pi \zeta)$ and $\ell(\pi f_{t} \zeta)$.  For each $r\in [0,1]$, the vector $\zeta(r) \in T^1 M$ determines a geodesic $\gamma_r$ that is normal to the stable horosphere $\pi W_\delta^s(v)$; this one-parameter family of geodesics gives a family of stable Jacobi fields $J_r \in \JJJ^s(\gamma_r)$.  By Lemma \ref{l.lambda_rate}, these satisfy
\[
\|J_r(t)\| \leq e^{-\int_0^t \lambda(\dot\gamma_r(\tau))\,d\tau} \|J_r(0)\|
\leq e^{-\int_0^t \tl(f_\tau v)\,d\tau} \|J_r(0)\|.
\]
Integrating over $r\in [0,1]$ gives
$
\ell(\pi f_t\zeta) 
\leq e^{-\int_0^t \tl(f_\tau v)\,d\tau} \ell(\pi\zeta).
$
By \eqref{eqn:ds}, taking an infimum over all such $\zeta$ gives \eqref{eqn:Ws-contract-1}.
\end{proof}

When $(v,T)\in \GGG(\eta)$, 
the following  is an immediate corollary of \eqref{eqn:delta}, \eqref{eqn:delta-2}, and Lemma \ref{lem:integrate}.

\begin{cor}\label{lem:Geta}
Given $\eta >0$ and $\delta= \delta(\eta)$  such that \eqref{eqn:delta0} holds and $(v,T)\in \GGG(\eta)$, every $w\in B_T(v,\delta)$ has $(w,T)\in \GGG(\frac\eta2)$.  Moreover,  
for every $w,w'\in W_\delta^s(v)$ and $0\leq t\leq T$ we have
\begin{equation}\label{eqn:Ws-contract}
d^s(f_t w, f_t w') \leq d^s(w,w') e^{-\frac{\eta}{2} t},
\end{equation}
and for every $w,w'\in f_{-T}W_\delta^u(f_{T}v)$ and $0\leq t\leq T$, we have
\begin{equation}\label{eqn:Wu-contract}
d^u(f_t w, f_t w') \leq d^u(f_{T}w,f_{T}w') e^{-\frac{\eta}{2}(T - t)}.
\end{equation}
\end{cor}

For $v$ with $\lambda(v)$ uniformly positive, uniform continuity of $\lambda$ gives a small but definite amount of contraction/expansion in a neighborhood of $v$, and we obtain the following consequence of Lemma \ref{lem:integrate}.

\begin{cor}\label{hyplem}
Given $\eta >0$
 and $\delta= \delta(\eta)$ such that \eqref{eqn:delta0} holds, for every $t_0>0$ there exists $\alpha \in (0,1)$ such that if 
$v\in \Reg(2\eta)$ and $v'\in B(v,\delta)$, we have
\begin{equation}\label{eqn:contracts}
d^s(f_t w, f_tw') \leq \alpha d^s(w,w') \text{ for all } 
w,w'\in W_\delta^s(v') \text{ and } t\geq t_0,
\end{equation}
and similarly
\begin{equation}\label{eqn:expands}
d^u(f_{-t} w, f_{-t}w') \leq \alpha d^u(w,w') \text{ for all } 
w,w'\in W_\delta^u(v') \text{ and } t\geq t_0.
\end{equation}
\end{cor}
\begin{proof}  
Let $t_1>0$ be such that $\dK(v,f_\tau v) <\delta$ for all $v\in T^1M$ and $|\tau|<t_1$. Then for $v'\in B(v,\delta)$ and $|\tau|<t_1$, $f_\tau v' \in \Reg(\eta)$  since by \eqref{eqn:delta0}, $|\lambda(v)-\lambda(f_tv')| \leq |\lambda(v)-\lambda(v')|+ |\lambda(v')-\lambda(f_tv')| \leq \eta$. Thus, 
both $\int_0^t \tl(f_\tau v')\,d\tau$ and $\int_0^t \tl(f_{-\tau} v')\,d\tau$ are bounded below by $\min\{t_1, t\} \frac\eta2$. Applying Lemma \ref{lem:integrate} proves \eqref{eqn:contracts} and \eqref{eqn:expands} with $\alpha = e^{-\min\{t_0, t_1\} \frac \eta 2}$.
\end{proof}

We establish a uniform growth property for local stable and unstable manifolds of vectors in a compact subset of the regular set.

\begin{prop}\label{prop:weak-expansivity}
For every $\epsilon,R > 0$ and every compact $X\subset \Reg$,
there exists $T > 0$ such that for every $w \in X$ and $t \geq T$, we have $f_t W^u_\epsilon(w) \supset W^u_R(f_tw)$ and $f_{-t} W^s_\epsilon(w) \supset W^s_R(f_{-t}w)$.
\end{prop}
\begin{proof}
We prove the first assertion; the second is similar. 
Consider the compact set
\[
Z = \{(v,w) \in T^1M \times T^1M \mid w \in X, v \in W^u(w), d^u(v,w) = \epsilon\}.
\]
Since $f_t$ is a homeomorphism between $W^u(w)$ and $W^u(f_t w)$ for each $t>0$, it suffices to exhibit $T$ such that for any $(v,w)\in Z$, we have $d^u(f_tv, f_tw) \geq R$ for all $t\geq T$.  For $t>0$, define $\Delta_t \colon Z\to (0,\infty)$ by $\Delta_t(v,w) = d^u(f_tv, f_t w)$.  Then each $\Delta_t$ is continuous, and $t\mapsto \Delta_t(v,w)$ is nondecreasing for each $(v,w)\in Z$ by Lemma \ref{lem:monotonic-leaf}.  If $\lim_{t\to\infty} \Delta_t(v,w) < \infty$, then $d^u(f_tv,f_tw)$ is bounded for all $t\in \RR$. The Flat Strip Theorem \cite[Proposition 1.11.4]{Ebe:96}
implies that $\gamma_v$ and $\gamma_w$ bound a flat strip, which cannot happen since $w\in \Reg$.  It follows that the functions $\Delta_t$ converge pointwise to $\infty$ monotonically on the compact domain $Z$. By Dini's theorem, this convergence must be uniform.  In particular, there exists $T>0$ such that $\Delta_t(v,w) \geq R$ for all $(v,w)\in Z$ and $t\geq T$, which completes the proof.
\end{proof}

Since $\Reg(\eta)$ is compact, the following corollary is immediate.
\begin{cor}\label{cor:weak-expansivity}
For every $\eta,\epsilon,R > 0$,
there exists $T > 0$ such that for every $w \in \Reg(\eta)$ and $t \geq T$, we have $f_t W^u_\epsilon(w) \supset W^u_R(f_tw)$ and $f_{-t} W^s_\epsilon(w) \supset W^s_R(f_{-t}w)$.
\end{cor}

\section{The specification property and closing lemma} \label{s.spec}

The following result verifies condition \ref{cond:spec} from Theorem \ref{t.abstract} by proving the specification property for a collection of orbit segments that contains $\GGG(\eta)$. We prove a stronger specification property than needed for Theorem \ref{t.abstract}. The stronger version, which is discussed after Definition \ref{def:spec}, is required for our equidistribution results in \S \ref{s.geodgeneral}.

\begin{thm} \label{specgeoflow} 
 Let $\FFF$ be the geodesic flow on a closed rank $1$ manifold $M$, and let $\eta>0$. Let $\CCC(\eta)$ be the set of orbit segments for $(T^1M, \FFF)$ that both start and end in $\Reg(\eta)$, that is:
\[
\CCC(\eta) = \{ (x, t) \in T^1M \times (0, \infty) : x \in \Reg(\eta) \text{ and } f_tx \in \Reg(\eta) \}.
\]
Then for every $\rho>0$, there exists $T>0$ such that given $(v_1,t_1),\dots, (v_k,t_k) \in \CCC(\eta)$ and $T_1,\dots, T_k \in \RR$ with $T_{j+1} - T_j \geq t_j + T$ for all $1\leq j < k$, there is $w\in T^1M$ such that for all $1\leq j\leq k$, we have $f_{T_j}w \in B_{t_j}(v_j,\rho)$. 
\end{thm}
We remark that we need the specification property for $\CCC(\eta)$, rather than $\GGG(\eta)$, in our proof of the pressure gap in \S \ref{s.entropygap}.
The proof of Theorem \ref{specgeoflow} is based on uniformity of the local product structure for the foliations $W^u$, $W^{cs}$ at the endpoints of orbits in $\CCC(\eta)$.
To make this idea precise, we define local product structure at a point for  a fixed scale and distortion constant. 
We work with the Knieper metric $\dK$ from \eqref{eqn:dK} and the leafwise metrics $d^s$ and $d^u$ from \eqref{eqn:ds}.  Throughout, $B(v,\delta)$ denotes the ball in the Knieper metric $\dK$.
\begin{defn}\label{def:lps}
The foliations $W^u$, $W^{cs}$ have local product structure $($LPS$)$  with constant $\kappa\geq 1$ in a $\delta$-neighborhood of $v\in T^1M$ if for every $\eps\in (0,\delta]$ and all $w_1,w_2 \in B(v,\eps)$, the intersection $W_{\kappa \eps}^u(w_1) \cap W_{\kappa \eps}^{cs}(w_2)$ contains a single point, which we denote by $[w_1,w_2]$, and if moreover we have
\begin{align*}
d^u(w_1, [w_1,w_2]) &\leq \kappa \dK(w_1,w_2), \\
d^{cs}(w_2, [w_1,w_2]) &\leq \kappa \dK(w_1,w_2).
\end{align*}
\end{defn}
It is clear that if $W^u$, $W^{cs}$ have LPS with constant $\kappa$ in a $\delta$-neighborhood of $v\in T^1M$, then for every $\eps \in (0,\delta]$, they have LPS with constant $\kappa$ in an $\eps$-neighborhood of $v$. We also have the following elementary lemma.
\begin{lem}\label{lem:lps}
 Let $w\in B(v,\delta/2)$ and suppose that $W^u$, $W^{cs}$ have LPS with constant $\kappa$ in a $\delta$-neighborhood of $v\in T^1M$. Then $W^u$, $W^{cs}$ have LPS with constant $2\kappa$ in a $\delta/2$-neighborhood  of $w$.
\end{lem}
\begin{proof}
Observe that if $\eps\in (0,\delta/2]$ and $w\in B(v,\delta/2)$, then for every $w_1,w_2\in B(w,\eps) \subset B(v,\delta)$, the hypothesis gives a unique point $[w_1,w_2] \in W_{\kappa\delta}^u(w_1) \cap W_{\kappa\delta}^{cs}(w_1)$; moreover, $d^u(w_1,[w_1,w_2]) \leq \kappa \dK(w_1,w_2) \leq \kappa(2\eps)$, with a similar bound on $d^{cs}(w_2,[w_1,w_2])$. Thus, $[w_1,w_2] \in W_{2\kappa\eps}^u(w_1) \cap W_{2\kappa\eps}^{cs}(w_2)$.
\end{proof}

\begin{lem}\label{lem:lps2}
For every $\eta>0$, there exist $\delta>0$ and $\kappa\geq 1$ such that at every $v\in \Reg(\eta)$, the foliations $W^u, W^{cs}$ have LPS with constant $\kappa$ in a $\delta$-neighborhood of $v$. 
\end{lem}
\begin{proof}
Lemma \ref{l.angle} gives a uniform lower bound on $\measuredangle(E^u(v),E^s(v))$ for $v \in \Reg(\eta)$. Since $E^c$ is orthogonal to $E^u$, this also gives a uniform lower bound on the angle between $E^u$ and $E^{cs}$.  Since horospheres are uniformly $C^2$ \cite{HI77}, the leaves $W^u,W^{cs}$ are uniformly $C^1$, and the result follows.
\end{proof}
We can define local product structure for the foliations $W^{cu}, W^s$ analogously to Definition \ref{def:lps}, and the statement of Lemma \ref{lem:lps2} is obtained analogously for the foliations $W^{cu}, W^s$.

\begin{prop} \label{prop:denseleaf}
Given $\eta>0$, there exists $\delta>0$ such that for any $\rho\in (0,\delta]$, there exists $T>0$ such that for every $v,w\in \Reg(\eta)$ and $v'\in B(v,\delta)$, $w'\in B(w,\delta)$, we have $f_t(W_\rho^u(v')) \cap W_\rho^{cs}(w') \neq\emptyset$ for every $t\geq T$.
\end{prop}
\begin{proof}
By Lemma \ref{lem:lps2}, there exist $\delta>0$ and $\kappa \geq 1$ such that at every $v\in \Reg(\eta)$, the foliations $W^u, W^{cs}$  have LPS with constant $\kappa$ in a $2\delta$-neighborhood of $v$.
 By Lemma \ref{lem:lps}, the foliations $W^u, W^{cs}$  have LPS with constant $2\kappa$ in a $\rho$-neighborhood of $v'$ and $w'$ for any $\rho \in(0, \delta]$. By Lemma \ref{lem:unif-dense}, there exists $R>0$ such that $W_{R-\rho}^u(f_t v')$ is $\rho/(2\kappa)$-dense in $T^1M$ for every $v'\in T^1M$ and $t\in \RR$, so by the local product structure, we have $W_R^u(f_t v') \cap W_\rho^{cs}(w') \neq \emptyset$.  By Corollary \ref{cor:weak-expansivity}, there exists $T>0$ such that $f_t(W_\rho^u(v')) \supset W_R^u(f_t v')$ for every $v\in \Reg(\eta)$, $v'\in B(v,\delta)$, and $t\geq T$, which completes the proof.
 \end{proof}

In particular, if $(v, s), (w, t)  \in \CCC(\eta)$, then Proposition \ref{prop:denseleaf} applies at $f_sv$ and $w$. We are now ready to prove the specification property on $\CCC(\eta)$.


\begin{proof}[Proof of Theorem \ref{specgeoflow}]
Fix $\eta,t_0>0$.  By Corollary \ref{hyplem} there are $\eps>0$ and $\alpha\in (0,1)$ such that whenever $v$ is within $\eps$ of $\Reg(\eta)$, for every  $w,w'\in W_\delta^u(v)$, and $t\geq t_0$, we have 
\begin{equation}\label{eqn:u-contracts}
d^u(f_{-t} w, f_{-t} w') \leq \alpha d^u(w,w').
\end{equation}
Let $\delta>0$ be as in Proposition \ref{prop:denseleaf} and fix $0 < \rho < \min(\delta,\eps)$.  Let
\begin{equation}\label{eqn:rho'}
\rho' = \rho e^{-\Lambda} (1-\alpha)/2.
\end{equation}
By Proposition \ref{prop:denseleaf}, there exists $T\geq t_0$ such that $f_t(W_{\rho'}^u(v))$ intersects $W_{\rho'}^{cs}(w)$ whenever $t\geq T$ and  $v,w$ are within $\delta$ of $\Reg(\eta)$.  

Given any
$(v_1, t_1), \dots, (v_k, t_k)\in \CCC(\eta)$ and $T_1,\dots, T_k \in \RR$ with $T_{j+1} - T_j \geq t_j + T$, we construct points $w_j$ iteratively 
such that $f_{T_i}w \in B_{t_i}(v_i,\rho)$ for all $i \in\{1, \ldots, j\}$. Then $w_k$ will satisfy the conclusion of Theorem \ref{specgeoflow}.  Without loss of generality we assume that $T_1=0$.

Start by letting $w_1 = v_1$ and $s_1 = t_1$. Since $T_2 = T_2 - T_1 \geq t_1 + T = s_1 + T$, we can apply Proposition \ref{prop:denseleaf} at $f_{s_1}w_1\in \Reg(\eta)$ and $v_2\in \Reg(\eta)$ to obtain
\[
(f_{T_2 - s_1} (W_{\rho'}^u(f_{s_1}w_1))) \cap W_{\rho'}^{cs}(v_2) \neq 0.
\]
In particular, there exists $w_2\in W^u(v_1)$ such that
\[
f_{s_1}w_2 \in W_{\rho'}^u(f_{s_1}w_1)
\quad\text{ and }\quad
f_{T_2}w_2 \in W_{\rho'}^{cs}(v_2).
\]
We iterate this procedure to obtain a sequence of points $w_j \in W^u(v_1)$ such that writing $s_j = T_j + t_j$, we have
\begin{align}
\label{eqn:zj}
f_{s_j}w_{j+1} &\in W_{\rho'}^u(f_{s_j}(w_j))
&&\text{and} &
f_{T_{j+1}}w_{j+1}&\in W_{\rho'}^{cs}(v_{j+1}).
\end{align}
To guarantee that such points and times exist for all $1\leq j\leq k$, we observe that once $w_j$ is chosen, we have 
\[
f_{s_j}w_j  = f_{t_j} (f_{T_j}w_j) \in f_{t_j} W_{\rho'}^{cs}(v_j) \subset W_{\rho'}^{cs}(f_{t_j} v_j).
\]
 Since $\rho' < \delta$, this gives $\dK(f_{s_j}w_j, f_{t_j} v_j) < \delta$. Using the fact that $T_{j+1} - s_j = T_{j+1} - T_j - t_j \geq T$, Proposition \ref{prop:denseleaf} provides $w_{j+1}$ satisfying \eqref{eqn:zj}.

We prove by induction in $j-i$ that given any $1\leq i \leq  j < k$, we have
\begin{equation}\label{eqn:du-ij}
d^u(f_{s_i} w_{j+1}, f_{s_i} w_j) \leq \alpha^{j-i} \rho'.
\end{equation}
Note that the base case $i=j$ follows immediately from \eqref{eqn:zj}.  For the inductive step we need the following lemma.

\begin{lem}\label{lem:get-near-vi}
Suppose $\ell \in \{0,1,\dots, k\}$ has the property that \eqref{eqn:du-ij} holds for all $1\leq i \leq j < k$ with $j-i \leq \ell$.  Then for any $1\leq i\leq j \leq k$ with $j-i\leq \ell$, we have
$d_{t_i}(f_{T_i} w_j, v_i) < \rho < \eps$.
\end{lem}
\begin{proof}
Using the triangle inequality and the fact that $d^u$ and $d^{cs}$ are respectively nondecreasing and nonincreasing, we get
\begin{equation}\label{eqn:dti}
d_{t_i}(f_{T_i} w_j, v_i)
\leq d^{cs}(f_{T_i} w_i, v_i) + e^\Lambda d^u(f_{s_i} w_j, f_{s_i} w_i).
\end{equation}
By the second half of \eqref{eqn:zj}, we have 
\begin{equation}\label{eqn:dcs}
d^{cs}(f_{T_i} w_i, v_i) \leq \rho'.
\end{equation}
Using the hypothesis on \eqref{eqn:du-ij}, we have
\[
d^u(f_{s_i} w_j, f_{s_i} w_i) \leq \sum_{m=i}^{j-1} d^u(f_{s_i} w_{m+1}, f_{s_i} w_m) \leq \sum_{m=i}^{j-1} \alpha^{m-i} \rho'
\leq \rho'(1-\alpha)^{-1}.
\]
Together with \eqref{eqn:dti}, \eqref{eqn:dcs}, and the definition of $\rho'$ in \eqref{eqn:rho'}, this gives
\[
d_{t_i}(f_{T_i} w_j,v_i) \leq \rho' + e^{\Lambda} \rho'(1-\alpha)^{-1} < \rho/2 + \rho/2 = \rho,
\]
which proves Lemma \ref{lem:get-near-vi}.
\end{proof}

Now suppose that $\ell \in \{0,1,\dots, k-1\}$ is such that \eqref{eqn:du-ij} holds for all $1\leq i\leq j < k$ with $j-i \leq \ell$.  Fix $1\leq i\leq j < k$ with $j-i = \ell+1$.  It follows from Lemma \ref{lem:get-near-vi} that $f_{s_{i+1}} w_j \in B(f_{t_{i+1}} v_{i+1},\eps)$.  Moreover, since $\alpha^\ell \rho' \leq \rho' < \rho < \delta$, \eqref{eqn:du-ij} gives $f_{s_{i+1}} w_{j+1} \in W_\delta^u(f_{s_{i+1}} w_j)$.
Observing that $f_{t_{i+1}} v_{i+1} \in \Reg(\eta)$ and $s_{i+1} - s_i \geq T \geq t_0$, we can apply \eqref{eqn:u-contracts} to obtain
\begin{equation}\label{eqn:ij-alpha}
d^u(f_{s_i} w_{j+1}, f_{s_i} w_j) \leq \alpha
d^u(f_{s_{i+1}} w_{j+1}, f_{s_{i+1}} w_j)
\leq \alpha^{j-i} \rho'.
\end{equation}
This completes the inductive step and shows that \eqref{eqn:du-ij} holds for all $1\leq i\leq j < k$.  Applying Lemma \ref{lem:get-near-vi} with $j=k$ proves that $w_k$ satisfies the conclusion of Theorem \ref{specgeoflow}. Since $\rho$ can be taken arbitrarily small, this completes the proof.
\end{proof}

We now prove a closing lemma for orbit segments in $\CCC(\eta)$. 

\begin{lem} \label{closing}
For all $\rho,\eta,\wiggle>0$, there exists $T$ so that for every $(v,t)\in \CCC(\eta)$, there exists $w\in B_t(v,\rho)$ and $\tau\in [T-\wiggle,T+\wiggle]$ such that $f_{t+\tau} w = w$.
\end{lem}

\begin{proof}
We follow the proof of the Anosov closing lemma based on the Brouwer fixed point theorem. By Lemma \ref{lem:lps2}, there exists $\delta_0>0$ and $\kappa \geq 1$ such that for $(v,t)\in \CCC(\eta)$, the foliations $W^u,W^{cs}$ have local product structure with constant $\kappa$ in a $\delta_0$-neighborhood of both $v$ and $f_tv$, and so do the foliations $W^s$, $W^{cu}$.

Fix $(v_0,t_0) \in \CCC(\eta)$.  By Corollary \ref{hyplem}, there exists $\eps>0$ and $\alpha\in (0,1)$ such that if $v'\in B(v_0,\eps) \cup B(f_{t_0} v_0,\eps)$, then for all $w,w'\in W_\eps^s(v')$ and $t\geq t_0$, we have
\begin{equation}\label{eqn:per-contracts}
d^s(f_t w, f_t w') \leq \alpha d^s(w,w'),
\end{equation}
and for all $w,w' \in W_\eps^u(v')$ and $t\geq t_0$, we have
\begin{equation}\label{eqn:per-expands}
d^u(f_{-t} w, f_{-t} w') \leq \alpha d^u(w,w').
\end{equation}
Let $\delta>0$ be as in Proposition \ref{prop:denseleaf}.  Without loss of generality, we assume that $\rho \leq \min(\eps, \delta_0, \delta,\wiggle)$.

By Theorem \ref{specgeoflow}, $\CCC(\eta)$ has specification at scale $\rho/(16\kappa)$; let $T_0$ be the transition time.  Let $n\in \mathbb{N}$ satisfy $2\kappa \alpha^n <1$.  By the specification property, there is a point $w_0$ whose forward orbit $\rho/(16\kappa)$-shadows first $(v,t)$, then $(v_0,t_0)$, then $(v_0,t_0)$ again, and so on until $(v_0,t_0)$ has been shadowed $n$ times, and then finally shadows $(v,t)$ once more.  In particular, we have $d_t(v,w_0) < \rho/(16\kappa)$, and $f_{t+T}(w_0) \in B(r,\rho/(16\kappa))$ for $T=n (t_0 + T_0) + T_0$, so $\dK(w_0,f_{t+T}w_0) < \rho/(8\kappa)$.  Assume that $n$ is chosen large enough that $T > 1+ \sigma$.


Fix $w_0$ and consider the map $W_{\rho/4}^{s}(w_0) \to W_{\rho/4}^s(w_0)$ defined by $u\mapsto W_{\rho/4}^{s}(w_0) \cap W_{\rho/4}^{cu}(f_{t+T}u)$.  This is well-defined because for all $u\in W_{\rho/4}^s(w_0)$, we can apply \eqref{eqn:per-contracts} at each of the $n$ times the orbit segment shadows $(v_0,t_0)$ to obtain $d^s(f_{t+T}u,f_{t+T}w_0) \leq \alpha^n d^s(u,w_0) \leq \alpha^n \rho < \rho/(8\kappa)$. Since $\dK \leq d^s$, this gives
\[
\dK(f_{t+T}u,w_0) \leq 
d^s(f_{t+T}u, f_{t+T}w_0) + \dK(f_{t+T}w_0, w_0) 
\leq
\rho/(4\kappa),
\]
and thus by the local product structure the map is well-defined.
By continuity of the map, the Brouwer fixed point theorem gives $w_1 \in W_{\rho/4}^s(w_0)$ with $w_1\in W_{\rho/4}^{cu}(f_{t+T}w_1)$, and thus  $w_1\in W_{\rho/4}^u(f_{t+T+r}w_1)$ for some $|r|<\rho\leq \wiggle$.  
Writing $\tau = T+r \in [T-\wiggle,T+\wiggle]$, we observe that 
\begin{equation}\label{eqn:w1-iterated}
W_{\rho/4}^u(w_1) \subset W_{\rho/2}^u(f_{t+\tau} w_1),
\end{equation}
and so given any $u\in W_{\rho/4}^u(w_1)$,  repeated application of \eqref{eqn:per-expands} gives
\[
f_{-(t+\tau)}(u) \in W_{\alpha^n \rho/2}^u(w_1) \subset W^u_{\rho/4}(w_1),
\]
so $f_{-(t+\tau)}$ sends $W_{\rho/4}^u(w_1)$ to itself continuously. The Brouwer fixed point theorem gives  $w = f_{t+\tau} w \in W_{\rho/4}^u(w_1)$.   Since the period of this orbit is in the required range, it only remains to check that $d_t(v,w) < \rho$, which we do by using the triangle inequality, \eqref{eqn:dtdcs}, \eqref{eqn:dKdu}, $\tau > 1$, and \eqref{eqn:w1-iterated} to get
\begin{align*}
d_t(v,w) &\leq d_t(v,w_0) + d_t(w_0,w_1) + d_t(w_1,w) \\
&\leq \frac{\rho}{16\kappa} + d^s(w_0,w_1) + d^u(f_{t+\tau} w_1, f_{t+\tau} w) \\
&\leq \frac{\rho}{16} + \frac{\rho}4 + d^u(f_{t+\tau} w_1, w)
\leq \frac{\rho}{16} + \frac{\rho}4 + \frac{\rho}2 < \rho.\qedhere
\end{align*}
\end{proof}

We remark that combining Theorem \ref{specgeoflow} and Lemma \ref{closing} yields a version of the specification property for $\CCC(\eta)$ where we can approximate a finite sequence of orbit segments by a \emph{periodic} orbit segment. We also have the following corollary for orbit segments $(v,t)$ that enter $\CCC(\eta)$ in both the time intervals $[0, M]$ and $[t-M, t]$. 

\begin{cor} \label{closingGM}
For all $\rho,\eta,\wiggle, M>0$, there exists $T$ so that for every $(v,t)$ with the property that there exists $p, s \in [0, M]$ such that $(f_pv, f_{t-p-s}v) \in \CCC(\eta)$, there exists $w\in B_t(v,\rho)$ and $\tau\in [T-\wiggle,T+\wiggle]$ such that $f_{t+\tau} w = w$ and $w \in \Reg$.

\end{cor}
\begin{proof}
Given $\rho>0$, by continuity of the flow and $\lambda$ and compactness of $T^1M$, there exists $\rho'>0$ such that
$\dK(v, w)< \rho'$ implies that $\dK(f_t v, f_tw)< \rho$ for every $t \in [-M, M]$, and if $d_K(v, w)< \rho'$ and $\lambda(v)> \eta$ then $\lambda (w)> 0$. Let $(v,t)$ and $p, s \leq M$ satisfy $(v', t') =(f_pv, t-p-s) \in \CCC(\eta)$.  Fix $\delta>0$. By Lemma \ref{closing}, we know that there exists $w'$ with $f_{t'+\tau}w'=w'$, where $\tau \in[T(\rho')-\delta, T(\rho')+\delta]$, and $d_{t'}(v', w')< \rho'$. Since $\lambda(v)> \eta$ and $d_K(v', w')< \rho'$, we have $\lambda(w')>0$ and thus $w'$ is a regular periodic point. 
\end{proof}

\section{Pressure estimates} \label{s.pressure}
In this section, we prove that the hypotheses of Theorem \ref{t.multiples} guarantee the pressure gap conditions $P^{\perp}_{\mathrm{exp}}(\varphi)<P(\varphi)$ and $P([\PPP] \cup [\SSS],\ph)<P(\ph)$ in Theorem \ref{t.abstract}.
Here, $\PPP=\SSS = \BBB(\eta)$ where $\BBB(\eta)$ is the collection introduced in \S \ref{s.decomp} and $\eta>0$ is sufficiently small.

\subsection{General estimates} 

We start with a general result for a continuous flow $\FFF$ on a compact metric space $X$,  which relates pressure for a collection of orbit segments to the free energies for an associated collection of measures. 
Given a collection $\CCC$ of orbit segments, let $\MMM(\CCC)$ denote the set of $\FFF$-invariant measures on $X$ that are obtained as limits of convex combinations of empirical measures along orbit segments in $\CCC$.  That is, for each $(x,t)\in \CCC$ define the empirical measure $\mathcal{E}_{x,t}$ by
\[
\int \psi\,d\mathcal{E}_{x,t} = \frac 1t\int_0^t \psi(f_s x)\,ds,
\]
for all $\psi\in C(X)$.
Consider for each $t\geq 0$ the convex hull
\[
\MMM_t(\CCC) = \Bigg\{ \sum_{i=1}^k a_i \mathcal{E}_{x_i, t} : a_i \geq 0,\ \sum a_i = 1,\ (x_i,t) \in \CCC \Bigg\}.
\]
We use the following set of $\FFF$-invariant Borel probability measures:
\begin{equation}\label{eqn:MC}
\MMM(\CCC) 
= \Big\{\lim_{k\to\infty} \mu_{t_k} : t_k\to\infty, \mu_{t_k}\in \MMM_{t_k}(\CCC) \Big\}.
\end{equation}
Note that $\MMM(\CCC)$ is non-empty as long as $\CCC$ contains arbitrarily long orbit segments (which happens whenever $P(\CCC,\ph) > - \infty$).

\begin{prop} \label{p.pressureestimate}
If $\varphi$ is a continuous potential, then
\[
P(\mathcal{C}, \varphi)\leq \sup_{\mu\in \mathcal{M}(\mathcal{C})} P_\mu(\varphi),
\]
where
we write $P_\mu(\ph) = h_\mu(\FFF) + \int\ph\,d\mu$
for convenience.
\end{prop}
\begin{proof}   
For an arbitrary fixed $\epsilon>0$, and any $t>0$,  let $E_t$ be a $(t,\epsilon)$-separated set for $\CCC_t$ of maximal cardinality with 
\[
\log \sum_{y\in E_t} e^{\Phi(y, t)}> \log \Lambda(\mathcal{C}, \varphi, \epsilon, t)-1.
\]
Then there is $t_k\to\infty$ such that
\begin{equation}\label{eqn:Etk}
\lim_{k\to\infty} \frac 1{t_k} \log \sum_{y\in E_{t_k}} e^{\Phi(y,t_k)}
\geq 
P(\CCC,\ph,\eps).
\end{equation}
Consider the measures
\[
\mu_t =
\frac{\sum_{y\in E_t} e^{\Phi(y, t)}\mathcal{E}_{y,t}}{\sum_{y\in E_t}e^{\Phi(y, t)}}.
\]
By construction, $\mu_t \in \MMM_t(\CCC)$.
Passing to a subsequence if necessary, we can assume that $\mu_{t_k} \to \mu\in \MMM(\CCC)$.
The second half of the proof of the variational principle \cite[Theorem 9.10]{Wa} shows that 
\[
h_\mu + \int \varphi \,d \mu \geq \liminf_{k \to \infty}\frac{1}{t_k} \log \sum_{y\in E_{t_k}} e^{\Phi(y,t_k)},
\] 
so \eqref{eqn:Etk} gives $P_\mu(\ph) \geq P(\CCC,\ph,\eps)$.
Taking $\eps>0$ arbitrarily small gives the required result.
\end{proof}

In general, the inequality in Proposition \ref{p.pressureestimate} may be strict. For example, let $\nu$ be an ergodic measure with positive entropy, $x$ be a generic point for $\nu$ in the sense that $\lim_{t\to\infty} \mathcal{E}_{x,t} = \nu$, and $\mathcal{C} = \{(x,t) : t>0\}$.  We have $P(\CCC,\ph) = \int\ph\,d\nu < h_\mu + \int\ph\,d\nu = P_\nu(\ph)= \sup_{\mu\in \MMM(\CCC)} P_\mu(\ph)$.

\subsection{Pressure estimates for bad orbits}\label{sec:pressure-estimates-for-bad-orbits}

Now we consider the geodesic flow and  estimate the pressure of the `bad' orbit segments.

\begin{prop}\label{p.pressure}
With $\BBB(\eta)$ as in \S\ref{s.decomp} and $\ph\colon T^1M\to\RR$ continuous, we have
$\lim_{\eta\to 0} P([\BBB(\eta)],\ph) = P(\Sing,\ph)$.  In particular, if $P(\mathrm{Sing}, \varphi)< P(\varphi)$, then there exists some $\eta>0$ such that $P([\BBB(\eta)], \varphi)<P(\varphi)$.
\end{prop}
\begin{proof}
Since the function $\lambda$ vanishes on $\Sing$, we have $\Sing\times \mathbb{N} \subset [\BBB(\eta)]$ for all $\eta>0$, which immediately gives $P(\Sing,\ph) \leq P([\BBB(\eta)],\ph)$. 
Thus it suffices to show that for every $\eps>0$ we have $P([\BBB(\eta)],\ph) < P(\Sing,\ph)+\eps$ whenever $\eta>0$ is sufficiently small.

To this end, consider for each $\eta>0$ the set of measures $\MMM_\lambda(\eta) = \{ \mu \in \MMM(T^1M) : \int\lambda\,d\mu \leq \eta\}$.  Given $(v,t)\in [\BBB(\eta)]$, we have
\[
\int_0^t \lambda(f_s v)\,ds \leq t\eta + 2\|\lambda\|,
\]
where the last term comes from the fact that we are considering $[\BBB(\eta)]$ instead of $\BBB(\eta)$.  By convexity, we have
$
\int\lambda \,d\mu_t \leq \eta + \frac 2t \|\lambda\|
$
for every $\mu_t \in \MMM_t([\BBB(\eta)])$, and thus every $\mu\in \MMM([\BBB(\eta)])$ satisfies $\int\lambda\,d\mu \leq \eta$, proving the inclusion $\MMM([\BBB(\eta)]) \subset \MMM_\lambda(\eta)$.  By Proposition \ref{p.pressureestimate} we have
\[
P([\BBB(\eta)],\ph) \leq \sup_{\mu\in \MMM([\BBB(\eta)])} P_\mu(\ph)
\leq \sup_{\mu\in\MMM_\lambda(\eta)} P_\mu(\ph),
\]
and so it suffices to show that for every $\eps>0$ this last quantity can be made smaller than $P(\Sing,\ph)+\eps$ by taking $\eta>0$ sufficiently small.

Note that $\MMM_\lambda(\eta)$ is compact in the weak* topology by continuity of $\lambda$.  Moreover, $\MMM(\Sing) \subset \MMM_\lambda(\eta)$ for all $\eta>0$, and by Lemma~\ref{c.singular}, we see that every $\mu$ with $\int\lambda\,d\mu=0$ is supported on $\Sing$, whence we conclude that
\begin{equation}\label{eqn:Msing}
\MMM(\Sing) = \bigcap_{\eta>0} \MMM_\lambda(\eta).
\end{equation}
Let $D$ be a metric on $\MMM(T^1 M)$ compatible with the weak* topology.  Since $\MMM_\lambda(\eta)$ is compact for each $\eta>0$, \eqref{eqn:Msing} gives
\[
D(\MMM_\lambda(\eta), \MMM(\Sing)) \to 0 \text{ as } \eta\to0.
\]
By \cite[Proposition 3.3]{knieper98}, $f_t$ is  $h$-expansive, so the entropy function $\mu \mapsto h(\mu)$ is upper semi-continuous, as is $\mu \mapsto P_\mu(\varphi)$. Thus, for any $\epsilon >0$, there exists $\gamma >0$ so that $D(\mu, \nu) < \gamma$ implies $P_{\mu}(\varphi) < P_{\nu}(\varphi) + \epsilon$. Choosing $\eta$ small enough so that 
$
D(\MMM_\lambda(\eta), \MMM(\mathrm{Sing})) < \gamma,
$
we obtain
\[
\sup_{\mu\in \MMM_\lambda(\eta)} P_\mu(\varphi) \leq \sup_{\mu\in \MMM(\mathrm{Sing})} P_\mu(\varphi)+ \epsilon
=P(\Sing,\ph)+\eps.
\]
Since $\eps>0$ was arbitrary, this completes the proof.
\end{proof}

\subsection{Pressure of obstructions to expansivity}\label{sec:Pexp}
We use the following lemma to prove that  $P^{\perp}_{\mathrm{exp}}(\varphi)\leq P(\sing,\ph)$. Let $\inj(M)$ be the injectivity radius of $M$.  

\begin{lem}\label{lem:ne-sing}
If $0 < \eps < \inj(M)/3$, then $\mathrm{NE}(\eps) \subset \Sing$.
\end{lem}
\begin{proof}
Given $v\in T^1M$, let $\tilde\gamma_v$ be the lift of $\gamma_v$ to the universal cover.  If $w\in \Gamma_\eps(v)$, then the choice of $\epsilon$ guarantees that $d(\tilde\gamma_v(t),\tilde\gamma_w(t)) \leq \eps$ for all $t\in \RR$.  If $\gamma_v$ and $\gamma_w$ are distinct geodesics, then
%
by the flat strip theorem \cite[Proposition 1.11.4]{Ebe:96}, they bound a flat strip in the universal cover. Thus, $v$ has a parallel Jacobi field, and hence $v\in \Sing$. Thus, $\mathrm{NE}(\eps) \subset \Sing$.  In particular, $\mu(\Sing)=1$.
\end{proof}


\begin{prop} \label{p.exp}
For a continuous potential $\ph$, $P^{\perp}_{\mathrm{exp}}(\ph) \leq P(\sing,\ph)$.
\end{prop}
\begin{proof}
By Lemma \ref{lem:ne-sing} and the Variational Principle, for any $\eps>0$,
\begin{align*}
P^\perp_{\mathrm{exp}}(\varphi, \epsilon)& =\sup_{\mu\in \mathcal{M}^e(\mathcal{F})}\left\{
h_\mu(f_1) + \int\varphi\, d\mu\, :\, \mu(\mathrm{NE}(\eps))=1\right\} \\
& \leq \sup_{\mu\in \mathcal{M} (\Sing)}\left\{
h_\mu(f_1) + \int\varphi\, d\mu\ \right \} = P(\Sing, \ph). \qedhere
\end{align*} 
\end{proof}

\section{Completing the proof of Theorem \ref{t.geodgeneral}} \label{s.geodgeneral}

Let $(\PPP, \GGG, \SSS) = (\BBB(\eta), \GGG(\eta), \BBB(\eta))$ be the decomposition described in \S \ref{s.decomp}. By Theorem~\ref{specgeoflow}, $\GGG(\eta)$ has specification for all $\eta>0$.  By Proposition~\ref{p.pressure}, 
for sufficiently small $\eta$ we have
$P([\PPP]\cup[\SSS], \varphi) = P([\BBB(\eta)], \varphi) < P(\varphi)$. By Proposition~\ref{p.exp}, $P^{\perp}_{\mathrm{exp}}(\varphi)\leq P(\sing,\ph)$. This verifies the hypotheses of Theorem~\ref{t.abstract}, and thus we conclude that $\varphi$ has a unique equilibrium state $\mu$.  

We now prove the remaining properties of $\mu$ stated in Theorem \ref{t.geodgeneral}.  Since $\mu$ is a unique equilibrium state, it must be ergodic (see \cite[Theorem 9.13]{Wa} for details), and thus either $\mu(\Sing)=0$ or $\mu(\Sing)=1$. Suppose the second case holds. Then by the variational principle, it would follow that $P(\Sing, \ph) \geq h_{\mu}(\FFF) + \int \ph\, d\mu=P(\ph)$. This contradicts the hypothesis of the theorem, and thus $\mu(\Reg)=1$. By Corollary \ref{c.hyperbolic}, it follows that $\mu$ is hyperbolic.

To prove $\mu$ is fully supported, and that weighted regular periodic orbits equidistribute to $\mu$, we recall details from \cite{CT4}. Given a decomposition $(\PPP, \GGG, \SSS)$ and $M>0$, we write $\GGG^M$ for the set of orbit segments $(x, t)$ whose decomposition satisfies $p(x,t), s(x,t) \leq M$. When the hypotheses of Theorem \ref{t.abstract} are satisfied, $\GGG^M$ has the following properties.

\begin{lem} \label{lem:ctlems}
There exists $M,C,\eps>0$ such that for all $t>0$,
\begin{equation} \label{lambdagm}
\Lambda(\GGG^M,\eps,t) > C e^{tP(\ph)} .
\end{equation}
Thus for sufficiently large $M$,  we have $P(\GGG^M, \varphi) = P(\varphi)$. We have the lower Gibbs property on $\GGG^M$: for all $\rho>0$, there exists $Q , T, M>0$ such that for every $(v, t) \in \GGG^M$ with $t \geq T$,
\[
\mu(B_t(v,\rho))\geq Q e^{-tP(\ph) + \Phi(v,t)}.
\]
As a consequence of the lower Gibbs property, if $(v, t) \in \GGG$ and $t$ is sufficiently large, then $\mu(B(v,\rho))>0$.
\end{lem}
\begin{proof}
Lemma 4.12 of \cite{CT4} shows that there exists $M,C,\eps>0$ such that $\Lambda(\GGG^M,\eps,t) > Q e^{tP(\ph)}$ for all $t>0$. For sufficiently large $M$,  it follows that  $P(\GGG^M, \varphi) = P(\varphi)$. The lower Gibbs property for $\GGG^M$ is provided by \cite[Lemma 4.16]{CT4}. That lemma has a hypothesis that $\rho> 2 \delta$, where $\delta$ is a scale at which $\GGG$ is required to have specification, and at which the pressure gap holds, see \cite[Remark 4.13]{CT4}; both of these conditions hold here for arbitrarily small $\delta$, so \cite[Lemma 4.16]{CT4} applies for all $\rho>0$.
Finally, note that $\GGG \subset \GGG^M$ for all $M$. Thus, if $(v,t)\in \GGG$ and $t\geq T$, then 
$\mu(B(v, \rho)) \geq \mu(B_t(v,\rho))\geq Q e^{-tP(\ph) + \Phi(v,t)}>0$.
\end{proof}

We also need the following consequence of Theorem \ref{specgeoflow}.

\begin{lem}\label{cor:long-good}
Given $\eta,\rho>0$, there exists $\eta_0>0$ such that for every $v\in \Reg(\eta)$ and every $T>0$, there are $t\geq T$ and $w\in B(v,\rho)$ such that $(w,t)\in \GGG(\eta_0)$.
\end{lem}
\begin{proof}
By Lemma \ref{lem:Geta}, we can decrease $\rho$ if necessary and assume that  if $(u,t)\in \GGG(\eta/2)$ and $u'\in B_t(u,\rho)$, then $u'\in \GGG(\eta/4)$.  Let $\tau$ be the transition time for the specification property for $\GGG(\eta/2)$ at scale $\rho$.
Let $v\in \Reg(\eta)$.  Then using the modulus of continuity for $\lambda$, we can find a fixed $\epsilon>0$ (independent of $v$) so that $(v,\epsilon) \in \GGG(\eta/2)$.  Fix $(u,t_0)\in \GGG(\eta/2)$, and let $k\in \mathbb{N}$ be such that $kt_0 \geq T$.  By the specification property, we can find a point $w$ that shadows $(v,\epsilon)$, and then shadows $k$ copies of $(u,t_0)$.  Then for each $j\geq 1$, $(f_{s_j+\tau_j} w, t_0)\in \GGG(\eta/4)$. Using this fact, and the definition of $\GGG$,  it is not hard to show the existence of a constant $\eta_0$ so that $(w,t)\in \GGG(\eta_0)$ where $t = kt_0 + \epsilon + \sum_{j=1}^{k-1}\tau_j$, and $\eta_0$ depends only on $\rho,\eta,\tau$.
\end{proof}

We are now ready to prove the following.
\begin{prop}\label{prop:support}
The unique equilibrium state $\mu$ provided by Theorem \ref{t.geodgeneral} is fully supported.
\end{prop}
\begin{proof}
Let $\Reg' = \{v : \lambda(v)>0\}$.  We show that $\mu(B(v,2\rho))>0$ for every $v\in \Reg'$ and $\rho>0$.
By Lemma \ref{lem:dense}, $\Reg'$ is dense, and by Lemma \ref{cor:long-good}, for every $v\in \Reg'$ and $\rho>0$ there exists $\eta_0>0$ such that for every $T>0$, there are $t\geq T$ and $w\in B(v,\rho)$ such that $(w,t)\in \GGG(\eta_0)$.  The decomposition $(\BBB(\eta_0),\GGG(\eta_0),\BBB(\eta_0))$ satisfies the conditions of Theorem \ref{t.abstract}, and so Lemma \ref{lem:ctlems} applies.  We are free to assume that $(w,t)$ is chosen with $t$ as large as we like,  so  Lemma \ref{lem:ctlems} shows that $\mu(B(v,2\rho)) \geq \mu(B(w,\rho))>0$.  
\end{proof}

We now address growth rates and equidistribution for regular closed geodesics.

\begin{prop} \label{prop:perasymp}
For all $\delta>0$, there exists $\beta>0$ so that the regular closed geodesics satisfy
\begin{equation}\label{eqn:weighted-sums}
\frac{\beta}T e^{TP(\ph)} \leq \Lambda^\ast_\Reg (\ph, T,\delta) \leq \beta^{-1} e^{TP(\ph)}
\end{equation}
for all sufficiently large $T$, where  $\Lambda^\ast_\Reg (\ph, T,\delta)$ is defined in \eqref{eqn:CTdelta}.
\end{prop}
\begin{proof}
The upper bound follows from \cite[Lemma 4.11]{CT4} and \eqref{eqn:LCL}.  To prove the lower bound, we first note that by \eqref{lambdagm}, there exists $C, M$ such that for every $t>0$ there exists a $(t,\eps)$-separated set $E_t \subset (\GGG^M)_t$ satisfying
$\sum_{v\in E_t} e^{\int_0^t \ph(f_s v)\,ds} \geq C e^{tP(\ph)}$.
Since $\ph$ has the Bowen property with respect to $\GGG$, then there exists $\rho, K>0$ such that
\begin{equation}\label{eqn:GM}
\sum_{v\in E_t} \inf_{w\in B_t(v,\rho)} e^{\int_0^t \ph(f_s w)\,ds} \geq C e^{-K} e^{tP(\ph)}.
\end{equation}
Without loss of generality, we assume $\rho < \eps/3$. We approximate each $v\in E_t$ by a regular closed geodesic using Corollary \ref{closingGM}.  That is, there exists a $T_0$ so that we can define a map $v\mapsto w$ so that
\begin{equation}\label{eqn:w-from-v}
f_{t+\tau} w =w \text{ for some } \tau \in (T_0-\delta,T_0], \quad \text{and}\quad
d_t(w,v) < \rho.
\end{equation}
Since $\rho < \eps/3$ and the set $E_t$ is $(t,\eps)$-separated, the map $v\mapsto w$ is injective on $E_t$, and its image $E'_t$ is $(t,\rho)$-separated.  Thus \eqref{eqn:GM} gives
\begin{equation}\label{eqn:E't}
\sum_{w\in E'_t} e^{\int_0^t \ph(f_s w)\,ds} \geq C e^{-K} e^{tP(\ph)}.
\end{equation}
Every $w\in E'_t$ is tangent to a regular closed geodesic $\gamma \in \Per_R(t+T_0-\delta,t+T_0]$, where $\Per_R$ is as  in \S\ref{s.periodicpressure}, and we have
\begin{equation}\label{eqn:Phi-phi}
\Phi(\gamma) \geq \int_0^t \ph(f_s w)\,ds - T_0\|\ph\|.
\end{equation}
Because $E'_t$ is $(t,\rho)$-separated, each $\gamma \in \Per_R(t+T_0-\delta,t+T_0]$ has at most $(t+T_0)/\rho$ elements of $E'_t$ tangent to it, so \eqref{eqn:E't} and \eqref{eqn:Phi-phi} give
\begin{equation}\label{nonprimegrowthrate}
\sum_{\gamma \in \Per_R(t+T_0-\delta,t+T_0]} e^{\Phi(\gamma)}
\geq \frac{\rho}{t+T_0} e^{-T_0\|\ph\|} C e^{-K} e^{tP(\ph)}.
\end{equation}
Now given any sufficiently large $T$, we can set $t=T - T_0$ and apply \eqref{nonprimegrowthrate} to get
\begin{equation}\label{eqn:nonprimes}
\sum_{\gamma \in \Per_R(T-\delta,T]} e^{\Phi(\gamma)} \geq \frac{\rho}{T} e^{-T_0(\|\ph\| + P(\ph))} C e^{-K} e^{TP(\ph)} = \frac{\beta'}T e^{TP(\ph)},
\end{equation}
where $\beta' = \rho e^{-T_0(\|\ph\|+P(\ph))} C e^{-K}$.
\end{proof}
\begin{prop}\label{prop:periodic}
The unique equilibrium measure $\mu$ is the weak$^\ast$ limit of  weighted regular periodic orbit measures.
\end{prop}
\begin{proof}
Proposition \ref{prop:perasymp} shows that $\liminf_{T\to\infty} \frac 1T \log \Lambda^\ast_\Reg (\ph, T,\delta) = P(\ph)$ for any fixed $\delta>0$. It follows that $P^\ast_{\Reg} (\ph) = P(\ph)$, so the result follows from Proposition \ref{prop:equidist}.
\end{proof}


\section{The Bowen property}\label{sec:Bowen}

We show that H\"older continuous potentials on $T^1 M$ have the Bowen property on $\GGG(\eta)$. Then 
we show that the geometric potential has the Bowen property on $\GGG(\eta)$, despite the fact that it is not known whether this potential is H\"older continuous. It is immediate from these results that any potential of the form $p \ph + q \ph^u$, where $\ph$ is H\"older and $p,q\in \RR$, has the Bowen property. 

\subsection{H\"older continuous potentials}\label{sec:Holder-Bowen}

We start by working along stable and unstable leaves, then use the local product structure.

\begin{defn}\label{def:su-Holder}
A potential $\ph\colon T^1 M \to \RR$ is \emph{H\"older along stable leaves} if there are $C,\theta,\eps>0$ such that for any $v\in T^1 M$ and $w\in W_\eps^s(v)$,  we have
$|\ph(v) - \ph(w)| \leq C d^s(v,w)^\theta$.
Similarly, $\ph$ is \emph{H\"older along unstable leaves} if there are $C,\theta,\eps>0$ such that 
$|\ph(v) - \ph(w)| \leq C d^u(v,w)^\theta$
whenever $v\in T^1M$ and $w\in W_\eps^u(v)$.
\end{defn}

By \eqref{eqn:dtdcs} and \eqref{eqn:dKdu}, which bound $\dK$ in terms of $d^u$ and $d^{s}$,
 a H\"older continuous potential is H\"older along both stable and unstable leaves. 

\begin{defn}
A potential $\ph$ 
has the \emph{Bowen property along stable leaves} with respect to $\CCC \subset T^1M \times [0,\infty)$ if there are $\delta,K>0$ such that 
\[
\sup \{ |\Phi(v,t)-\Phi(w,t)| : (v,t) \in \CCC,\ w\in W_\delta^s(v) \} \leq K.
\]
A potential $\ph$ has the \emph{Bowen property along unstable leaves} with respect to $\CCC$ if there are $\delta,K>0$ such that
\[
\sup \{ |\Phi(v,t)-\Phi(w,t)| : (v,t) \in \CCC,\ w\in f_{-t}W_\delta^u(f_tv) \} \leq K.
\]
\end{defn}

\begin{lem}\label{lem:su-Bowen}
If $\ph$ is H\"older along stable leaves (respectively unstable leaves), then it has the Bowen property along stable leaves (respectively unstable leaves) with respect to $\GGG(\eta)$ for any $\eta>0$.
\end{lem}
\begin{proof}
We give the proof for stable leaves; the unstable case is similar.  Let $\delta>0$ be as in Lemma \ref{lem:Geta}.  Let $(v, T) \in \GGG(\eta)$ and $w \in W_\delta^s(v)$. By Lemma \ref{lem:Geta} and the H\"older property along stable leaves, we have 
$|\ph(f_t v) - \ph(f_t w)| \leq C e^{-\frac{\eta}{2}\theta t}$ for each $t\in [0,T]$. Thus, we have
\[
|\Phi(v,T) - \Phi(w,T)|
\leq C \int_0^{T} e^{-\frac{\eta}{2}\theta t}\,dt
\leq C \int_0^{\infty} e^{-\frac{\eta}{2}\theta t}\,dt.
\]
This bound is independent of $v$ and $T$, which proves the lemma.
\end{proof}

\begin{lem}\label{lem:lps-Bowen}
Given $\eta>0$, suppose that $\ph\colon T^1 M \to \RR$ has the Bowen property on $\GGG(\eta/2)$ with respect to both stable and unstable leaves.  Then $\ph$ has the Bowen property on $\GGG(\eta)$.
\end{lem}
\begin{proof}
Since curvature of horospheres is uniformly bounded on $T^1 M$, 
there are $\delta_0,C>0$
such that for every $v\in T^1M$ and $w\in W^u(v)$ with $d^u(v,w) \leq \delta_0$, 
we have $d^u(v,w) \leq C\dK(v,w)$. 
Using Lemma \ref{lem:lps2}, let $\delta_1>0$ be such that for every $(v,T) \in \GGG(\eta)$, the foliations $W^u$, $W^{cs}$ have local product structure  with constant $\kappa$ in a $\delta_1$-neighborhood of both $v$ and $f_T v$.  By Lemma \ref{lem:Geta}, there exists $\delta_2>0$ so that  for  $(v,T)\in \GGG(\eta)$, every $w \in B_T(v, \delta_2)$ has $(w, T) \in \GGG(\eta/2)$.  Let $\delta_3,K>0$ be the constants associated to the Bowen property for $\phi$ with respect to $\GGG(\eta/2)$ along stable and unstable leaves, and assume without loss of generality that $\delta_3 < \delta_0$.

Now take $0 < \delta < \min(\delta_0, \delta_1,\delta_2,\delta_3/(2\kappa C))$.  Fix $(v,T)\in \GGG(\eta)$ and $w\in B_T(v,\delta)$.  By LPS, there is $v'\in W_{\delta\kappa}^{cs}(v) \cap W_{\delta\kappa}^u(w)$.  We claim that $f_T v' \in W_{\delta_3}^u(f_T w)$.  Suppose this fails; then there is $t\in [0,T]$ such that
\begin{equation}\label{eqn:d3d0}
\delta_3  < d^u(f_t v', f_t w) \leq \delta_0
\end{equation}
but since $v'\in W_{\delta\kappa}^{cs}(v)\subset B_T(v,\delta\kappa)$, we have
\[
\dK (f_t v', f_t w) \leq \dK(f_t v', f_t v) + \dK(f_t v, f_t w) \leq 2\delta\kappa,
\]
and so $d^u(f_tv',f_tw) \leq 2\delta\kappa C < \delta_3$, contradicting \eqref{eqn:d3d0}.  It follows that $v'\in f_{-T} W_{\delta_3}^u(f_Tw)$.  Let $\rho\in [-\kappa\delta,\kappa\delta]$ be such that $f_\rho(v') \in W_{\delta_3}^s(v)$; then
\begin{multline*}
|\Phi(v,T) - \Phi(w,T)| \leq |\Phi(v,T) - \Phi(f_\rho v',T)|
+ |\Phi(f_\rho v',T) - \Phi(v',T)| \\
\qquad
+ |\Phi(v',T) - \Phi(w,T)| 
\leq K + 2\kappa\delta\|\ph\| + K.\qedhere
\end{multline*}
\end{proof}

The following is an immediate consequence of Lemmas \ref{lem:su-Bowen} and \ref{lem:lps-Bowen}.

\begin{cor} \label{cor:h-b}
If $\ph$ is H\"older continuous, then it has the Bowen property with respect to $\GGG(\eta)$ for any $\eta>0$.  
\end{cor}

\subsection{The geometric potential} 

The \emph{geometric potential} for the geodesic flow is given by
\[
\vg(v)=-\lim_{t\to 0} \frac{1}{t}\log \det(df_t|_{E^u_v})
= -\frac d{dt}\Big|_{t=0} \log \det(df_t|_{E^u_v}).
\]

When $M$ has dimension $2$, the function $\ph^u$ is H\"older along unstable leaves \cite[Proposition III]{GW99}, and so the problem of proving the Bowen property for $\ph^u$ on $\GGG(\eta)$ reduces to proving it along stable leaves, where it is not known whether $\ph^u$ is H\"older.  In higher dimensions, it is not known whether $\ph^u$ is H\"older continuous on either stable or unstable leaves; an advantage of our approach is that we sidestep the question of H\"older regularity by proving the Bowen property on $\GGG$ directly.

We will find it more convenient to work with the potential function
\begin{equation}\label{eqn:psi-u}
\psi^u(v)=-\lim_{t\to 0}\frac 1 t \log \det(J^u_{v,t})
= -\frac d{dt}\Big|_{t=0} \log \det(J^u_{v,t}),
\end{equation}
where $J^u_{v,t}\colon v^{\perp}\to (f_t v)^{\perp}$  is the linear map  that takes $w\in v^{\perp}$ to the value at $t$ of the unstable Jacobi field along $\gamma_v$ that has value $w$ at $0$.
Note that for all $t$,  $(f_t v)^\perp$ is a subspace of $T_{\pi(f_tv)} M$ with norm induced by the Riemannian metric, whereas $E^u_{f_tv}$ is a subspace of $T_{f_tv} T^1M$, so the Jacobian determinants computed in  $\ph^u$ and $\psi^u$ give different values.  However, we can use \eqref{eqn:unif-comp}, which tells us that these norms are uniformly comparable, to show that the rate at which $J_{v,t}^u$ and $df_t$ expand volumes are also uniformly comparable.

\begin{lem}\label{c.equiv_pressure} 
There exists $K$ so that $|\int_0^T \ph^u(f_t v)\,dt - \int_0^T \psi^u(f_t v)\,dt| \leq K$ for all $v\in T^1M$ and $T>0$.
\end{lem}
\begin{proof}
Given $p\in M$ and $v\in T_p^1M$, the canonical projection $\pi\colon T^1 M \to M$ has derivative $d\pi_v\colon T_vT^1M \to T_pM$ that sends $E_v^u$ onto $v^\perp$. By \eqref{eqn:unif-comp}, the Sasaki norm on $E_v^u$ and the Riemannian norm on $v^\perp$ satisfy  
\[
(1+\Lambda^2)^{-1/2} \| \xi\| \leq \|d\pi_v\xi\| \leq \|\xi\| 
\text{ for all } \xi\in E_v^u.
\]
It follows that $(1+\Lambda^2)^{-n/2} \leq \det d\pi_v \leq 1$ for all $v\in T^1M$. Since $df_T = d\pi_{f_T v}^{-1} \circ J_{v,T}^u \circ d\pi_v$, we have $\det df_T = \det(d\pi_{f_T v})^{-1} \det (J_{v,T}^u) \det (d\pi_v)$, and thus
\begin{align*}
\left |\int_0^T \ph^u(f_t v)\,dt - \int_0^T \psi^u(f_t v)\,dt \right | & = \left |\log \det df_T - \log \det J_{v,T}^u \right | \\ &\leq |\log \det d\pi_{f_T v}| + |\log \det d\pi_v| \\
&\leq 2|\log (1+\Lambda^2)^{-n/2}| = n\log(1+\Lambda^2). \qedhere
\end{align*}

\end{proof}

It follows that $q\ph^u$ and $q\psi^u$ share the same equilibrium states for any $q \in \RR$,
and that $q\ph^u$ has the Bowen property on $\GGG(\eta)$ if and only if $\psi^u$ does. From now on, we work with $\psi^u$. 

Given $v\in T^1M$ and $t\in \RR$, we define $\UUU^u_v(t)$ to be the second fundamental form of the unstable horosphere $H^u(f_tv)$, as in \S\ref{sec:lambda-def}. Then  $\UUU^u_v(t)$ is a positive semidefinite symmetric linear operator on $(f_t v)^\perp$ such that if $J(t)$ is an unstable Jacobi field along $\gamma_v$, then $J'(t) = \UUU^u_v(t) J(t)$; see Lemma \ref{J'-is-U}.  

Now \eqref{eqn:psi-u} gives $\psi^u(v) = -\tr\UUU^u_v(0)$, so
\begin{equation}\label{eqn:psi-u-trace}
\int_0^T \psi^u(f_t v)\,dt = -\int_0^T \tr\UUU^u_v(t)\,dt.
\end{equation}

The rest of this section is devoted to proving the following.

\begin{prop}\label{prop:bowenforgeo}
For every $\eta>0$ there are $\delta,Q,\xi>0$ such that given any $(v,T)\in \GGG(\eta)$, $w\in W_\delta^s(v)$, and $w'\in f_{-T} W_\delta^u(f_{T} v)$, for every $0\leq t\leq T$ we have
\begin{align}\label{eqn:geom-bowen}
|\tr\UUU_v^u(t) - \tr\UUU_w^u(t)| &\leq Q e^{-\xi t}, \\
\label{eqn:geom-bowen-2}
|\tr\UUU_v^u(t) - \tr\UUU_{w'}^u(t)| &\leq Q \big(e^{-\xi t} + e^{-\xi(T - t)}\big).
\end{align}
\end{prop}

Since $\int_0^T \psi^u(f_t v)\,dt = -\int_0^T \tr\UUU^u_v(t)\,dt$,  \eqref{eqn:geom-bowen} shows that $\psi^u$ has the Bowen property on $\GGG(\eta)$ along stable leaves, and \eqref{eqn:geom-bowen-2} 
gives it along unstable leaves. Thus, by Lemma \ref{lem:lps-Bowen}, $\psi^u$ has the Bowen property on $\GGG(2 \eta)$, so we obtain the desired result:

\begin{cor} \label{cor:bow}
For every $\eta>0$, the potential $\psi^u$, and thus the potential $\ph^u$, has the Bowen property on $\GGG(\eta)$.
\end{cor}

To prove Proposition \ref{prop:bowenforgeo}, we study $\UUU_v^u(t)$ by using the fact that its time evolution is governed by a Riccati equation, which we now describe.
For $v \in T^1M$, let $\KKK(v)\colon v^\perp \to v^\perp$ be the symmetric linear map such that $\langle \KKK(v)X,Y\rangle = \langle R(X,v)v,Y\rangle$
for $X,Y \in v^\perp$. 
The eigenvalues of $\KKK(v)$ are sectional curvatures of planes containing $v$. Consequently $\KKK(v)$ is negative semidefinite.
We recall from  \eqref{eqn:Jacobi} that Jacobi fields along $\gamma_v$ evolve according to $J''(t) + \KKK(f_t v) J(t)=0$.  Lemma \ref{J'-is-U} shows that if $J(t)$ arises from varying $\gamma=\gamma_v$ through unit speed geodesics orthogonal to a hypersurface $H$, then $J'(t) = \UUU(t) J(t)$, where $\UUU(t)$ is the second fundamental form of $f_t H$. 
Differentiating this, the second-order ODE above becomes 
\[
0 = J''(t) + \KKK(\dot\gamma(t)) J(t)
= (\UUU'(t) + \UUU^2(t) + \KKK(\dot\gamma(t))) J(t).
\]
This shows that $\UUU(t)$ is a solution of the Riccati equation along $\gamma$:
\begin{equation}\label{e.ricc}
\UUU'(t) + \UUU^2(t) + \KKK(\dot\gamma(t)) = 0.
\end{equation}
In particular, $\UUU_v^u(t)$, which we defined as the second fundamental form of the unstable horosphere $H^u(f_t v)$, is a solution of \eqref{e.ricc}.

Using parallel translation along $\gamma$ to identify the spaces $\dot\gamma(t)^\perp$, we can represent 
$\UUU$ and $\KKK$ by symmetric $(n-1)\times(n-1)$ matrices; this matrix Riccati equation 
was introduced by Green in \cite{lG58}.  
 When $M$ is a surface, 
 the Riccati equation \eqref{e.ricc} along $\gamma_v$ becomes
\begin{equation}\label{e.ricc-2}
U'(t) + U^2(t) + K(f_t v) = 0,
\end{equation}
where $K(f_t v)$ is the Gaussian curvature at $\gamma_v(t)$. A nice exposition of the Riccati equation for non-positive curvature surfaces is in \cite{aM81}. 

We now prove Proposition \ref{prop:bowenforgeo}. 
Let $V$ be the space of symmetric $(n-1)\times (n-1)$ matrices, equipped with the semi-metric
\[
\rho(\AAA,\BBB) = |\tr\AAA - \tr\BBB|.
\]
Given $v\in T^1M$ and $s \leq t \in \RR$, let $\RRR_{s,t}^v\colon V\to V$ denote the time-evolution map from time $s$ to time $t$ for the nonautonomous ODE 
\begin{equation} \label{e.ricc-3}
\UUU'(\tau) + \UUU^2(\tau) + \KKK(f_\tau v) = 0.
\end{equation}
That is, $\RRR_{s, t}^v(\AAA)= \UUU(t)$, where $\UUU$ is the solution of \eqref{e.ricc-3} with $\UUU(s)=\AAA$. Then given $v,w\in T^1M$, we have 
\begin{multline}\label{eqn:rhoUvUw}
\rho(\UUU_v^u(t), \UUU_w^u(t)) = \rho(\RRR_{0,t}^v \UUU_v^u(0), \RRR_{0,t}^w \UUU_w^u(0)) \\
\leq \rho(\RRR_{0,t}^v\UUU_v^u(0), \RRR_{0,t}^v\UUU_w^u(0))
+\rho(\RRR_{0,t}^v\UUU_w^u(0),\RRR_{0,t}^w\UUU_w^u(0)).
\end{multline}
To estimate the first term, we will establish contraction properties of $\RRR_{0,t}^v$ on a suitable subset of $V$.  
Given $\AAA,\BBB\in V$, write
$\AAA \succcurlyeq \BBB$ if $\AAA-\BBB$ is positive semi-definite and $\AAA \succ \BBB$ if $\AAA-\BBB$ is positive definite.  Similarly, write $\AAA \preccurlyeq \BBB$ if $\AAA-\BBB$ is negative semi-definite and $\AAA \prec \BBB$ if $\AAA-\BBB$ is negative definite.  Fix $b>0$ such that $-b^2$ is a strict lower bound for the sectional curvatures of $M$, and let $D = \{\UUU\in V : 0 \preccurlyeq \UUU \preccurlyeq b\III\}$. The following lemma, proved in \S \ref{s:invdom}, shows that $D$ is a forward-invariant domain for the maps $\RRR_{s,t}^v$.
\begin{lem}\label{lem:invdom}
For every $v\in T^1M$ and $s\leq t\in \RR$, we have $\RRR_{s,t}^v D \subset D$.
\end{lem}
Henceforth, we use the letter $Q$ generically for a constant whose precise value will be different at different occurrences. Recall that the function $\tl \geq 0$ was defined in \S\ref{sec:Geta} as $\tl(v) = \max(0, \lambda(v) - \frac\eta2)$. 
The following lemma allows us to estimate the first term in \eqref{eqn:rhoUvUw}.

\begin{lem}\label{lem:Rcontract}
For every $\eta>0$, there is a constant $Q>0$
 such that for every $v\in T^1M$, $s\leq t\in \RR$, and $\UUU_0,\UUU_1\in D$, we have
\begin{equation}\label{eqn:Rcontract}
\rho(\RRR_{s,t}^v \UUU_0, \RRR_{s,t}^v \UUU_1) \leq Q e^{- \int_s^t \tl (f_\tau v) \,d\tau} 
\|\UUU_0 - \UUU_1\|.
\end{equation}
\end{lem}

We prove Lemma \ref{lem:Rcontract} in \S \ref{s.Rcontract}. To estimate the second term in \eqref{eqn:rhoUvUw}, we fix $v,w\in T^1M$, $t\geq 0$, and $\UUU_0\in D$, and consider the function $R=R^{v,w,t}_{\UUU_0} \colon [0,t]\to D$ given by
\begin{equation}\label{eqn:rs}
R(s) = \RRR_{s,t}^v \RRR_{0,s}^w \UUU_0,
\end{equation}
so $R(s)$ evolves $\UUU_0$ by the Riccati equation for $w$ until time $s$, then evolves by the Riccati equation for $v$ from time $s$ to time $t$.  Our proof of Lemma \ref{eqn:Rcontract} shows that  $\UUU_w^u(0) \in D$, so we can set $\UUU_0=\UUU_w^u(0)$ to obtain a path in $D$ that connects $R(0) = \RRR_{0,t}^v\UUU_w^u(0)$ to $R(t) = \RRR_{0,t}^w\UUU_w^u(0)$. Thus we can estimate the second term in \eqref{eqn:rhoUvUw} by bounding the length of the path $R$ in the pseudo-metric $\rho$. 
\begin{lem}\label{lem:rho0st}
Given any $v,w\in T^1M$ and $t\geq 0$, the function $R=R^{v, w, t}_{\UUU_w^{u}(0)}$ 
satisfies the following bound for all $0 \leq s_1 \leq s_2 \leq t$:
\begin{equation}\label{eqn:abscts}
\rho(R(s_1),R(s_2)) \leq \int_{s_1}^{s_2}
Q e^{- \int_s^t \tl(f_\tau v) \,d\tau}\|\KKK(f_s v) - \KKK(f_s w)\| \,ds.
\end{equation}
\end{lem}
We prove Lemma \ref{lem:rho0st} in \S \ref{s.rho0st}. We now explain how  to prove Proposition \ref{prop:bowenforgeo} from  Lemmas \ref{lem:Rcontract}  and \ref{lem:rho0st}. 
Given $\eta>0$, let $\delta>0$ be as in \eqref{eqn:delta}.  Given $(v,T)\in \GGG$ and $w\in W_\delta^s(v)$, smoothness of $\KKK\colon T^1 M \to V$ together with  \eqref{eqn:Ws-contract-1} gives
\[
\|\KKK(f_s v) - \KKK(f_s w)\| \leq Q \dK(f_s v, f_s w)  
\leq Q d^s(f_s v, f_s w) \leq Q \delta 
e^{-\int_0^s \tl(f_\tau v)\,d\tau}
\]
for all $s\in [0,T]$. 
We conclude that for every $t\in [0,T]$, the integrand in \eqref{eqn:abscts} is bounded above by
\[
Q e^{-\int_0^t \tl (f_\tau v)\,dt} \leq Q e^{-\int_0^t \lambda(f_\tau v)\,dt + \frac\eta 2 t}  \leq Q e^{-\frac\eta2t},
\]
where the last inequality holds because $(v, T) \in \GGG(\eta)$. Thus, \eqref{eqn:abscts} gives  the estimate
$\rho(R(s_1),R(s_2)) \leq (s_2-s_1)Q e^{-\frac\eta2t}$.
 Fixing $\xi < \frac\eta 2$, and setting $s_1=0$, $s_2=t$, 
 we obtain
\[
\rho(\RRR_{0,t}^v \UUU_w^u(0), \RRR_{0,t}^w \UUU_w^u(0))
\leq Qt e^{-\frac\eta 2 t} < Qe^{-\xi t},
\]
which bounds the second term of \eqref{eqn:rhoUvUw}. By \eqref{eqn:Rcontract}, we have
\begin{equation}\label{eqn:Rcontract2}
\rho(\RRR_{0,t}^v\UUU_v^u(0), \RRR_{0,t}^v\UUU_w^u(0))
\leq Q e^{-\int_0^t \tilde\lambda(f_\tau v)\,d\tau }\leq Q e^{-\frac\eta 2 t},
\end{equation}
which bounds the first term of \eqref{eqn:rhoUvUw}.
Thus, both terms of \eqref{eqn:rhoUvUw} are bounded above by  $Qe^{-\xi t}$,which proves the first half of Proposition \ref{prop:bowenforgeo}.

To prove \eqref{eqn:geom-bowen-2}, first observe that when $(v,T)\in \GGG$ and $f_T w' \in W_\delta^u(f_T v)$, we can use \eqref{eqn:Wu-contract} to get
\[
\|\KKK(f_sv) - \KKK(f_s w')\| \leq Q e^{- \int_s^T \tl(f_t w')\,dt} \leq Q e^{-\frac\eta2 (T-s)}.
\]
Now letting $R=R^{v,w',t}_{\UUU^{u}_{w'}(0)}$ and $t\in [0,T]$, \eqref{eqn:abscts} gives the bound  
\begin{align*}
\rho(R(0), R(t)) & \leq Q \int_0^t  \|\KKK(f_sv) - \KKK(f_s w')\|\,ds \\
 &\leq Q \int_0^t  e^{-\frac\eta2 (T-s)}\,ds
 \leq Q e^{-\frac\eta 2(T-t)}.
\end{align*}
Thus, $\rho(\RRR_{0,t}^v \UUU_{w'}^u(0), \RRR_{0,t}^{w'} \UUU_{w'}^u(0)) \leq Q e^{-\frac\eta 2(T-t)}$.
Also, \eqref{eqn:Rcontract2} holds with $w'$ in place of $w$.  
Using these bounds in \eqref{eqn:rhoUvUw} gives \eqref{eqn:geom-bowen-2} with $\xi = \frac\eta2$. Modulo the proofs of Lemmas \ref{lem:invdom}, \ref{lem:Rcontract} and \ref{lem:rho0st}, which are given in the next sections, this completes the proof of Proposition \ref{prop:bowenforgeo}.

\subsection{Proof of Lemma \ref{lem:invdom}} \label{s:invdom}

The following three lemmas give forward invariance of the domain $D$ under the maps $\RRR_{s,t}^{v}$ for any $v \in T^1M$.

\begin{lem}\cite[p.\ 50]{Copp}\label{l.monotone}
Suppose $\UUU_1(t)$ and $\UUU_2(t)$ are symmetric solutions of \eqref{e.ricc} with $\UUU_1(t_0) \succcurlyeq \UUU_2(t_0)$. Then $\UUU_1(t) \succcurlyeq \UUU_2(t)$ for all $t$. Similarly,
if $\UUU_1(t_0) \succ \UUU_2(t_0)$, then $\UUU_1(t) \succ \UUU_2(t)$ for all $t$.
\end{lem}
\begin{proof} Both $\DDD(t) = \UUU_1(t) - \UUU_2(t)$ and $\MMM(t) = \frac12(\UUU_1(t) + \UUU_2(t))$ are symmetric and by a straightforward computation, satisfy
\[
\DDD' + \DDD\MMM + \MMM\DDD = 0.
\]
Let $\XXX(t)$ be the solution of $\XXX'(t) = \MMM(t) \XXX(t)$ with $\XXX(t_0) = \III$. Then $\XXX(t)$ is non-singular for all $t$ and, since $\MMM$ is symmetric,
\[
(\XXX^*\DDD\XXX)' = \XXX^*(\DDD' + \DDD\MMM + \MMM\DDD)\XXX = 0.
\]
Thus $\XXX^*\DDD\XXX(t)$ is constant, so the signature of $\DDD(t)$ is constant.
\end{proof}

\begin{lem}\label{l.lowerbound}
Let $\UUU(t)$ be a symmetric solution of \eqref{e.ricc} with $\UUU(t_0) \succcurlyeq 0$. Then $\UUU(t) \succcurlyeq 0$ for all $t \geq t_0$.
\end{lem}
\begin{proof} 
Let $\UUU_\eps(t)$ be the (symmetric) solution of 
\begin{equation}\label{e.ricc-eps}
\UUU_\eps'(t) +\UUU_\eps^2(t) + \KKK(\dot\gamma(t)) -\epsilon^2\III= 0
\end{equation}
with $\UUU_\eps(t_0) = \UUU(t_0) \succcurlyeq 0$.  Then $\lim_{\eps\to 0} \UUU_\eps(t) = \UUU(t)$ for all $t$, so it suffices to prove that $\UUU_\eps(t) \succcurlyeq 0$ for all $t\geq t_0$ and $\eps>0$.  Let 
\[
S = \{t \geq t_0 \mid \UUU_\eps(t_1) \succcurlyeq 0 \text{ for all } t_1 \in [t_0,t] \}.
\]
Suppose $S$ is bounded above, and let $t_1 = \sup S$.  Let $\UUU^1_\eps$ be the solution of \eqref{e.ricc-eps} with $\UUU^1_\eps(t_1) = 0$.  By Lemma \ref{l.monotone}, we have $\UUU_\eps(t) \succcurlyeq \UUU_\eps^1(t)$ for all $t\in \RR$.  However, $(\UUU_\eps^1)'(t_1) = -\KKK(\dot\gamma(t)) + \eps^2\III$ is positive definite, so there is some $t_2>t_1$ with the property that $\UUU_\eps^1(t) \succ 0$ for all $t\in (t_1, t_2]$, and consequently $\UUU_\eps(t) \succcurlyeq 0$ for all $t\in (t_1,t_2]$.  This means that $t_2\in S$, contradicting maximality of $t_1$.  We conclude that $S = [t_0,\infty)$, which proves the lemma.
\end{proof}

Recall that $b>0$ was chosen so that $-b^2$ is a strict lower bound for the sectional curvatures of $M$.

\begin{lem}\label{l.upperbound} 
 Suppose $\UUU(t)$ is a solution of \eqref{e.ricc} with $b\III\succcurlyeq \UUU(t_0)$. Then $b\III \succcurlyeq \UUU(t)$ for $t \geq t_0$.
\end{lem}
\begin{proof} 
Proceed as in Lemma \ref{l.lowerbound} by observing that
$\UUU'(t) = - \UUU^2(t) -  \KKK(\dot\gamma(t)) \prec 0$ if $\UUU(t) = b\III$ and applying Lemma~\ref{l.monotone}.
\end{proof}

We conclude that $D$ is an invariant domain for evolution under the Riccati equation \eqref{e.ricc}. Thus, for every $v\in T^1M$ and $s\leq t\in \RR$, we have $\RRR_{s,t}^v D \subset D$. 

\subsection{Proof of Lemma \ref{lem:Rcontract}} \label{s.Rcontract}
We begin by proving convergence results to $\UUU^u$ for Riccati solutions with positive semi-definite initial conditions.

\begin{lem}\label{l.unifconv}
Let $\UUU^u_{v,\tau}$ be the solution of the Riccati equation along $\gamma_v$ such that $\UUU^u_{v,\tau}(-\tau) = 0$.
Then $\UUU^u_{v,\tau}(0) \to \UUU^u_v(0)$ as $\tau \to \infty$. The convergence is uniform in $v$. 
\end{lem}

\begin{proof} We have $\UUU^u_{v,\tau}(-\tau) = 0 \preccurlyeq \UUU^u_v(f_{-\tau}v) = \UUU^u_v(-\tau)$. It follows from Lemma~\ref{l.monotone} that $\UUU^u_{v,\tau}(t) \preccurlyeq \UUU^u_v(t)$ for all $t$, in particular when $t = 0$. On the other hand, Lemma~\ref{l.lowerbound} tells us that $\UUU^u_{v,\tau}(t) \succcurlyeq 0$ for  $t \geq -\tau$. It follows that if $0 \leq \tau_1 \leq \tau_2$, then $0 \preccurlyeq \UUU^u_{v,\tau_1}(0) \preccurlyeq \UUU^u_{v,\tau_2}(0)\preccurlyeq \UUU^u_v(0)$.  We would like to deduce that $\UUU^u_{v,\tau}(0)$ converges to $\UUU_v^u(0)$ as $\tau \to \infty$.

Observe that for every $x\in \RR^{n-1}$, the sequence $\langle x,\UUU_{v,\tau}^u(0)x\rangle$ is monotonic in $\tau$, and hence has a limit as $\tau\to\infty$.  Since this holds for every $x$, we conclude that $\lim_{\tau\to\infty} \UUU_{v,\tau}^u(0)$ exists and that it is $\preccurlyeq \UUU_v^u(0)$; it remains to show that the limit is in fact $ \UUU^u_v(0)$ for each $v$.

Let $J_{v,w,\tau}$ be a Jacobi field along $\gamma_v$ that satisfies $J_{v,w,\tau}(0) = w \in v^\perp$ and $J'_{v,w,\tau}(-\tau) = 0$.  Since the norm of a Jacobi field is a convex function, we have $\|J_{v,w,\tau}(t)\| \leq \|w\|$ for $-\tau \leq t \leq 0$. If  $\tau_k$ is a sequence such that $\tau_k \to \infty$ and $J_{v,w,\tau_k}$ is a sequence that converges to a Jacobi field $J$, then  $\|J(t)\| \leq \|J(0)\|$ for all $t \leq 0$, and hence $J$ is the unstable Jacobi field with initial value $w$. Since we have the same limit for any such subsequence, it  follows that  
$J_{v,w,\tau}$ converges as $\tau \to \infty$ to the unstable Jacobi field with initial value $w$. Thus $\UUU^u_{v,\tau}(0) \to \UUU^u_v(0)$ for each~$v$.

Now given any $x\in \RR^{n-1}$, Dini's theorem tells us that $\langle x,\UUU_{v,\tau}^u(0)x\rangle  \to \langle x ,\UUU_v^u(0)x\rangle$ uniformly in $v$.  Since a symmetric matrix $U$ is completely determined by $\langle x,Ux\rangle$ for a finite number of values of $x$, 
this shows that  $\UUU_{v,\tau}^u(0) \to \UUU_v^u(0)$ uniformly in $v$. \end{proof}
\begin{cor}
For any $v \in T^1M$, $\UUU^u_v(0) \in D$.
\end{cor}
\begin{proof}
Lemma \ref{lem:invdom} tells us that $\UUU^u_{v,\tau}(0) \in D$  for all $\tau$. Since $D$ is compact, it follows from Lemma \ref{l.unifconv} that $\UUU_v^u(0) \in D$. 
\end{proof}

\begin{prop}\label{p.unifconv}
For each  $\epsilon > 0$ there is $\tau_0(\epsilon)>0$ such that if $\UUU(t)$ is a solution of the Riccati equation along the geodesic $\gamma_v$ and $t_0\in \RR$ is such that $\UUU(t_0) \succcurlyeq 0$, then $\UUU(t) \succcurlyeq \UUU^u_v(t)  -\epsilon\III$ for every $t\geq t_0 + \tau_0(\eps)$.
\end{prop}
\begin{proof}
Lemma \ref{l.unifconv} gives $\tau_0 = \tau_0(\eps)$ such that $\UUU_{w,\tau}^u(0) \succcurlyeq \UUU_w^u(0) - \eps\III$ for all $w\in T^1M$ and $\tau \geq \tau_0(\eps)$.  Let $w = f_\tau v$ and $\tau = t$.
\end{proof}

To prove Lemma \ref{lem:Rcontract}, it suffices to consider the case when $s=0$; to obtain the result when $s\neq 0$, replace $v\in T^1M$ by $f_s v$. By Proposition \ref{p.unifconv}, there is $\tau_0 = \tau_0(\frac\eta 2)$ such that for any $v\in T^1M$ and $\UUU_0 \in D$, we have $\RRR_{0,t}^v \UUU_0 \succcurlyeq \UUU_v^u(t) - \frac\eta2 \III$ for all $t\geq \tau_0$. We start by proving an estimate that is useful for controlling the pseudo-metric $\rho$ locally. 
For $\UUU \in D$, we write $\UUU(t)$ to denote $\RRR_{0,t}^v\UUU$.

\begin{lem} \label{l.local}
If $\underline\UUU, \overline\UUU \in D$ have 
$\underline\UUU \preccurlyeq \UUU_j \preccurlyeq \overline\UUU$ for $j=0,1$, then
\begin{equation}\label{eqn:rhoDelta}
\rho(\UUU_0(t),\UUU_1(t)) \leq e^{\tau_0\|\lambda\|} e^{-\int_0^t\tl(f_\tau v)\,d\tau}(\tr\overline\UUU - \tr\underline\UUU).
\end{equation}
\end{lem}

\begin{proof}
By Lemma \ref{l.monotone}, we have $\underline\UUU(t) \preccurlyeq \UUU_j(t) \preccurlyeq \overline\UUU(t)$ for all $t\in \RR$ and $j=0,1$.
Positive semi-definiteness of $\UUU_j(t) - \underline\UUU(t)$ gives $\tr\UUU_j(t) = \tr\underline\UUU(t) + \tr(\UUU_j(t) - \underline\UUU(t)) \geq \tr\underline\UUU(t)$.  Similarly, $\tr\UUU_j(t) \leq \tr\overline\UUU(t)$, and so
\[
\rho(\UUU_0(t),\UUU_1(t))
= |\tr\UUU_0(t) - \tr\UUU_1(t)|
\leq \tr\overline\UUU(t) - \tr\underline\UUU(t)
=: \Delta(t).
\]
Writing $\underl_1(t) \leq \underl_2(t) \leq \cdots \leq \underl_{n-1}(t)$ for the eigenvalues of $\underline\UUU(t)$, and similarly for the eigenvalues of $\overline\UUU(t)$, we have
\begin{align*}
\Delta'(t) &= \tr (\overline\UUU'(t) - \underline\UUU'(t))
= \tr (\underline\UUU(t)^2 - \overline\UUU(t)^2) \\
&= -\big(\tr(\overline\UUU(t)^2) - \tr(\underline\UUU(t)^2)\big) \\
&= -\sum_{i=1}^{n-1} (\overl_i(t)^2 - \underl_i(t)^2)
= -\sum_{i=1}^{n-1} (\overl_i(t) - \underl_i(t))(\overl_i(t)+\underl_i(t)).
\end{align*}
Weyl's Monotonicity Theorem \cite[Corollary III.2.3]{rB97} states that the positive semi-definiteness of $\overline\UUU(t) - \underline\UUU(t)$ gives $\overl_i(t) \geq \underl_i(t) \geq 0$, so $\Delta'(t) \leq 0$ for all $t\geq 0$.  Moreover, since $\underline\UUU(0)\in D$, Proposition \ref{p.unifconv} gives $\underl_i(t) \geq \lambda(f_t v) - \frac\eta2$ for all $t\geq \tau_0$, and thus $\underl_i(t) \geq  \tl(f_tv)$.  Thus, for $t\geq \tau_0$, we have
\[
\Delta'(t)
\leq - \sum_{i=1}^{n-1} 2\tl(f_tv)(\overl_i(t) - \underl_i(t))
= - 2\tl(f_tv) \Delta(t),
\]
and so 
\[
\begin{aligned}
\rho(\UUU_0(t),\UUU_1(t))
&\leq
\Delta(\tau_0) e^{-\int_{\tau_0}^t 2\tl(f_\tau v) \,d\tau} \\
&\leq\Delta(0) e^{\int_0^{\tau_0}  \tl(f_\tau v)\,d\tau} e^{-\int_0^t \tl(f_\tau v)\,d\tau} 
\\
&\leq (\tr\overline\UUU - \tr\underline\UUU)e^{\tau_0\|\lambda\|} e^{-\int_0^t \tl(f_\tau v)\,d\tau}.\qedhere
\end{aligned} 
\]
\end{proof}

We now apply the estimate \eqref{eqn:rhoDelta} locally on the interior of $D$, and show how to use this to obtain the global estimate \eqref{eqn:Rcontract}. 
First assume that $\UUU_0,\UUU_1$ are positive definite, and let $\eps>0$ be such that $\UUU_0,\UUU_1 \succcurlyeq \eps\III$ and $n = \|\UUU_0 - \UUU_1\|/\eps$ is an integer.
 Given $q\in (0,1)$, let $\UUU_q = (1-q)\UUU_0 + q\UUU_1$  and observe that $\UUU_q \succcurlyeq \eps\III$.  For every $0\leq k < n$, we have
$
\|\UUU_{(k+1)/n} - \UUU_{k/n}\| < \eps.
$

Now let $\underline\UUU_k = \UUU_{k/n} - \eps\III$ and $\overline\UUU_k = \UUU_{k/n} + \eps\III$. For $j=  k/n, (k+1)/n$, we have $\underline \UUU_k \preccurlyeq  \UUU_{j} \preccurlyeq \overline\UUU_k$ , so writing
$\UUU_q(t) = \RRR_{0,t}^v(\UUU_q)$ for $q\in [0,1]$, and applying Lemma \ref{l.local} 
gives
\[
\rho(\UUU_{k/n}(t), \UUU_{(k+1)/n}(t)) \leq
e^{\tau_0\|\lambda\|} e^{-\int_0^t \tl(f_\tau v)\,d\tau} 2\eps.
\]
Summing over all $k$, and using the fact that $n\eps=\|\UUU_0 - \UUU_1\|$, gives
\[
\rho(\UUU_0,\UUU_1) \leq 2e^{\tau_0\|\lambda\|}  e^{-\int_0^t\tl(f_\tau v)\,d\tau} \|\UUU_0 - \UUU_1\|.
\]
This proves \eqref{eqn:Rcontract} when $\UUU_0,\UUU_1$ are positive definite.  For the positive semidefinite case, replace $\UUU_j$ with $\UUU_j^\delta := \UUU_j + \delta\III$ for $\delta>0$, and observe that
$\lim_{\delta\to 0} \RRR_{0,t}^v \UUU_j^\delta = \RRR_{0,t}^v \UUU_j$. This completes the proof of Lemma \ref{lem:Rcontract}.

\subsection{Proof of Lemma \ref{lem:rho0st}} \label{s.rho0st}

Fixing $v,w\in T^1M$ and $t\geq 0$, let $R\colon [0,t]\to D$ be as in \eqref{eqn:rs}, and define $G\colon [0,t] \times [0,t] \to [0,\infty)$ by
\begin{equation}\label{eqn:Gs}
G(s',s'') = Q \|\KKK(f_{s'} v) - \KKK(f_{s'} w)\| \exp\left(-\int_{s''}^t \tl(f_\tau v)\,d\tau\right) .
\end{equation}
Note that $G(s,s)$ is the integrand on the right hand side of the estimate \eqref{eqn:abscts} that we wish to prove.  

We need a uniform continuity property of the map $(v,s,\UUU) \mapsto \|\RRR_{0,s}^v\UUU\|$.

\begin{lem}\label{lem:delta-s-small}
For every $\eps > 0$ there is $\delta > 0$ such that for any given $v,w \in T^1M$, $\UUU \in D$, and $s_0 \in [0,\delta]$, we have $\|\RRR^v_{0,s_0}\UUU - \RRR^w_{0,s_0}\UUU\| \leq (\|\KKK(v) - \KKK(w)\| + \epsilon)s_0$.
\end{lem}
\begin{proof}
When $s = 0$ we have $\RRR^v_{0,0}\UUU = \UUU = \RRR^w_{0,0}\UUU$. Let $F(v,s,\UUU) = \RRR^v_{0,s}\UUU$. Then
  $$
  \RRR^v_{0,s_0}\UUU - \RRR^w_{0,s_0}\UUU = \int_0^{s_0} (  \partial F/\partial s(v,s,\UUU) - \partial F/\partial s(v,s,\UUU))\,ds.
 $$

The Riccati equation tells us that  $\partial F/ \partial s(v,0,\UUU) = -\KKK(v) - \UUU^2$ for any  $(v,\UUU) \in T^1M \times D$. Hence 
$$
\partial F/ \partial s(v,0,\UUU) - \partial F/ \partial s(w,0,\UUU) = \KKK(w) - \KKK(v).
$$
For any small enough $\Delta > 0$, the function $F$ is well defined and $C^\infty$ on the compact space $T^1M \times [0,\Delta] \times D$. Hence $\partial F/\partial s$ is uniformly continuous on this space. In particular, given $\epsilon > 0$, there is $\delta \in (0, \Delta)$ such that
 $$
\| \partial F/\partial s(v,0,\UUU) - \partial F/\partial s(v,s,\UUU) \| < \epsilon/2
  $$
for all $(v,\UUU) \in T^1M \times D$ and all $s \in [0,\delta]$. It follows that if $0 \leq s \leq \delta$, then
  \[
  \| \partial F/\partial s(v,s,\UUU) -  \partial F/\partial s(w,s,\UUU)\| \leq \|\KKK(v) - \KKK(w)\| + \epsilon. \qedhere
 \]
\end{proof}

Now fix $\eps>0$, let $Q$ be the constant from Lemma \ref{lem:Rcontract}, and choose $\delta > 0$ so that Lemma \ref{lem:delta-s-small} holds with $\eps/Q$ in place of $\eps$. Fix $v,w\in T^1M$ and $\UUU_0\in D$. Suppose  $0 \leq s' \leq s'' \leq s' + \delta$. Noting that $\RRR_{0,s''}^w\UUU_0=\RRR_{s',s''}^w\RRR_{0,s'}^w\UUU_0 $, Lemma~\ref{lem:delta-s-small} gives us
\begin{equation} \label{e:from7.19}
\|\RRR_{s',s''}^v\RRR_{0,s'}^w\UUU_0 - \RRR_{0,s''}^w\UUU_0\| \leq (\|\KKK(f_{s'} v) - \KKK(f_{s'} w)\| + \eps/Q)(s'' - s').
\end{equation}
If $s'' \leq t$, then letting $R=R^{v,w,t}_{\UUU_0}$, we have
\begin{align*}
\rho(R(s'),R(s''))& = \rho(\RRR^v_{s',t}\RRR^w_{0, s'} \UUU_0, \RRR^v_{s'',t}\RRR^w_{0, s'} \UUU_0)\\ 
&= \rho(\RRR^v_{s'',t}(\RRR_{s',s''}^v\RRR_{0,s'}^w\UUU_0), \RRR^v_{s'',t}(\RRR^w_{0, s''} \UUU_0)),
\end{align*}
and so applying Lemma \ref{lem:Rcontract} and the estimate \eqref{e:from7.19} gives
\begin{align*}
\rho(R(s'),R(s'')) &\leq Q  e^{-\int_{s''}^t \tl(f_\tau v)\,d\tau} \|\RRR_{s',s''}^v\RRR_{0,s'}^w\UUU_0 - \RRR_{0,s''}^w\UUU_0\| \\
& \leq (G(s',s'') + \eps)(s''-s'),
\end{align*}
where $G$ is as in \eqref{eqn:Gs}.

Now given $0\leq s_1\leq s_2\leq t$ as in the statement of Lemma \ref{lem:rho0st}, let $n$ be large enough that $(s_2-s_1)/n \leq \delta$, and put $s_i^* = s_1 + \frac in (s_2-s_1)$ for $0\leq i\leq n$; we obtain
\[
\rho(R(s_1),R(s_2)) \leq \sum_{i=1}^n \rho(s_{i-1}^*,s_i^*) \leq \sum_{i=1}^n (G(s_{i-1}^*,s_i^*) + \eps) \Big( \frac{s_2-s_1}n\Big).
\]
As $n\to\infty$, the sum on the right converges to $\int_{s_1}^{s_2} (G(s,s) + \eps)\,ds$ since $G$ is continuous.  Since $\eps>0$ was arbitrary, this proves \eqref{eqn:abscts}.


\section{Pressure gap} \label{s.entropygap}
In this section, we prove Theorem \ref{thm:pressure-gap} which states that $P(\Sing,\ph) < P(\ph)$ when $\ph$ is locally constant in a neighborhood of $\Sing$.  We give an outline of the argument in \S\ref{sec:gap-outline}, and the details in \S\S\ref{s:singreg}--\ref{s.pressureproduction}.

\subsection{Outline of proof}\label{sec:gap-outline}
The first step in the proof of Theorem \ref{thm:pressure-gap} is to approximate orbit segments $(v, t)$ in the singular set by regular orbit segments in $\CCC(\eta)$; this is \S\ref{s:singreg}.  The second step is to establish a partition sum estimate for the regular orbit segments in $\CCC(\eta)$ that approximate $\Sing$; this is \S\ref{sec:pressure estimates}.  The third and final step is to use the specification property for $\CCC(\eta)$ to construct a collection of orbits with greater topological pressure than the singular set; 
this is \S\ref{s.pressureproduction}.  Here we briefly outline each step.

\emph{Step one: Approximation.} In \S\ref{s:singreg}, the approximation map $\Pi_t \colon \Sing\to \Reg$ is defined. Theorem \ref{thm:sing-to-reg} gives its most important properties.  The idea is to move the beginning of the orbit segment from $v$ to an appropriate $v' \in W^s(v) \cap \Reg(\eta)$, and then move the end of the orbit segment from $f_t(v')$ to an appropriate $f_t(w) \in W^u(f_t v') \cap \Reg(\eta)$, obtaining $\Pi_t(v) := w$ with the property that $(w, t) \in \CCC(\eta)$.

Naively, one might hope that we could construct $w=\Pi_t(v)$ so that $d(f_sv, f_sw)$ is small for most $s\in [0,t]$.  One can see that this is too much to ask by considering the example of $v\in \Sing$ whose orbit stays in the middle of a flat strip. Instead, the best approximation property we can get is that $f_sw$ will be close to \emph{some} orbit in $\Sing$ along the middle of the orbit, i.e., for all $s \in [L, t-L]$, where $L$ is independent of $v$ and $t$.  For example, if $v$ stays in the middle of a flat strip on a surface, then $\Pi_t(v)$ will approach the \emph{edge} of the flat strip. This is illustrated in Figure \ref{fig:regularizing}. This is the reason we assume that $\ph$ is locally constant on a neighborhood of $\Sing$; this condition guarantees that $|\int_0^t \ph(f_s v)\,ds - \int_0^t \ph(f_s w)\,ds|$ is uniformly bounded even through $f_sv$ and $f_sw$ may be far apart.


\begin{figure}[htbp]
\includegraphics[width=\textwidth]{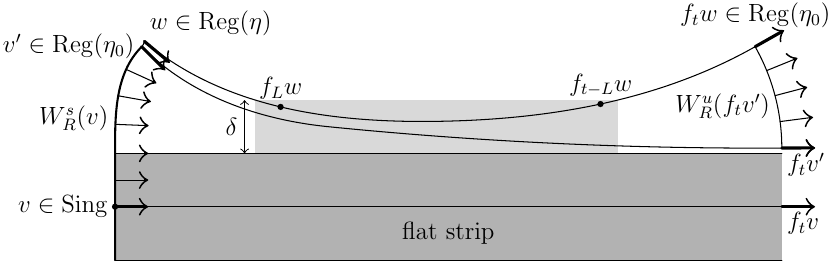}
\caption{The approximation map $\Pi_t\colon \Sing\to \Reg$.}
\label{fig:regularizing}
\end{figure}

\emph{Step two: Partition sum estimates.} In \S\ref{sec:pressure estimates}, we obtain a good lower bound for the partition sums associated to $\Pi_t(\Sing)  \subset \CCC(\eta)_t$. The crucial step is to control the multiplicity of the map $\Pi_t$, which is carried out in Proposition \ref{prop:multiplicity}.  This allow us to produce a $(t, \epsilon)$-separated set $E_t\subset \Sing$ such that $\Pi_t(E_t)$ contains a $(t,\delta)$-separated set $E_t''$ satisfying (Lemma \ref{lem:En''})
\begin{equation}\label{eqn:En''0}
\sum_{w\in E_t''} e^{\inf_{u\in B_n(w,\delta)} \int_0^t \ph(f_s u)\,ds} \geq \beta e^{tP(\Sing,\ph)}.
\end{equation}
Moreover, each $w\in E_t''$ has the property that $f_s w$ is close to $\Sing$ when $s\in [L,t-L]$, and not close to $\Sing$ when $s=0,t$.



\emph{Step three: Producing pressure.} In \S\ref{s.pressureproduction}, we 
construct a collection of orbit segments which guarantees the pressure gap 
by using the specification property for $\CCC(\eta)$ together with the partition sum estimate \eqref{eqn:En''0}.  Given a suitable $\tau>0$, $N\in \mathbb{N}$ large and $\alpha>0$ small such that $\alpha N\in \mathbb{N}$, we formally split the time interval $[0,\tau N]$ into $\alpha N$ subintervals by choosing a subset $J$ of $\alpha N -1$ elements from the set $\{\tau, 2 \tau \ldots, \tau(N-1)\}$ to determine the endpoints of the subintervals.  We call the elements of $J$ `gluing times' because we will use the specification property to glue together orbit segments which start at these times.  The constant $\tau$ is chosen so that any two gluing times are far enough apart for our construction. We write $n_1,\dots, n_{\alpha N}$ for the lengths of the subintervals selected by $J$, and we use the specification property with transition time $T$ to glue together orbit segments in $\CCC(\eta)$ with lengths $n_i-T$: if $(v_1,n_1-T),\dots, (v_k,n_{\alpha N}-T) \in \CCC(\eta)$, Theorem \ref{specgeoflow} gives $w\in T^1M$ so that $(w, N)$ shadows each of the $(v_i,n_i-T)$ in turn with transition time $T$. In this way, for each $v_1\in E''_{n_{1}-T},\dots, v_{n_{\alpha N}} \in E''_{n_{\alpha N}-T}$ we obtain an orbit segment $(w,\tau N)$ which shadows each $(v_i,n_i-T)$ in turn.  The set $J$ of  gluing times can be recovered from $(w, \tau N)$ as the set of times $s \in [0, \tau N]$ for which $f_sw$ is not close to $\Sing$. Thus, our construction guarantees different orbit segments $(w, \tau N)$ for different choices of $J$.
%


We sum over all $\binom {N-1}{\alpha N-1} \approx e^{(-\alpha \log \alpha) N}$ choices of gluing times and  using \eqref{eqn:En''0} we obtain an $(N,\delta/3)$-separated set $F_N$ for which 
\[
\sum_{w\in F_N} e^{\int_0^{\tau N} \ph(f_t w)\,dt} \gtrapprox e^{-(\alpha \log \alpha) N} e^{-Q \alpha N} e^{N \tau P(\Sing,\ph)},
\]
where the constant $Q$ is independent of $\alpha$.  This gives
\[
P(\ph) \geq \alpha \tau^{-1}  (-\log \alpha - Q) + P(\Sing,\ph).
\]
The first term is positive if $\alpha >0$ is small, which gives the pressure gap that we want and proves Theorem \ref{thm:pressure-gap}. 
Intuitively, $\alpha$ must be chosen small because otherwise the error in the Birkhoff integral introduced by the specification property at each gluing time dominates the growth in complexity from the different itineraries available.

\subsection{Replacing singular orbit segments with regular ones} \label{s:singreg}

Fix $\eta_0>0$ small enough that $\Reg(\eta_0)$ has nonempty interior.  By Lemma \ref{lem:unif-dense}, there exists $R>0$ such that for every $v\in T^1M$ we have both $W_R^s(v) \cap \Reg(\eta_0)\neq\emptyset$ and $W_R^u(v)\cap \Reg(\eta_0)\neq\emptyset$.  In particular,
we can define maps $\Pi^s, \Pi^u \colon T^1 M \to \Reg(\eta_0)$ such that  $\Pi^\sigma(v)\in W_R^\sigma(v)$ for every $v\in T^1 M$ and $\sigma = s,u$.  
Given $t>0$, we use these to define a map $\Pi_t \colon \Sing\to \Reg$ by
\begin{equation}\label{eqn:Pit}
\Pi_t = f_{-t} \circ \Pi^u \circ f_t \circ \Pi^s.
\end{equation}
That is, given $v\in \Sing$ we choose $v' = \Pi^s(v) \in W_R^s(v)$ with $\lambda(v') \geq \eta_0$, and $w = f_{-t}(\Pi^u(f_t v'))$ such that $f_t w \in W_R^u(f_t v')$ and $\lambda(f_t w) \geq \eta_0$, as shown in Figure \ref{fig:regularizing}.


\begin{thm}\label{thm:sing-to-reg}
For every $\delta>0$ and $\eta \in (0,\eta_0)$, there exists $L>0$ such that for every $v\in \Sing$ and $t\geq 2L$, the image $w = \Pi_t(v)$ has the following properties:
\begin{enumerate}[label=\textup{(\arabic{*})}]
\item\label{wfw}
 $w, f_t(w) \in \Reg(\eta)$;
\item\label{nearSing}
$\dK(f_s(w),\Sing) < \delta$ for all $s\in [L,t-L]$;
\item \label{component}
for every $s\in [L,t-L]$, $f_s(w)$ and $v$ lie in the same connected component of $B(\Sing,\delta) := \{w\in T^1M : \dK(w,\Sing)<\delta)\}$.
\end{enumerate}
\end{thm}

We emphasize that Theorem \ref{thm:sing-to-reg} does \emph{not} allow us to conclude that $f_s(w)$ is close to $f_s(v)$;
 all we know is that $f_s(w)$ is close to \emph{some} singular vector for $s\in [L,t-L]$. For example, if $f_s(v)$ is in the middle of a flat strip on a surface, then $f_s(w)$ will be close to the edge of the flat strip for $t\in [L,t-L]$.


\begin{proof}[Proof of Theorem \ref{thm:sing-to-reg}]
Let $\delta,\eta,\eta_0$ be as in the statement of the theorem. For property \ref{wfw}, it is immediate from the definition of $\Pi_t$ that $\lambda(f_tw) \geq \eta$.  By uniform continuity of $\lambda$, we can take $\eps_0$ sufficiently small such that if $v_2 \in W_{\eps_0}^u(v_1)$ and $\lambda(v_1) \geq \eta_0$, then $\lambda(v_2) \geq \eta$.  By Corollary \ref{cor:weak-expansivity}, there exists $T_0>0$ such that if $t \geq T_0$ and $f_{t}(w) \in W_R^u(f_{t}v')$, then $w\in W_{\eps_0}^u(v')$. Thus, if $\lambda(v')\geq \eta_0$, then $\lambda(w) \geq \eta$. Thus, item \ref{wfw} of the theorem holds for any $t\geq T_0$.

We turn our attention to item \ref{nearSing}.
 By Proposition \ref{prop:proximity}, there are $\eta',T_1>0$ such that
\begin{equation}\label{eqn:T1}
\text{if }\lambda^u(f_s v)\leq \eta' \text{ for all } |s|\leq T_1, \text{ then } \dK(v,\Sing) < \delta.
\end{equation}
Given $v\in \Sing$, we have $\Pi^s(v) = v' \in W_R^s(v)$, and $\lambda(f_s v) = 0$ for all $s$.

Again by uniform continuity of $\lambda^u$, we can take $\eps_1$ sufficiently small such that if $v_2 \in W_{\eps_1}^s(v_1)$ and $\lambda^u(v_1) <\eta'/3$, then $\lambda^u(v_2) \geq \eta'/2$. By Corollary \ref{cor:weak-expansivity}, there is $T_2>0$ such that for all $t\geq T_2$ and $v \in W_R^s(v')$ we have $f_t v \in W_{\epsilon_1}(f_t v')$. Thus if  $\lambda(f_tv) < \eta'/3$, then $\lambda(f_tv') < \eta'/2$. Thus, since $\lambda^u(v)=0$, the inclusion $v \in W_R^s(v')$ implies that $\lambda^u(f_t v') < \eta'/2$ for all $t\geq T_2$.

 

Using uniform continuity of $\lambda^u$ one more time, we can take $\eps_2$ sufficiently small such that if $v_2 \in W_{\eps_2}^u(v_1)$ and $\lambda^u(v_2) < \eta'/2$, then $\lambda^u(v_1) < \eta'$.  
By Corollary \ref{cor:weak-expansivity}, there is $T_3>0$ such that if $t\geq T_3$ and $f_t v_2 \in W_R^u(f_t v_1)$, then $v_2 \in W_{\eps_2}^u(v_1)$.  Thus, if $t\geq T_3$, $f_tv_2 \in W_R^u(f_t v_1)$ and $\lambda^u(v_2) < \eta'/2$, then $\lambda^u(v_1) < \eta'$. Thus, for $s\in [T_2, t - T_3]$, the inclusion $f_{t-s}(f_s w) = f_t w \in W_R^u(f_t v') = W_R^u(f_{t-s}(f_s v'))$ implies that $\lambda^u(f_s w) < \eta'$.

Applying \eqref{eqn:T1} gives $\dK(f_sw,\Sing) < \delta$ for all $s\in [T_2 + T_1, t-T_3 - T_1]$. Thus, taking $L = \max(T_0, T_1 + \max(T_2,T_3))$, assertions \ref{wfw} and \ref{nearSing} follow for $t\geq 2L$.

For item \ref{component} of the theorem, we observe that $v$ and $w$ can be connected by a path $u(r)$ that follows first $W_R^s(v)$, then $f_{-t}(W_R^u(f_tv'))$ (see Figure \ref{fig:regularizing}), and that the arguments giving $\dK(f_sw,\Sing)<\delta$ also give $\dK(f_su(r),\Sing)<\delta$ for every $s\in [L,t-L]$ and every $r$.  We conclude that $f_sv$ and $f_sw$ lie in the same connected component of $B(\Sing,\delta)$ for every such $s$.  
\end{proof}

\subsection{Multiplicity of $\Pi_t$ and pressure estimates} \label{sec:pressure estimates}
We now bound the number of $v\in \Sing$ whose images under $\Pi_t$ are close in the $d_t$ metric, recalling that 
\[
d_t(v,w) = \max_{s\in [0,t]} \dK(f_s v, f_s w) = \max_{s\in [0,t+1]} d(\gamma_v(s),\gamma_w(s)) .
\]
\begin{prop}\label{prop:multiplicity}
For every $\eps>0$, there exists $C>0$ 
such that if $E_t \subset \Sing$ is a $(t,2\eps)$-separated set for some $t>0$, then for every $w\in T^1M$, we have $\#\{v\in E_t \mid d_t(w,\Pi_t v) < \eps\} \leq C$.
\end{prop}
\begin{proof}
Let $\widetilde M$ be the universal cover of $M$ and $B\subset \widetilde M$ a fundamental domain.  Define $\widetilde \Pi^{s,u}$ and $\widetilde \Pi_t$ in the obvious way, by lifting $\Pi^{s,u}$ and $\Pi_t$ to the universal cover. We write $\widetilde d$ for the lift of the Riemannian metric $d$ to $\widetilde M$, and $\widetilde d_t$ for the lift of the metric $d_t$. Every $\widetilde v\in T^1\widetilde M$ has
\[
\widetilde d(\pi\widetilde v, \pi\widetilde\Pi_t \widetilde v) \leq
\widetilde d(\pi\widetilde v, \pi \widetilde\Pi^s\widetilde v) + \widetilde d(\pi \widetilde\Pi^s \widetilde v, \pi\widetilde\Pi_t\widetilde v) 
\leq d^s(v, \Pi^s v) + d^u(\Pi^s v, \Pi_t v) \leq 2R,
\]
recalling that $\pi \widetilde v\in \widetilde M$ is the footprint of $\widetilde v$.
Given $v\in T^1M$, let $\widetilde v_B\in T^1 \widetilde M$ be the lift of $v$ with $\pi\widetilde v_B\in B$; then we have $\pi\widetilde\Pi_t\widetilde v_B \in A_{2R} := \bigcup_{x\in B} B_{\widetilde d}(x,2R)$.

 Fix $\eps>0$ and let $\Gamma = \Gamma_{(2R+ \eps), B} :=\{g\in \pi_1(M) \mid gB \cap A_{2R+\eps} \neq\emptyset\}$.  Note that $\#\Gamma<\infty$ because $\bar{B}$ is compact.
For $t>0$, let $E_t \subset \Sing$ be any $(t,2\eps)$-separated set, and fix an arbitrary $w\in T^1 M$.  We define 
\[
E_t^{w,\eps} := \{v\in E_t \mid d_t(w,\Pi_t v)< \eps \}.
\]
Let $X\subset B$ and $Y\subset A_{2R+\eps}$ 
be finite $\eps$-dense sets.  We will show that $\#E_t^{w,\eps}\leq (\#\Gamma)(\#X)(\#Y) =: C$.


Since $d_t(w,\Pi_t v) < \eps$, there exists a lift $\widetilde w$ of $w$ with $\widetilde d_t(\widetilde w, \widetilde \Pi_t \widetilde v_B)<\eps$. It follows that $\pi\widetilde w \in A_{2R+\eps}$, and thus $\pi\widetilde w \in gB$ for some $g\in \Gamma$. Thus, $E_t = \bigcup_{g \in \Gamma} E^g_t$, where
\[
E_t^g = \{v \in E_t^{w, \eps} \mid \widetilde d_t(\widetilde w, \widetilde \Pi_t \widetilde v)< \eps \text{ where $\widetilde w$ is the lift of $w$ to } gB\}. 
\]
For a fixed $g\in \Gamma$  and $v\in E_t^g$, we approximate $\widetilde v_B$ and $f_t \widetilde v_B$ using the sets $X$ and $Y$. Recall that $\pi\widetilde v_B \in B$ by definition, and we will show that  the location of $f_t \widetilde v_B$ in $T^1 \widetilde M$ is controlled by using $f_t\widetilde w$ as a reference point.  Given $v\in E_t^g$, let $x=x(v)\in X$ be such that $\widetilde d(x,\pi\widetilde v_B)<\eps$. Let $h$ be the unique element of $\pi_1(M)$ so that $\pi f_t \widetilde w \in hB$. Then $\widetilde d(hB, \pi f_t(\widetilde \Pi_t \widetilde v_B))<\epsilon$, and thus $\pi f_t \widetilde v_B \in h(A_{2R+\eps})$. Thus, there is some $y=y(v)\in Y$ such that $\widetilde d(\pi f_t\widetilde v_B, h(y))< \eps$.

Now we show that the map $x\times y \colon E_t^g \to X\times Y$ is injective.  Given $v_1,v_2\in E_t^g$ and $s\in [0,t]$, let $\rho(s) = \widetilde d(\gamma_{\widetilde v_1}(s),\gamma_{\widetilde v_2}(s))$ and note that $\rho$ is convex. In particular, it takes its maximum value at an endpoint.  If $x(v_1)=x(v_2)$, then $\rho(0) < 2\eps$, and if $y(v_1)=y(v_2)$, then $\rho(t)<2 \eps$.  Thus if $v_1,v_2$ have the same image under $x\times y$, we get $\widetilde d_t((\widetilde v_1)_B,(\widetilde v_2)_B) < 2\eps$, and thus $d_t(v_1, v_2)< 2 \eps$. Since $E_t^g$ is $(t,2\eps)$-separated, this gives $v_1=v_2$. Injectivity shows that $\#E_t^g \leq (\#X)(\#Y)$ for every $g\in \Gamma$, which proves Proposition  \ref{prop:multiplicity} with $C = \#\Gamma\#X\#Y$.
\end{proof}

We now obtain a lower bound on partition sums for the singular set, beginning with the following general lemma.
\begin{lem} \label{gen:partitionsum}
Let $(X, \FFF)$ be a continuous flow on a compact metric space,
let $\ph: X \to \RR$ be continuous and let $\eps>0$. Then for all $t>0$, 
\begin{equation}\label{eqn:fatpartsum}
\sup
\left\{ \sum_{x\in E} e^{\sup_{y\in B_t(x,\eps)} \Phi(y, t)} \mid E\subset X \text{ is $(t,\epsilon)$-separated} \right\} \geq e^{t P(X, 2 \eps, \ph)}.
\end{equation}
\end{lem}
\begin{proof}
The argument is given in the proof of \cite[Lemmas 4.1 and 4.2]{CT4}. 
\end{proof}
The expression on the left hand side of \eqref{eqn:fatpartsum} is a `fattened up' version of a partition sum, and is easily seen to be within a multiplicative constant of $\Lambda(X, \eps, t, \ph)$ whenever $\ph$ has the Bowen property on $X$ at scale $\epsilon$. In particular, under the hypotheses of Theorem \ref{thm:pressure-gap}, the potential $\ph$ is locally constant on a neighborhood of $\Sing$, and so for sufficiently small $\eps$, the left hand side of \eqref{eqn:fatpartsum} is equal to $\Lambda(\Sing,\eps,t, \ph)$.
Moreover, the geodesic flow is entropy-expansive \cite[Proposition 3.3]{knieper98}, so for sufficiently small $\eps>0$, we have $P(\Sing,\ph) = P(\Sing,\ph, 2\eps)$ and 
Lemma \ref{gen:partitionsum} gives
\begin{equation}\label{eqn:Lambda-n}
\Lambda(\Sing,\eps,t,\ph) \geq e^{tP(\Sing, \ph)}.
\end{equation}

Fix $\eta_0>0$ as at the beginning of \S\ref{s:singreg}.  Fix $\eta\in (0, \eta_0)$ and choose $\delta>0$ small enough such that
$ \Lambda(\Sing,2\delta,t,\ph) \geq e^{tP(\Sing, \ph)}$ for every $t$, $\ph$ is locally constant on $B(\Sing,2\delta)$, and $\lambda(v) < \eta$ for all $v\in B(\Sing,2\delta)$. 

Let $U_1,\dots, U_k$ be the components of $B(\Sing,2\delta)$, and let $\Phi_i \in \RR$ be the constant value that $\ph$ takes on $U_i$.
By \eqref{eqn:Lambda-n}, for every $t$,  there exists a $(t,2\delta)$-separated set $E_t \subset \Sing$ such that 
\begin{equation}\label{eqn:En}
\sum_{i=1}^k e^{t\Phi_i} \#(E_t \cap U_i) \geq e^{tP(\Sing,\ph)}.
\end{equation}
We consider the image of $E_t$ under the map $\Pi_t$.  
Let $L=L(\eta,\delta)$ be as in Theorem \ref{thm:sing-to-reg}.  Write $E_t' = \Pi_t(E_t)$; then
\begin{equation}\label{eqn:ends-in-eta}
w,f_t(w)\in \Reg(\eta) \text{ for every } w\in E_t' \text{ with } t> 2L.
\end{equation}
Given $v\in E_t \cap U_i$, the third item of Theorem \ref{thm:sing-to-reg} shows that for any $u\in B_t(\Pi_tv,\delta)$, we have
\begin{equation}\label{eqn:En'}
\int_0^t \ph(f_s u)\,ds \geq (t-2L)\Phi_i - 2L\|\ph\| \geq t\Phi_i - 4L \|\ph\|.
\end{equation}
By Proposition \ref{prop:multiplicity}, for each $w\in E_t'$ there are at most $C=C(\delta)$ elements $w'\in E_t'$ with $d_t(w,w') < \delta$; this leads to the following lemma.

\begin{lem}\label{lem:En''}
Let $E_t \subset \Sing$ be a $(t,2\delta)$-separated set satisfying \eqref{eqn:En}, and let $E_t' = \Pi_t(E_t)$.  There exists a $(t,\delta)$-separated set $E_t'' \subset E_t'$ such that, setting $\beta = C^{-1} e^{-4L\|\ph\|}$, we have
\begin{equation}\label{eqn:En''}
\sum_{w\in E_t''} e^{\inf_{u\in B_t(w,\delta)} \int_0^t \ph(f_s u)\,ds} \geq \beta e^{tP(\Sing,\ph)}.
\end{equation}
\end{lem}
\begin{proof}
Given $1\leq i\leq k$, let $E_{t,i}' = \Pi_n(E_t \cap U_i)$ and take a maximal $(t,\delta)$-separated subset $E_{t,i}'' \subset E_{t,i}'$.  Now $\# E_t' \cap B_t(w,\delta) \leq C$ for all $w\in E_{t,i}''$ and $E_{t,i}''$ is $(t,\delta)$-spanning for $E_{t,i}'$, so $\#E_{t,i}'' \geq C^{-1} \#E_{t,i}'$.  Sum over $i$ and use \eqref{eqn:En} and \eqref{eqn:En'} to get \eqref{eqn:En''} for $E_t'' = \bigcup_{i=1}^k E_{t,i}''$.
\end{proof}

\subsection{Proof of Theorem \ref{thm:pressure-gap}: creating topological pressure} \label{s.pressureproduction}
We now carry out the scheme for producing topological pressure outlined at the start of \S\ref{s.entropygap}, which completes our proof of Theorem \ref{thm:pressure-gap}. We work with the scales $\eta,\delta>0$ fixed in the previous subsection. In particular, we recall that
\begin{equation}\label{eqn:delta-eta}
d(v,\Sing) < 2\delta \Rightarrow \lambda(v) < \eta,
\end{equation}
and $L=L(\delta, \eta)$ is the constant from Theorem \ref{thm:sing-to-reg}.
By Theorem \ref{specgeoflow}, we can choose $T=T(\delta, \eta)$ large enough so that the following specification property holds on $\mathcal{C}(\eta) = \{(v,t) \in T^1 M \times (0,\infty) \mid v,f_tv\in \Reg(\eta)\}$: for every  $\{(v_j,t_j)\}_{j=1}^k \subset \mathcal{C}(\eta)$ and every $T_1,T_2,\dots, T_k$ with the property that $T_{j+1} -T_j \geq  t_j + T$, there is $w\in T^1 M$ such that
\begin{equation}\label{eqn:specd}
f_{T_j}(w) \in B_{t_j}(v_j,\delta/3) \text{ for all } 1\leq j\leq k.
\end{equation}
Let $\alpha>0$ be small and $N\in \mathbb{N}$. We assume that $\alpha N\in \mathbb{N}$ for notational convenience so that we do not need to consider $\lfloor \alpha N \rfloor$ throughout our arguments.  
Let $\tau=2L+T$ and consider the set
\[
\mathcal{A} = \{\tau, 2\tau, 3\tau, \dots, (N-1)\tau\} \subset [0,N\tau].
\]
We select  $\alpha N-1$ of the $N-1$ elements in $\mathcal{A}$ as `gluing times' for our construction. We write $\mathbb{J}_N^\alpha = \{J\subset \mathcal{A} \mid \#J = \alpha N-1\}$ for the collection of all such possibilities, and note that $\#\mathbb{J}_N^\alpha = \binom{N-1}{\alpha N-1}$.

We now obtain estimates for a fixed choice of gluing times.   Given $J\in \mathbb{J}_n^\alpha$, we write $J = \{N_1\tau,\dots, N_{\alpha N - 1}\tau\}$. We set $N_0=0$ and $N_{\alpha N} = N$. Let $n_1 \tau,\dots, n_{\alpha N} \tau$ be the gaps between successive elements of $J$. Then $n_i = N_i- N_{i-1} \in \mathbb N$ for all $i \in \{1, \ldots, \alpha N\}$, and $N_j = \sum_{i=1}^j n_i$. 
 
 Fix $N,\alpha>0$ such that $\alpha N\in \mathbb{N}$.  Fix $J\in J_N^\alpha$ and let $n_1,\dots, n_{\alpha n}$ be as above. We consider the $(t, \delta)$ separated sets $E_t''$, recalling that $(x, t) \in C(\eta)$ for all $x\in E_{t}''$. Thus, writing $t_j = n_j\tau - T$, we can apply the specification property \eqref{eqn:specd} to elements of $E_{t_j}''$ with $T_j = N_{j-1}\tau$  so that
$T_{j+1}-T_j = (N_{j}-N_{j-1})\tau = n_j \tau=t_j+T$. We conclude that for every choice of $\mathbf{v} = (v_1,\dots, v_{\alpha N}) \in \prod_{j=1}^{\alpha N} E_{n_j\tau-T}''$, there exists $G(\mathbf{v}) = w\in T^1 M$ such that
\begin{equation}\label{eqn:specdd}
f_{N_{j-1}\tau}(w) \in B_{n_j\tau-T}(v_j, \delta/3) \text{ for all } 1\leq j\leq \alpha N.
\end{equation}
Let $\mathcal{X}_J$ be the set of all $w$ produced by the above procedure. That is, $\mathcal{X}_J$ is the image of the map $G: \prod_{j=1}^{\alpha N} E''_{n_j \tau - T} \to T^1M$.
\begin{lem}\label{lem:XJk-sep}
The set 
$\mathcal{X}_J$ is $(N\tau,\delta/3)$-separated.  
\end{lem}
\begin{proof}
Given $w^1\neq w^2 \in \mathcal{X}_J^\mathbf{k}$, there are $\mathbf{v}^1\neq\mathbf{v}^2 \in 
\prod_{j=1}^{\alpha N} E_{n_j\tau-T}''
$ such that $G(\mathbf{v}^i) = w^i$.  Each of the sets $E_{n_j\tau-T}''$ is $(n_j\tau-T,\delta)$-separated, so there exists $j\in \{1,\dots, \alpha N\}$ and $t\in [0,n_j\tau-T]$ with $d(f_t v^1_j, f_t v^2_j) \geq \delta$.  It follows from \eqref{eqn:specd}, writing $s=N_{j-1} \tau + t$, that
\[
d(f_sw^1,f_sw^2) \geq d(f_t v^1_j, f_t v^2_j) - d(f_t v^1_j, f_s w^1_j) - d(f_t v^2_j, f_s w^2)  > \delta/3.  \qedhere
\]
\end{proof}
We have the following control on when the orbit of $w\in \mathcal{X}_J$ is close to $\Sing$.

\begin{lem}
Every $w\in \mathcal{X}_J$ satisfies
\begin{equation}\label{eqn:JAJ}
\begin{aligned}
d(f_t w, \Sing) > 5\delta/3
&\text{ for all } t\in J,\\
d(f_t w,\Sing) < 4\delta/3
&\text{ for all } t\in \mathcal{A}\setminus J.
\end{aligned}
\end{equation}
\end{lem}
\begin{proof}
Let $w\in \mathcal{X}_J$. If $t=n\tau \in J$, then \eqref{eqn:specdd} gives $d(f_tw,v_j) < \delta/3$.  Since $v_j \in \Reg(\eta)$, it follows from \eqref{eqn:delta-eta} that $d(v_j,\Sing) \geq 2\delta$ and we conclude that $d(f_tw,\Sing) > 5\delta/3$.   

If $t=n\tau \in \mathcal{A} \setminus J$, there exists $j$ such that $N_{j-1} < n < N_j$. Writing $s=n\tau - N_{j-1} \tau$, we  have $d(f_t w, f_s v_j) < \delta/3$  by \eqref{eqn:specdd}. We have $s\in[L, t_j-L]$ since
\[
L \leq \tau \leq  s \leq (N_j-1) \tau - N_{j-1}\tau =n_j\tau -\tau < t_j-L.
\]
By (2) of Theorem \ref{thm:sing-to-reg}, $d(f_s v_j,\Sing) < \delta$, and so $d(f_tw, \Sing)< 4\delta /3$. 
\end{proof}
We use these lemmas and the estimate \eqref{eqn:En''} to obtain the following partition sum estimate for $\mathcal{X}_J$.

\begin{prop}\label{prop:XJ}
Let $Q =   T \|\ph\| - \log \beta$, where $\beta$ is the constant appearing in Lemma \ref{lem:En''}. Then for every $N,\alpha>0$ with $\alpha N \in \mathbb N$ and every $J\in \mathbb{J}_N^\alpha$, the $(N\tau,\delta/3)$-separated set $\mathcal{X}_J$ satisfies
\begin{equation}\label{eqn:sum-XJ}
\sum_{w\in \mathcal{X}_J} e^{\int_0^{N\tau} \ph(f_tw)\,dt}
\geq e^{-\alpha N Q} e^{N\tau P(\Sing,\ph)}.
\end{equation}
\end{prop}
\begin{proof}
The set $\mathcal{X}_J$ is $(N\tau,\delta/3)$-separated by Lemma \ref{lem:XJk-sep}.  We estimate $\int_0^{N\tau} \ph(f_tw)\,dt$ by breaking the integral over $[0,N\tau]$ into pieces corresponding to the intervals $[N_{j-1}\tau, N_j\tau - T]$, on which the orbit of $u$ is within $\delta/3$ of the orbit of $v_j$; the remaining pieces have a total length of $\alpha N T$, and so for $w=G(\mathbf{v})$, we obtain
\[
\int_0^{N\tau} \ph(f_tw)\,dt \geq -\alpha N T \|\ph\| + 
\sum_{j=1}^{\alpha N} \inf_{w\in B_{t_j}(v_j,2\delta/3)} \int_0^{t_j} \ph(f_t v_j)\,dt,
\]
where $t_j = n_j\tau - T$.
Using this together with the estimate \eqref{eqn:En''} gives 
\begin{multline*}
\sum_{w\in \mathcal{X}_J} e^{\int_0^{N\tau} \ph(f_t w)\,dt}
\geq
\sum_{\mathbf{v}} 
e^{-\alpha NT\|\ph\|}
\prod_{j=1}^{\alpha N} 
e^{\inf_{w\in B_{t_j}(v_j,\delta)} \int_0^{t_j} \ph(f_t v_j)\,dt} \\
 \geq e^{-\alpha NT\|\ph\|}
\prod_{j=1}^{\alpha N}
\beta e^{t_j P(\Sing,\ph)}
\geq e^{\alpha N(\log \beta - T\|\ph\|)} e^{N\tau P(\Sing,\ph)},
\end{multline*}
where the sum indexed by $\mathbf{v}$ is over $\prod_{j=1}^{\alpha N} E''_{n_j\tau - T}$.
\end{proof}

We now show as a consequence of Proposition \ref{prop:XJ} that different choices of gluing time data from $\mathbb{J}_N^\alpha$  lead to the construction of separated points. 

\begin{lem}\label{lem:JJ'}
If $J\neq J'\in \mathbb{J}_N^\alpha$ and $v\in \mathcal{X}_J$, $w\in \mathcal{X}_{J'}$, then there is $t\in [0,N\tau ]$ such that $d(f_tv,f_tw) \geq \delta/3$.
\end{lem}
\begin{proof}
Since $J\neq J'$, there exists $t\in \mathcal{A}$ with $t\in J$ and $t\notin J'$.  By \eqref{eqn:JAJ}, we have 
$d(f_t v, \Sing) > 5\delta/3$
and $d(f_tw,\Sing) < 4\delta/3$, so $d(f_t v, f_tw) \geq \delta/3$.
\end{proof}

It follows immediately that the set $F_N := \bigcup_{J\in \mathbb{J}_N^\alpha} \mathcal{X}_J$ is $(N\tau ,\delta/3)$-separated.  Moreover, \eqref{eqn:sum-XJ} gives
\begin{equation}\label{eqn:sum-FN}
\sum_{w\in F_N} e^{\int_0^{N\tau } \ph(f_t w)\,dt}
\geq
\binom{N-1}{\alpha N - 1} e^{-\alpha NQ} e^{N\tau P(\Sing,\ph)}.
\end{equation}
Using the fact that $\frac{N - k}{\alpha N-k} \geq \frac 1\alpha$ for all $1\leq k < \alpha N$, we obtain the estimate
$\binom{N-1}{\alpha N - 1} = \prod_{k=1}^{\alpha N-1} \frac{N-k}{\alpha N-k} \geq (\frac 1\alpha)^{\alpha N - 1} \geq \alpha e^{(-\alpha \log \alpha) N}$, and so
\eqref{eqn:sum-FN} gives
\[
\Lambda(T^1 M,\ph,\delta/3,N\tau )
\geq 
\alpha e^{(-\alpha\log\alpha)N} 
e^{-\alpha NQ} e^{N\tau P(\Sing,\ph)}.
\]
Taking a logarithm, dividing by $N\tau $, and sending $N\to\infty$ gives
\begin{equation}\label{eqn:P-Psing}
P(\ph) \geq -\tfrac \alpha{\tau } \log\alpha -\tfrac{\alpha Q}{\tau } + P(\Sing,\ph).
\end{equation}
Recall that $\alpha$ is chosen independently from the constants $\tau$ and $Q$, so we can choose $\alpha$ with $\alpha<e^{-Q}$ to ensure that the right-hand side of \eqref{eqn:P-Psing} is greater than $P(\Sing,\ph)$. This completes the proof of Theorem \ref{thm:pressure-gap}.


\section{Proof of Theorems \ref{t.multiples}, \ref{t.geometric} and \ref{t.highergeometric}}\label{s.mainresults}

Now we apply Theorem~\ref{t.geodgeneral} to obtain 
Theorems~\ref{t.multiples},  \ref{t.geometric} and \ref{t.highergeometric}.
\begin{proof}[Proof of Theorem \ref{t.multiples}]
By Corollaries \ref{cor:h-b} and \ref{cor:bow}, if $\ph$ is H\"older continuous or $\ph= q \vg$, then it has the Bowen property on $\GGG(\eta)$ for all $\eta>0$. Then Theorem ~\ref{t.geodgeneral} applies, yielding the statement of Theorem \ref{t.multiples}.
\end{proof}

\begin{proof}[Proof of Theorem~\ref{t.geometric}]  
For surfaces, we have $h_{\mathrm{top}}(\sing)=0$ and $\vg(v)=0$ for all $v\in \sing$, so $P(\sing, q\vg)=0$ for all $q \in \RR$.  
We show that $P(q\vg)>0$ for all $q\in (-\infty, 1)$.  
Let $\psi^u$ be as in \eqref{eqn:psi-u}. Then
$\psi^u \geq \lambda^u\geq \lambda$. By Corollary \ref{c.singular}, for any invariant measure $\mu$ we have $\mu(\{v : \lambda(v)>0\})=\mu(\Reg)$, so it follows that $\int\psi^u\,d\mu>0$ whenever $\mu(\Reg)>0$.  In particular, the Liouville measure $\mu_L$ has 
$0 > \int \psi^u\,d\mu_L = \int\vg\,d\mu_L = -\int \lambda^+(\mu_L) = -h_{\mu_L}(\FFF)$ by Lemma \ref{c.equiv_pressure} and the Pesin entropy formula, so for every $q\in (-\infty,1)$ we have
\[
P(q\psi^u) \geq h_{\mu_L}(\mathcal{F}) + \int q \psi^u\, d\mu_L
> h_{\mu_L} + \int \psi^u\,d\mu_L = 0 = P(\Sing,q\psi^u).
\]
Then Theorem \ref{t.multiples} gives uniqueness and the desired properties for $\mu_q$.


Since the flow is entropy expansive, the entropy map is upper semi-continuous, and so by work of Walters \cite{pW92}, the function $q\mapsto P(q\vg)$ is $C^1$ on any interval where each $q\vg$ has a unique equilibrium state.  In particular, it is $C^1$ on $(-\infty,1)$.
\end{proof}

To show that the equilibrium states $\mu$ obtained in Theorem~\ref{t.geometric} are Bernoulli, we apply a result by Ledrappier, Lima, and Sarig \cite{LLS} showing that if $M$ is any 2-dimensional manifold, $\ph\colon T^1M\to \RR$ is H\"older or a scalar multiple of $\vg$, and $\mu$ is a positive entropy ergodic equilibrium measure for the geodesic flow on $T^1M$, then $\mu$ is Bernoulli. Although their result is stated for positive entropy measures, this assumption is only used to guarantee that the measure has a positive Lyapunov exponent, see \cite[Theorem 1.3]{LS}. Since our measure $\mu$ is hyperbolic by Corollary \ref{c.hyperbolic}, it follows that 
 \cite{LLS} applies.


We now prove Theorem \ref{t.highergeometric}, and investigate the pressure gap for the potentials $q\ph^u$ for higher dimensional manifolds.

\begin{proof}[Proof of Theorem \ref{t.highergeometric}]
For the proof of Theorem \ref{t.highergeometric}, first observe that given any continuous $\ph$, the set $\{q\in \RR : P(\Sing,q\ph) < P(q\ph)\}$ is open since both sides of the inequality vary continuously in $q$.  Then Theorem \ref{t.highergeometric} is a direct consequence of Theorems \ref{t.multiples} and \ref{thm:pressure-gap}.
\end{proof}

As
remarked after the statement of Theorem \ref{t.highergeometric}, if $M$ is a rank 1 manifold 
such that $\htop(\Sing)=0$, then we have $P(\Sing,q\vg)\leq 0$ for all $q\geq 0$ since $\vg\leq 0$. Thus, the argument in the proof of Theorem \ref{t.geometric} gives the pressure gap on $[0,1)$.  Since the gap is an open condition, it holds on $(-q_0,1)$ for some $q_0>0$. Finally, we show that the pressure gap holds under a bounded range condition.

\begin{lem}\label{cor:sup-inf}
Let $M$ be a closed rank 1 manifold and $\varphi\colon T^1M\to \RR$ be continuous. 
If 
\begin{equation} \label{bddrange}
\sup_{v\in \Sing}\varphi(v)-\inf_{v\in T^1M}\varphi(v) < \htop (\FFF) - \htop(\Sing),
\end{equation}
then $P(\Sing, \ph)< P(\ph)$.
\end{lem}
If $\dim(M)=2$, then $\htop(\Sing)=0$, so the right hand side of \eqref{bddrange} is just $\htop(\FFF)$. If $\ph = q \vg$ or is H\"older,  the bounded range hypotheses \eqref{bddrange} gives another criterion which ensures that Theorem \ref{t.multiples} applies. In particular, it follows that the value of $q_0$ in Theorem \ref{t.highergeometric} can be taken with $q_0 \geq (\htop(\FFF)- \htop(\Sing))/2 \| \vg \|$.

\begin{proof}[Proof of Lemma \ref{cor:sup-inf}]
First rewrite \eqref{bddrange} as
\begin{equation}\label{eqn:bddrange2}
\htop(\Sing) + \sup_{v\in\Sing} \varphi(v)
< \htop(\mathcal{F}) + \inf_{v\in T^1M} \varphi(v).
\end{equation}
The variational principle for $\FFF|_\Sing$ gives
\[
P(\Sing,\ph) = \sup_{\nu\in \MMM(\FFF|_\Sing)} \bigg\{ h_\nu(\FFF) + \int\ph\,d\nu \bigg\} \leq \htop(\Sing) + \sup_{v\in \Sing}\ph(v).
\]
Now let $m$ be the measure of maximal entropy for $\FFF$. Then
\[
\htop(\mathcal{F}) + \inf \varphi  = h_m(\FFF) + \inf \varphi \leq h_m(\FFF) + \int \varphi \, dm \leq P(\varphi).
\]
Together with \eqref{eqn:bddrange2}, these give $P(\Sing,\varphi)<P(\varphi)$.
\end{proof}
We note that Gromov's example \cite[\S6]{knieper98} can be modified to make $\htop(\FFF)- \htop(\Sing)$ arbitrarily small, so there is no hope that \eqref{bddrange} yields a universal lower bound on $q_0$. We do not know if a small entropy gap restricts the value of $q_0$. Understanding this issue for the Gromov example would give insight into the general case.

\section{Examples} \label{s.examples}

In this section, we investigate examples of the geodesic flow on rank 1 manifolds with $\Sing\neq \emptyset$.  First,  we give a  class of 
manifolds for which we establish the existence of unique equilibrium states for a $C^0$-generic set of potential functions. This class includes any rank 1 surface equipped with an analytic metric. The second example is a modification of an example due to Heintze in which we establish the uniqueness of an equilibrium state for $q\vg$ for all $q\in \mathbb{R}$.

\subsection{Examples where pressure gap holds generically}
\label{sec:genericity}

We show that when the singular set is a finite union of disjoint compact sets, on each of which the geodesic flow is uniquely ergodic, then the set of H\"older potentials for which there is a pressure gap is $C^0$-generic.  

\begin{prop}\label{prop:C0-dense}
Suppose that $\Sing$ is a union of disjoint compact sets $Z_1,\dots, Z_k$, on each of which the geodesic flow is uniquely ergodic.  Let $H_0 \subset C(T^1 M)$ be the set of all $\psi$ that are constant on a neighborhood of each $Z_i$, and let $H \subset C(T^1M)$ be the set of all $\ph$ that are cohomologous to some $\psi\in H_0$.
Then $H$ is $C^0$-dense in $C(T^1M)$.
\end{prop}
\begin{proof}
Given $\ph\in C(T^1M)$ and $T>0$, consider the ergodic average function  $\ph_T(v) := \frac 1T \int_0^T \ph(f_s v)\,ds$.  Then $\ph$ and $\ph_T$ are cohomologous; indeed, writing
$\zeta(v) := \frac 1T \int_0^T (T-s) \ph(f_s v)\,dv$,
an elementary computation shows that the derivative of $\zeta$ in the flow direction is $\ph_T - \ph$.

Let $\mu_i$ be the unique invariant measure on $Z_i$, and write $c_i = \int\ph\,d\mu_i$.  Given $\eps>0$, there are $T_1,\dots, T_k$ such that for every $T\geq T_i$, we have $|\ph_T(v) - c_i| < \eps$ for every $v\in Z_i$.  Let $T=\max_i T_i$, and let $\psi = \ph_T$.

There is a function $\widetilde\psi\in H_0$ taking the value $c_i$ on a neighborhood of $Z_i$ and having $\|\widetilde \psi - \psi\|_{C^0} < \eps$.  Let $\widetilde \ph = \widetilde \psi + \ph - \psi$, then $\widetilde\ph$ is cohomologous to $\widetilde \psi$, so $\widetilde\ph\in H$, and we have $\|\widetilde\ph - \ph\|_{C^0} = \|\widetilde\psi - \psi\|_{C^0} < \eps$.
\end{proof}

Under the hypotheses of Proposition \ref{prop:C0-dense}, every $\ph\in H$ is cohomologous  to some $\psi\in H_0$ to which Theorem \ref{thm:pressure-gap} applies, giving
$P(\Sing,\ph) = P(\Sing,\psi) < P(\psi) = P(\ph)$.
Since $P(\ph)$ and $P(\Sing,\ph)$ vary continuously as $\ph$ varies (w.r.t.\ $C^0$), the set of potentials with the pressure gap is $C^0$-open, and since it contains $H$, it is $C^0$-dense.  
Writing $C^h$ for the space of H\"older potentials on $T^1M$, observe that $C^h$ is $C^0$-dense in $C(T^1M)$, so the intersection $H \cap C^h$ is $C^0$-dense in $C^h$. 
This shows that the set of H\"older potentials for which the pressure gap holds, which is clearly $C^0$-open in the space of H\"older potentials, is $C^0$-dense.

\subsubsection*{Analytic metrics on surfaces} 

For a rank 1 surface with an analytic metric, it is a folklore result that  $\Sing$ is a finite union of periodic orbits; we sketch the idea of proof. 
If $\Sing$ is a not a finite union of periodic orbits, then the geodesic flow has a transversal whose intersection with $\Sing$ is not discrete.  In particular, there is a geodesic segment $T$ and vectors $v_n, v\in \Sing$, $v_n\neq v$, such that $T$ is orthogonal to $v$, transverse to $v_n$, and $v_n \to v$. 
Locally, the geodesics $\gamma_{v_n}$ intersect $\gamma_v$ at most once, and so for almost all $t$ close to $0$, a short geodesic segment orthogonal to $f_t(v)$ contains a sequence of points $x_n$ such that $x_n\to \pi g_t(v)$ and the Gaussian curvature of $M$ at $x_n$ is $0$.  Since curvature is real analytic, it must vanish along each of these geodesic segments and hence it is constant in a neighborhood of $\pi(v)$; since $M$ is connected, it must vanish everywhere, which is a contradiction. 
Thus, $\Sing$ is a finite union of periodic orbits and so Proposition \ref{prop:C0-dense} applies.

\subsubsection*{Non-generic pressure gap and questions}
The pressure gap does not hold generically if the manifold has a flat strip. One can take a potential $\ph$ supported near a periodic trajectory in the middle of the strip; if the support is small enough and the size of the potential large enough, one can guarantee that any regular trajectory has an ergodic average much smaller than the average along the periodic orbit, and conclude that $P(\Sing,\ph) > P(\Reg,\ph)$.  This inequality is stable under $C^0$-perturbations of $\ph$, so there is a $C^0$-open set of potentials $\psi$ for which $P(\Sing,\psi) = P(\psi)$.  It would be interesting to further investigate which classes of rank $1$ manifolds have the pressure gap for $C^0$-generic (or $C^\alpha$-generic) H\"older potentials.

\subsection{Heintze example}\label{sec:heintze}

The following example of a rank 1 manifold is attributed to Heintze and described by 
Ballmann, Brin, and Eberlein \cite[Example 4]{BBE}.  Consider an $n$-dimensional manifold $N$ of constant negative curvature and finite volume with only one cusp.  The cross section of the cusp is a flat $(n-1)$-dimensional torus $T$.
Next cut off the cusp and flatten the manifold near the cut so the resulting manifold is locally isometric to the direct product of $T$ and the unit interval.  Now consider another copy of the same manifold and identify the two copies along $T$ to obtain a manifold $M$ with nonpositive sectional curvature. 
The rank of any tangent vector to a geodesic in $T$ is $n$; however, any tangent vector to a geodesic transverse to $T$ has rank $1$.

\begin{figure}[htbp]
\includegraphics[width=.5\textwidth]{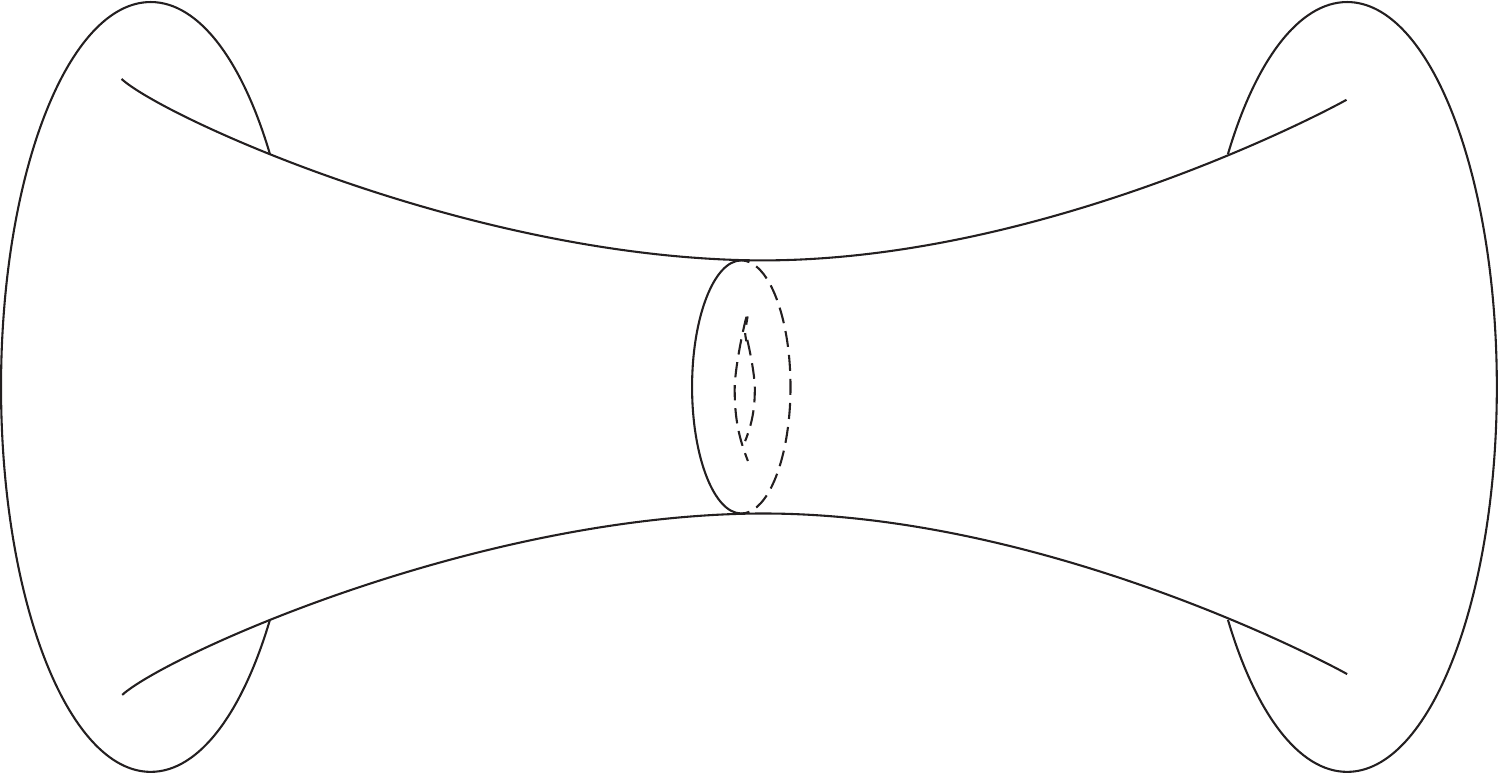}
\caption{Modified Heintze's example}\label{f.mheintze}
\end{figure}

We now modify this example.  We first assume for simplicity that $n=3$ and the cross-section is a 2-torus $T$.  
%
Start by perturbing the metric on a compact subset of $N$ so that there are two periodic orbits for which the corresponding invariant measures $\mu_1$ and $\mu_2$ have $\int\vg\,d\mu_1 < \int\vg\,d\mu_2$.  Then choose $T$ as above; by choosing $T$ to lie far enough out in the cusp we can guarantee that curvature is still constant in a neighborhood of $T$.  Instead of flattening the metric near the cut, replace the direct product metric on $T\times [0,1]$ with a warped product, see for instance \cite[p. 204]{Oneill} in which the tori are scaled by $\chi \cosh(d)$ where $d$ is the distance from the center 2-torus and $\chi>0$.
The fact that $\cosh(d)$ has a minimum at $d = 0$ means that the central 2-torus is totally geodesic; all of the vectors tangent to it are singular and have exponent zero in the direction tangent to the 2-torus. However now the sectional curvature in the direction orthogonal to the central 2-torus is negative and gives a nonzero Lyapunov exponent, see Figure \ref{f.mheintze}.

For this modified example, $\Sing$ consists of vectors tangent to the center 2-torus.  Every $v\in \Sing$ has zero Lyapunov exponent in the center direction, but the other Lyapunov exponents are nonzero since the sectional curvature corresponding to these directions is nonzero.  By the warped product construction every vector in $\Sing$ will have the same positive Lyapunov exponent $\lambda>0$, and since $h(\Sing)=0$, we know
$P(\Sing, q\ph^u) = -q \lambda$
for all $q\in \RR$.  
By varying the parameter $\chi$ in the construction, we can vary $\lambda$ so that $\int\vg\,d\mu_1 < -\lambda < \int\vg\,d\mu_2$.
Thus for all $q>0$ we have
\[
P(q\vg) \geq q\int\vg\,d\mu_2 > -q\lambda = P(\Sing,q\vg),
\]
and the corresponding inequality for $q<0$ follows by considering $\mu_1$.

Finally, since $\htop(\mathcal{F})=P(0)>0$ we see that $P(q\vg)> P(\Sing, q\vg)$ for all $q\in \RR$. 
Thus we have an example of a compact smooth 3-manifold $M$ that is rank $1$ of nonpositive curvature for which $\Sing \neq \emptyset$, and indeed $M$ does not support a metric of strictly negative curvature (since $\pi_1(M)$ contains $\ZZ^2$), but on the other hand $q\ph^u$ has a unique equilibrium state for every $q\in \RR$, which is fully supported. In particular, the Liouville measure is the unique equilibrium state for $\ph^u$.

\subsection*{Acknowledgments} We would like to thank the anonymous referees for their helpful comments which have benefited this article. Much of this work was carried out in a SQuaRE program at the American Institute of Mathematics. We thank AIM for their support and hospitality.

\bibliographystyle{amsplain}
\bibliography{bcft-references}

\end{document}